\documentclass[11pt]{article}
\usepackage{fullpage}
\usepackage{amsmath,amssymb,amsthm}
\usepackage[english]{babel}
\usepackage[left=3cm,right=3cm,top=3cm,bottom=4cm]{geometry}

\usepackage{tikz,tikz-cd}
\usetikzlibrary{positioning}
\usepackage{braket}
\usepackage{mathtools,stmaryrd}

\SetSymbolFont{stmry}{bold}{U}{stmry}{m}{n}
\usepackage{lmodern,bm,bbm,mathrsfs}
\usepackage{manfnt}

\usepackage{enumitem} 

\usepackage[scr=esstix]{mathalfa}

\usepackage{hyperref}
\usepackage[disable]{todonotes}

\newtheorem*{theorem}{Theorem}

\newtheorem*{lemma}{Lemma}
\newtheorem*{proposition}{Proposition}

\theoremstyle{definition}
\newtheorem*{definition}{Definition}
\newtheorem*{remark}{Remark}

\renewcommand{\bar}[1]{\overline{#1}}
\newcommand{\cat}[1]{\mathscr{#1}}
\renewcommand{\hat}[1]{\widehat{#1}}
\newcommand{\lie}[1]{\mathfrak{#1}}
\renewcommand{\tilde}[1]{\widetilde{#1}}
\renewcommand{\vec}[1]{\bm{#1}}

\newcommand{\bC}{\mathbb{C}}
\newcommand{\bE}{\mathbb{E}}
\newcommand{\bF}{\mathbb{F}}

\newcommand{\bL}{\mathbb{L}}
\newcommand{\bM}{\mathbb{M}}

\newcommand{\bP}{\mathbb{P}}

\newcommand{\bR}{\mathbb{R}}

\newcommand{\bZ}{\mathbb{Z}}

\newcommand{\cE}{\mathcal{E}}
\newcommand{\cF}{\mathcal{F}}
\newcommand{\cI}{\mathcal{I}}
\newcommand{\cK}{\mathcal{K}}
\newcommand{\cL}{\mathcal{L}}

\newcommand{\cN}{\mathcal{N}}
\newcommand{\cO}{\mathcal{O}}
\newcommand{\cOb}{\mathcal{O}b}
\newcommand{\cQ}{\mathcal{Q}}
\newcommand{\cT}{\mathcal{T}}
\newcommand{\cV}{\mathcal{V}}
\newcommand{\cY}{\mathcal{Y}}
\newcommand{\cZ}{\mathcal{Z}}
\newcommand{\fc}{\mathfrak{c}}
\newcommand{\fn}{\mathfrak{n}}

\newcommand{\fA}{\mathfrak{A}}
\newcommand{\fB}{\mathfrak{B}}
\newcommand{\fC}{\mathfrak{C}}
\newcommand{\fE}{\mathfrak{E}}
\newcommand{\fF}{\mathfrak{F}}
\newcommand{\fM}{\mathfrak{M}}
\newcommand{\fN}{\mathfrak{N}}
\newcommand{\fQ}{\mathfrak{Q}}

\newcommand{\fX}{\mathfrak{X}}
\newcommand{\fY}{\mathfrak{Y}}
\newcommand{\se}{\mathsf{e}}
\newcommand{\sE}{\mathsf{E}}
\newcommand{\sN}{\mathsf{N}}
\newcommand{\sQ}{\mathsf{Q}}
\newcommand{\sT}{\mathsf{T}}
\newcommand{\sV}{\mathsf{V}}
\newcommand{\sW}{\mathsf{W}}
\newcommand{\sZ}{\mathsf{Z}}
\newcommand{\scE}{\mathscr{E}}
\newcommand{\scI}{\mathscr{I}}
\newcommand{\scL}{\mathscr{L}}

\newcommand{\scV}{\mathscr{V}}

\newcommand{\BS}{\mathrm{BS}}
\newcommand{\DT}{\mathrm{DT}}
\newcommand{\exc}{\mathrm{exc}}
\newcommand{\ind}{\mathsf{ind}}

\newcommand{\loc}{\mathrm{loc}}

\newcommand{\pl}{\mathrm{pl}}
\newcommand{\PT}{\mathrm{PT}}
\newcommand{\pt}{\mathrm{pt}}
\newcommand{\sst}{\mathrm{sst}}
\newcommand{\st}{\mathrm{st}}
\newcommand{\vir}{\mathrm{vir}}

\newcommand\bigbullet{\scalebox{1.5}{$\bullet$}}

\DeclareMathOperator{\Aut}{Aut}
\DeclareMathOperator{\End}{End}
\DeclareMathOperator{\ch}{ch}
\DeclareMathOperator{\cl}{cl}

\DeclareMathOperator{\coker}{coker}
\DeclareMathOperator{\cocone}{cocone}
\DeclareMathOperator{\cone}{cone}
\DeclareMathOperator{\ev}{ev}
\DeclareMathOperator{\Ext}{Ext}
\DeclareMathOperator{\cExt}{\mathcal{E}\!{\it xt}}
\DeclareMathOperator{\fl}{fl}

\DeclareMathOperator{\Frac}{Frac}
\DeclareMathOperator{\GL}{GL}
\DeclareMathOperator{\Hilb}{Hilb}
\DeclareMathOperator{\fHilb}{\mathfrak{Hilb}}
\DeclareMathOperator{\cHom}{\mathcal{H}\!{\it om}}
\DeclareMathOperator{\Hom}{Hom}
\DeclareMathOperator{\id}{id}
\DeclareMathOperator{\im}{im}
\DeclareMathOperator{\length}{length}

\DeclareMathOperator{\PGL}{PGL}

\DeclareMathOperator{\Quot}{Quot}
\DeclareMathOperator{\rank}{rank}

\DeclareMathOperator{\rk}{rk}
\DeclareMathOperator{\SO}{SO}
\DeclareMathOperator{\Spec}{Spec}
\DeclareMathOperator{\SU}{SU}
\DeclareMathOperator{\supp}{supp}

\DeclarePairedDelimiter{\inner}{\langle}{\rangle}
\DeclarePairedDelimiter{\abs}{\lvert}{\rvert}

\tikzset{%
  vertex/.style={shape=circle,fill=black,minimum size=6pt,inner sep=0},
  framing/.style={shape=rectangle,fill=black,minimum size=6pt,inner sep=0},
  baseline={([yshift=-0.8ex]current bounding box.center)}
}

\newcounter{count:assumps}

\title{The $3$-fold K-theoretic DT/PT vertex correspondence holds}
\author{Nikolas Kuhn, Henry Liu, Felix Thimm}
\date{\today}

\begin{document}

\maketitle

\begin{abstract}
  We prove the $3$-fold DT/PT correspondence for K-theoretic vertices
  via wall-crossing techniques. We provide two different setups,
  following Mochizuki and following Joyce; both reduce the problem to
  $q$-combinatorial identities on word rearrangements. An important
  technical step is the construction of symmetric almost-perfect
  obstruction theories (APOTs) on auxiliary moduli stacks, e.g. master
  spaces, from the symmetric DT or PT obstruction theory. For this, we
  introduce symmetrized pullbacks of symmetric obstruction theories
  along smooth morphisms to Artin stacks.
\end{abstract}

\tableofcontents

\section{Introduction}
\subsection{DT and PT theory}

\subsubsection{}

Let $X$ be a smooth quasi-projective $3$-fold over $\bC$, with the
action of a torus $\sT \coloneqq (\bC^\times)^r$ (where $r$ could be
$0$) such that the fixed locus $X^\sT$ is proper. Donaldson--Thomas
(DT) theory, as originally formulated by Thomas \cite{Thomas2000},
studies the moduli schemes
\[ \DT_{\beta, n}(X) \coloneqq \left\{\begin{array}{c} \text{ideal sheaves } \cI_Z \subset \cO_X \text{ of} \\ 1\text{-dimensional subschemes } Z \end{array} : \begin{array}{c} [Z] = \beta \\ \chi(\cO_Z) = n\end{array}\right\} \]
for $\beta \in H_2(X, \bZ)$ and $n \in \bZ$. Of particular importance
in this subject are the generating series
\begin{equation} \label{eq:DT-partition-function}
  \sZ_\beta^{\DT}(X) \coloneqq \sum_{n \in \bZ} Q^n \int_{[\DT_{\beta, n}(X)]^{\vir}} 1 \in A^\sT_*(\pt)_{\loc}((Q))
\end{equation}
of (equivariant) {\it DT invariants of curves} in $X$, where:
\begin{itemize}
\item $A^\sT_*(-)$ is the $\sT$-equivariant Chow group, and
  $A^\sT_*(-)_{\loc} \coloneqq A^\sT_*(-) \otimes_{A^\sT_*(\pt)} \Frac A^\sT_*(\pt)$;
\item $[\DT_{\beta, n}(X)]^{\vir} \in A^\sT_*(\DT_{\beta, n}(X))_{\loc}$
  is the virtual fundamental class for the perfect obstruction theory
  $\Ext_X^*(\cI_Z, \cI_Z)_0$, defined via $\sT$-equivariant
  localization if necessary.
\end{itemize}
DT invariants of this flavor have deep connections to Gromov--Witten
theory \cite{Maulik2011}, representation theory of quantum groups
\cite{Okounkov2018}, supersymmetric gauge and string theories
\cite{Nekrasov2016}, etc., to give a (non-representative) sample.

\subsubsection{}

It is productive to refine cohomological enumerative invariants into
K-theoretic ones. The K-theoretic analogue of $\sZ_\beta^{\DT}(X)$ is
the series
\begin{equation} \label{eq:DT-K-partition-function}
  \sZ_\beta^{\DT, K}(X) \coloneqq \sum_{n \in \bZ} Q^n \chi\left(\DT_{\beta, n}(X), \hat\cO^\vir\right) \in K_\sT(\pt)_\loc((Q))
\end{equation}
of {\it K-theoretic DT invariants of curves} in $X$, where, in analogy
with \eqref{eq:DT-partition-function}:
\begin{itemize}
\item $K_\sT(-) \coloneqq K_0(\cat{Coh}_\sT(-))$ is the Grothendieck
  K-group of $\sT$-equivariant coherent sheaves;
\item $\cO^\vir \in K_\sT(\DT_{\beta, n}(X))_\loc$ is the virtual
  structure sheaf, see e.g. \cite{CioKap09};
\item $\hat\cO^\vir \coloneqq \cO^\vir \otimes \cK_\vir^{1/2}$
  is the Nekrasov--Okounkov twist \cite{Nekrasov2016} by a square root
  of the virtual canonical bundle, a special but important feature of
  K-theoretic DT-like theories.
\end{itemize}
The integrand in \eqref{eq:DT-K-partition-function} can be much more
general (see \S\ref{sec:DT-PT-descendants}), though $\sZ_\beta^{\DT,K}(X)$
is already a very sophisticated quantity even in the simplest
non-trivial examples.

While K-theoretic enumerative geometry is a rapidly-growing subject,
basic questions are often very difficult because vanishing arguments
via dimension counts are unavailable. Only in special circumstances,
including K-theoretic DT theory using $\hat\cO^\vir$ (but not
$\cO^\vir$), do {\it rigidity} arguments provide an appropriate
substitute.

The remainder of this paper works in equivariant K-theory, but most
statements and results hold equally well in equivariant
cohomology/Chow.

\subsubsection{}

Every DT invariant of curves in $X$ implicitly contains contributions
from freely-roaming points, or $0$-dimensional subschemes, in $X$.
Algebraically, these can be formally removed by considering
$\sZ_\beta^{\DT,K}(X)/\sZ_0^{\DT,K}(X)$. Geometrically, Pandharipande
and Thomas \cite{Pandharipande2009} define the moduli schemes of \emph{stable pairs}
\[ \PT_{\beta, n}(X) \coloneqq \left\{\cO_X \xrightarrow{s} \cF : \begin{array}{c} \cF \text{ pure of dimension } 1 \\ \dim \supp \coker(s) = 0 \end{array}, \, \begin{array}{c} [\cF] = \beta \\ \chi(\cF) = n \end{array}\right\} \]
for $\beta \in H_2(X, \bZ)$ and $n \in \bZ$, which constrain point
contributions to lie within the Cohen--Macaulay curve $\supp \cF$. In
contrast, and we take this perspective from here on, elements in DT
moduli spaces are two-term complexes $\cO_X \twoheadrightarrow \cO_Z$
for an arbitrary curve $Z$. Letting
\[ I^\bullet \coloneqq [\cO_X \xrightarrow{s} \cF] \in D^b\cat{Coh}_\sT(X), \]
these PT moduli spaces carry a perfect obstruction theory
$\Ext_X^*(I^\bullet, I^\bullet)_0$ analogous to that of DT moduli
spaces where
$I^\bullet = [\cO_X \twoheadrightarrow \cO_Z] \simeq \cI_Z$. 
Exactly as in \eqref{eq:DT-partition-function} and
\eqref{eq:DT-K-partition-function} but replacing the $\DT$ moduli
space with $\PT$, define the series
\[ \sZ_\beta^{\PT}(X) \in A_*^\sT(\pt)_\loc((Q)), \qquad \sZ_\beta^{\PT,K}(X) \in K_\sT(\pt)_\loc((Q)) \]
of (equivariant) {\it cohomological} and {\it K-theoretic PT
  invariants of curves} in $X$.

\subsubsection{}

The question of comparing $\sZ_\beta^{\DT,K}(X)/\sZ_0^{\DT,K}(X)$ and
$\sZ_\beta^{\PT,K}(X)$ is the {\it K-theoretic DT/PT correspondence},
to be addressed in this paper using wall-crossing techniques in
equivariant K-theory. While our techniques are most naturally applied
to smooth quasi-projective Calabi--Yau $3$-folds, in this paper we
focus on the case of smooth quasi-projective toric $3$-folds (possibly
Fano), and we prove a stronger DT/PT correspondence at the level of
{\it K-theoretic vertices}.

\subsection{DT and PT vertices}

\subsubsection{}

Suppose $X$ is toric, with the action of its dense open torus $\sT$.
Let $\Delta(X)$ be the toric $1$-skeleton, whose vertices $v$ are
toric charts $\bC^3$ with toric coordinates $\vec t_v$ and whose edges
$e$ are the non-empty double intersections $\bC^\times \times \bC^2$
with toric coordinates $\vec t_e$. Then $\sT$-equivariant localization
and \v Cech cohomology imply
\begin{equation} \label{eq:factorization-partition-function}
  \begin{aligned}
    \sZ_\beta^{\DT, K}(X) &= \sum_{\vec\lambda} \prod_e \sE^K_{\vec\lambda(e)}(\vec t_e) \prod_v \sV^{\DT, K}_{\vec\lambda(e_1), \vec\lambda(e_2), \vec\lambda(e_3)}(\vec t_v) \\
    \sZ_\beta^{\PT, K}(X) &= \sum_{\vec\lambda} \prod_e \sE^K_{\vec\lambda(e)}(\vec t_e) \prod_v \sV^{\PT, K}_{\vec\lambda(e_1), \vec\lambda(e_2), \vec\lambda(e_3)}(\vec t_v)
  \end{aligned}
\end{equation}
factorize into {\it edge} contributions $\sE^K$ and {\it vertex}
contributions $\sV^{\DT,K}$ and $\sV^{\PT,K}$. Here $e_1, e_2, e_3$
denote the three incident edges at each vertex $v$, and the sum is
over all assignments $\vec\lambda$ of an integer partition to each
edge in $\Delta(X)$; see \cite[\S 4]{Maulik2006} for details. One can
arrange this factorization such that the edges $\sE^K$ are simple
combinatorial products which are the same for DT and PT, and indeed
for any other flavor of curve-counting theory. The bulk of the
complexity lies in the K-theoretic vertices
$\sV^{\DT, K}_{\lambda,\mu,\nu}$ and $\sV^{\PT, K}_{\lambda,\mu,\nu}$,
which we define below using a special geometry $X$.

\subsubsection{}
\label{sec:DT-PT-threefold-geometry}

Following \cite{Maulik2011}, let
$\bar X \coloneqq \bP^1 \times \bP^1 \times \bP^1$ with the action of
$\sT \coloneqq (\bC^\times)^3$ by scaling with weights
$t_1, t_2, t_3$. Let
\[ X \coloneqq \bar X \setminus (L_1 \cup L_2 \cup L_3) \]
where the $L_i$ are the three $\sT$-invariant lines passing through
$(\infty, \infty, \infty)$, and let
\[ \iota_i\colon D_i \hookrightarrow \bar X \]
be the three $\sT$-invariant boundary divisors given by setting the
$i$-th coordinate to $\infty$. Clearly each $D_i \cap X \cong \bC^2$,
and for $i \neq j$ the divisors $D_i \cap X$ and $D_j \cap X$ are
disjoint in $X$. Set
\[ D \coloneqq D_1 \cup D_2 \cup D_3. \]
Our sheaves will live on $\bar X$ but have pre-specified restrictions
to $D$, so interesting things occur only in the toric chart
$\bar X \setminus D \cong \bC^3$.

\subsubsection{}

\begin{definition}
  Fix integer partitions $\lambda, \mu, \nu$, corresponding to
  $\sT$-fixed points in $\Hilb(D_i \cap X)$ for $i = 1, 2, 3$
  respectively. They specify a curve class
  \[ \beta_C \coloneqq (|\lambda|, |\mu|, |\nu|) \in H_2(\bar X, \bZ). \]
  For $M \in \{\DT, \PT\}$, let
  \begin{equation} \label{eq:DT-and-PT-moduli-spaces}
    M_{(\lambda,\mu,\nu), n} \coloneqq \left\{I = [\cO_{\bar X} \xrightarrow{s} \cE] : (\iota_i^*I)_{i=1}^3 = (L\iota_i^*I)_{i=1}^3 \cong (\lambda, \mu, \nu)\right\} \subset M_{\beta_C, n}(\bar X)
  \end{equation}
  where $\cong$ means isomorphic to the two-term complex $[\cO_{D_i}
    \twoheadrightarrow \cO_{Z_i}]$ denoted by $\lambda$ or $\mu$ or
  $\nu$. These $M_{(\lambda,\mu,\nu),n}$ carry symmetric \footnote{For
    us, ``symmetric'' always means up to an overall equivariant
    weight; see Definition~\ref{def:obstruction-theories}.} perfect
  obstruction theories given by $\Ext_{\bar X}^*(I, I(-D))$, using
  which we define the {\it equivariant K-theoretic DT} or {\it PT
    vertex}
  \begin{equation} \label{eq:DT-and-PT-vertices}
    \sV^{M, K}_{\lambda,\mu,\nu} \coloneqq \sum_{n \in \bZ} Q^n \chi(M_{(\lambda,\mu,\nu), n}, \hat\cO^\vir) \in K_\sT(\pt)_\loc((Q)),
  \end{equation}
  cf. \eqref{eq:DT-K-partition-function}. This definition is
  equivalent to that of \cite[\S 4]{Maulik2006} or \cite[\S
  4]{Pandharipande2009a}, which use a formal redistribution of
  vertex/edge contributions in a \v Cech cohomology computation. See
  \S\ref{sec:enumerative-invariants} for details.
\end{definition}

\subsubsection{}

The notion of vertices and factorizations of the form
\eqref{eq:factorization-partition-function} have a long history in
mathematical physics, and the equivariant vertices
\eqref{eq:DT-and-PT-vertices} subsume many previous notions of vertex.
Taking the {\it Calabi--Yau limit} $t_1t_2t_3 \to 1$ produces the {\it
  topological vertices} of \cite{Aganagic2005, Iqbal2008, Li2009}, or,
with a bit more care \cite{Nekrasov2016} to incorporate an extra
parameter into the CY limit, the {\it refined topological vertices} of
\cite{Iqbal2009}.

The general principle is that the $\sT$-fixed points in
$M_{(\lambda,\mu,\nu),n}$ are in bijection with plane partitions or
other (labeled) configurations of boxes in $\bZ^3$, with legs along
the three positive axes. Various limits dramatically simplify their
contributions in equivariant localization. Fixed loci for $M = \PT$
are especially complicated and can form positive-dimensional families
if all three legs are non-trivial \cite{Pandharipande2009a}. In this
way, equivariant vertices may be viewed from the lens of enumerative
combinatorics, with connections to statistical mechanics
\cite{Okounkov2006,Jenne2021}.

Even if one only wants to study $\sV^{M,K}$ or $\sZ^{M,K}$ for
$M = \DT$, the $M = \PT$ case usually has nicer properties and confers
many technical advantages. For example, $\sZ^{\PT}$ is a rational
function in $Q$ \cite[Conjecture 2]{Maulik2006a} \cite[Theorem
1.1(b)]{Bridgeland2011} \cite{Pandharipande2013} while $\sZ^{\DT}$ is
not. The DT/PT correspondence implies that it is enough to study
$M = \PT$.

\subsection{The main theorem and possible extensions}
\label{sec:DT-PT}

\subsubsection{}

\begin{theorem}[K-theoretic DT/PT vertex correspondence] \label{thm:dt-pt}
  \[ \sV^{\DT,K}_{\lambda,\mu,\nu} = \sV^{\PT,K}_{\lambda,\mu,\nu} \sV^{\DT,K}_{\emptyset,\emptyset,\emptyset}. \]
\end{theorem}

This theorem is the main goal of this paper. Using the factorization
\eqref{eq:factorization-partition-function}, it is an immediate
corollary that, for smooth toric $3$-folds $X$,
\[ \sZ^{\DT,K}_\beta(X) = \sZ^{\PT,K}_\beta(X) \sZ^{\DT,K}_0(X), \]
which we refer to as the K-theoretic DT/PT correspondence {\it for
  partition functions} of $X$. 

\subsubsection{}

Theorem~\ref{thm:dt-pt} first appeared as a conjecture for
cohomological equivariant vertices \cite[Conjecture
4]{Pandharipande2009a}. Later, it was stated as a conjecture in its
present K-theoretic form in \cite[Equation (16)]{Nekrasov2016}, where
the importance of using $\hat\cO^{\vir}$ instead of $\cO^{\vir}$ was
first recognized.

Some forms of Theorem~\ref{thm:dt-pt} were previously known in
the CY limit $t_1t_2t_3 \to 1$. For up to two non-trivial legs
$\lambda, \mu, \nu$, there is a relatively elementary proof using
transfer matrix techniques \cite{Kononov2021}, which also works for
refined topological vertices. For three non-trivial legs, there is a
much harder combinatorial proof via dimer models \cite{Jenne2021},
assuming a technical conjecture on the smoothness of PT fixed loci for
the CY subtorus $\{t_1t_2t_3=1\} \subset \sT$.

More geometrically, Bridgeland \cite{Bridgeland2011} and Toda
\cite{toda_dtpt, Toda2020} use Hall algebra and wall-crossing
techniques to prove the (non-equivariant) cohomological DT/PT
correspondence for partition functions of {\it arbitrary} smooth
projective Calabi--Yau threefolds. Interestingly, it is known how to
remove the CY restriction here if one uses non-virtual, i.e. ordinary,
fundamental classes \cite{Stoppa2011}. In equivariant cohomology, an
indirect proof of Theorem~\ref{thm:dt-pt} is available by composing
the GW/DT \cite{Maulik2011} and GW/PT \cite[Theorem 21]{Maulik2010}
vertex correspondences.

Morally, all these known cases take place in cohomology. Significantly
more complexity appears in equivariant K-theory, and when working with
vertices instead of partition functions.

\subsubsection{}

From the lens of representation theory, there is an expectation that
the critical cohomology (resp. K-theory) of $\bigsqcup_n
M_{(\lambda,\mu,\nu),n}$, for $M \in \{\DT, \PT\}$, are modules for
shifted $\hat{\lie{gl}}_1$ Yangians (resp. quantum loop algebras)
\cite{Feigin2017, Rapcak2020, Gaiotto2022}, and that fully-equivariant
DT and PT vertices arise as appropriate $qq$-characters
\cite{Nekrasov2016a} of these DT and PT modules. One can then imagine
that the DT and PT modules are related in some way, e.g. taking duals
followed by induction/restriction, and that the DT/PT vertex
correspondence of Theorem~\ref{thm:dt-pt} is what appears after taking
$qq$-characters. In the $1$-legged case there is some preliminary
work in this direction \cite{Liu2021}, but much of the general
technical work remains to be done.

Such a representation-theoretic approach would provide a satisfying 
conceptual explanation for the simple form of Theorem~\ref{thm:dt-pt} 
beyond our proof via geometric wall-crossing (\S\ref{sec:intro-strategy}).

\subsubsection{}

There exist many variations on DT and PT vertices and the DT/PT
correspondence. In what follows, we give some variations where the
techniques that we use to prove Theorem~\ref{thm:dt-pt} should be
almost immediately applicable. Other variations, e.g. higher-rank
\cite{Lo2012, Sheshmani2016} or motivic DT/PT \cite{Davison2021}, may
also be amenable to our techniques with some more work.

\subsubsection{}
\label{sec:DT-PT-descendants}

The integrands $\hat\cO^\vir$ in the vertices
\eqref{eq:DT-and-PT-vertices} may be tensored with {\it descendent
  insertions} --- symmetric polynomials $f$ in
the K-theoretic Chern roots of universal classes
\[ \cO^{[n]} \coloneqq R\pi_*\scI \in K_\sT(M_{(\lambda,\mu,\nu),n}), \]
or some concrete modification thereof, for instance $\cO^{[n]} - \cO$.
Here $\scI$ is the universal family on
$\pi\colon M_{(\lambda,\mu,\nu),n} \times \bar X \to
M_{(\lambda,\mu,\nu),n}$ of the pairs $I$ parameterized by $M$. Let
$\sV^{M,K}_{\lambda,\mu,\nu}(f)$ denote the resulting {\it descendent
  vertices}. A natural question is how descendent insertions transform
under the DT/PT correspondence; namely, whether
\[ \sV^{\DT,K}_{\lambda,\mu,\nu}(f) = \sum_i \sV^{\PT,K}_{\lambda,\mu,\nu}(f_i^{(1)}) \sV^{\DT,K}_{\emptyset,\emptyset,\emptyset}(f_i^{(2)}) \]
for some operation
$\Delta(f) \coloneqq \sum_i f_i^{(1)} \otimes f_i^{(2)}$ where
$f_i^{(1)}, f_i^{(2)}$ are also symmetric polynomials.
Theorem~\ref{thm:dt-pt} says that $\Delta(1) = 1 \otimes 1$. For
non-trivial $f$, a conjectural formula exists for $1$-legged
cohomological equivariant vertices \cite[Conjecture
5.3.1]{Oblomkov2019}, roughly of the form
$\Delta(\ch_k) = \sum_{k_1+k_2=k} \ch_{k_1} \otimes \ch_{k_2}$ up to
some normalization.

We plan to tackle this problem in future work, using a variant of
Joyce's recent universal wall-crossing machine \cite{joyce_wc_2021}
adapted to moduli stacks with symmetric obstruction theories, in
particular moduli stacks of (complexes of) coherent sheaves on CY
$3$-folds.

\subsubsection{}
\label{sec:PT-BS}

If $X$ is the (singular) coarse moduli space of a projective CY
3-orbifold $\fX$ satisfying the hard Lefschetz condition, and
$\pi\colon \tilde X \to X$ is the preferred crepant resolution given
by the derived McKay correspondence, then there is a notion of {\it
  Bryan--Steinberg (BS)} or {\it $\pi$-stable pairs} for $\pi$
\cite{BrSt16}. Roughly, in the same way that PT-stability disallows
$0$-dimensional components in DT-stable curves, BS-stability disallows
those $1$-dimensional components in PT-stable curves of $\tilde X$
which are supported completely in the exceptional fibers of $\pi$. We
expect, cf. \cite[Theorem 6]{BrSt16} \cite[\S 1.1]{Beentjes2022}, that
our wall-crossing techniques can be used to prove the {\it K-theoretic
  PT/BS correspondence}
\[ \sZ^{\PT,K}_\beta(\tilde X) = \sZ^{\BS,K}_\beta(\pi) \sZ^{\PT,K}_{\exc}(\pi), \]
where $\sZ^{\BS,K}_\beta(\pi)$ is the analogue of the partition
function \eqref{eq:DT-K-partition-function} for BS-stable pairs, and
$\sZ^{\PT,K}_{\exc}(\pi)$ is the partition function for PT-stable
pairs supported on exceptional fibers of $\pi$.

Furthermore, if $X$ is toric, then the toric building blocks of $\pi$
are crepant resolutions of $[\bC^3/G]$ where $G \subset \SO(3)$ or
$G \subset \SU(2)$ is a finite subgroup \cite[Lemma 24]{BrGh09}, and
for each such geometry, there should exist a BS vertex
$\sV^{\BS,K}_{\vec\lambda}(G)$ such that the BS analogue of the
factorization \eqref{eq:factorization-partition-function} holds; see
\cite{Liu21} for the simplest case of the type-A surface singularity
$X = (\bC^2/\bZ_n) \times \bC$. We expect our wall-crossing techniques
to also prove the {\it K-theoretic PT/BS vertex correspondence}
\[ \sV^{\PT,K}_{\vec\lambda}(G) = \sV^{\BS,K}_{\vec\lambda}(G) \sV^{\PT,K}_0(G). \]

PT/BS correspondences, for vertices or otherwise, are an important
step in understanding the DT {\it crepant resolution conjecture (CRC)}
\cite[\S 4]{BrCaYo12}, which relates multi-regular curve counts on
$\fX$ with curve counts on $\tilde X$.

\subsubsection{}

A simpler version of the DT CRC
\[ \sZ_0^{\DT}(\fX) = \frac{\sZ_{\exc}^{\DT}(\tilde{X})\, \sZ^{\DT,\vee}_{\exc}(\tilde{X})}{\sZ_0^{\DT}(\tilde{X})}, \]
relates point counts on $\fX$ with exceptional curve counts on $\tilde
X$, where $\sZ^{\DT,\vee}_{\exc}$ is a formal repackaging of the
invariants in $\sZ^{\DT}_{\exc}$. This (cohomological)
``point-counting'' CRC was proved by Calabrese
\cite{Calabrese2012OnTC} via wall-crossing, by comparing DT invariants
on $\tilde{X}$ with counts of perverse coherent sheaves (so-called
{\it non-commutative DT invariants}) for the resolution $\pi$. The
relation between these non-commutative DT invariants of $\pi$ and the
DT invariants of $\tilde{X}$ is also treated by Toda using
wall-crossing \cite{toda_13_ncdt_ii}. We expect that our techniques
and the wall-crossing set-up of \cite{toda_13_ncdt_ii} will directly
give generalizations of Calabrese's results to K-theory and for
K-theoretic vertices.

\subsubsection{}
\label{sec:DT4-PT4}

If $X$ is a Calabi--Yau $4$-fold instead of a $3$-fold, recent
developments \cite{Borisov2017, Oh2023} in constructing virtual cycles
for moduli of sheaves on $X$ mean that a conjectural {\it $4$-fold
  K-theoretic DT/PT vertex correspondence} exists as well
\cite{Cao2022}:
\[ \sV^{\DT,K}_{\pi^1,\pi^2,\pi^3,\pi^4} = \sV^{\PT,K}_{\pi^1,\pi^2,\pi^3,\pi^4} \sV^{\DT,K}_{\emptyset,\emptyset,\emptyset,\emptyset} \]
where the $\pi^i$ are plane partitions. It subsumes the $3$-fold DT/PT
correspondence because these $4$-fold vertices become equal to their
$3$-fold counterparts under a certain limit when the fourth leg is
trivial. However, significantly less computational evidence is
available: explicit calculation of general $4$-fold PT vertices via
equivariant localization, beyond the case of two non-trivial legs
where all fixed points are isolated, only became possible very
recently \cite{Liu2023}.

We believe that our wall-crossing techniques can be extended to
$4$-fold vertices and partition functions, modulo some (difficult!)
technical issues stemming from the usage of $(-2)$-shifted symplectic
structures in the definition of $4$-fold virtual cycles. Preliminary
results on $4$-fold versions of the DT crepant resolution conjecture
also exist \cite{CaKoMo23}, and so it is not unreasonable to also
expect $4$-fold versions of the PT/BS correspondences of
\S\ref{sec:PT-BS}.

\subsection{Strategy and tools}
\label{sec:intro-strategy}

\subsubsection{}

We give {\it two} proofs of Theorem~\ref{thm:dt-pt} in the hopes of
making this paper twice as exciting. Both proofs involve wall-crossing
with respect to a family of weak stability conditions on an ambient
moduli stack, following ideas of Toda \cite[\S
  3]{toda_dtpt}, but they differ in how the wall-crossing step is
implemented. The Mochizuki-style approach in \S\ref{sec:mochizuki-WCF}
directly crosses the relevant wall, using a somewhat ad-hoc setup
stemming from \cite{Mochizuki2009, Nakajima2011, Kuhn2021}. On the
other hand, the Joyce-style approach in \S\ref{sec:joyce-WCF} is a
simplified application of ideas from Joyce's recent {\it universal}
wall-crossing machinery \cite{joyce_wc_2021}, which crosses the wall
by first moving onto it and then off of it in two steps. We believe
\S\ref{sec:mochizuki-WCF} is a good warm-up for the more sophisticated
setting of \S\ref{sec:joyce-WCF}. The content of \S\ref{sec:joyce-WCF}
may also be viewed as a brief exposition of Joyce's universal
wall-crossing machinery in its simplest form.

In both approaches, the end result is a fairly non-trivial identity
which can be written in the language
(Definition~\ref{def:word-rearrangements}) of word rearrangements. In
\S\ref{sec:prop_comb}, we prove an explicit and fairly-involved
combinatorial result in order to simplify this identity into the
desired DT/PT vertex correspondence. On the other hand, in
\S\ref{sec:joyce-combinatorics-trick}, we provide a trick, stemming
from ideas of Toda, which sidesteps all the combinatorics.

In the remainder of this subsection, we outline the wall-crossing
ingredients which are common to both approaches.

\subsubsection{}

To begin, in \S\ref{sec:moduli-stacks} we construct moduli stacks
$\fN_{(\lambda,\mu,\nu),n}$, and \S\ref{sec:stability-conditions}
equips them with a family $(\tau_\xi)_\xi$ of weak stability
conditions, with exactly one wall $\tau_0$ where there are strictly
semistable objects, such that the semistable loci
\[ \fN_{(\lambda,\mu,\nu),n}^{\sst}(\tau^-), \quad \fN_{(\lambda,\mu,\nu),n}^{\sst}(\tau^+), \]
for weak stability conditions $\tau^\pm$ on the two sides of this
wall, are the DT and PT moduli schemes
\eqref{eq:DT-and-PT-moduli-spaces} respectively. On the wall $\tau_0$,
the strictly semistable objects are sums of the form
\[ [\cO_{\bar X} \to \cE] \oplus [0 \to \cZ_1] \oplus \cdots \oplus [0 \to \cZ_k] \]
where all the $\cZ_i$ have zero-dimensional support; for use in the
wall-crossing argument, we therefore also define moduli stacks $\fQ_n$
of such coherent sheaves $\cZ$. Finally,
\S\ref{sec:enumerative-invariants} equips all these stacks with
symmetric obstruction theories, perfect on stable loci, such that the
enumerative invariants $\sN_{(\lambda,\mu,\nu),n}(\tau^\pm) \coloneqq
\chi(\fN_{(\lambda,\mu,\nu),n}^\sst(\tau^\pm), \hat\cO^\vir)$ form the
desired DT/PT vertices \eqref{eq:DT-and-PT-vertices}.

\subsubsection{}
\label{sec:general-framing-argument}

Geometric wall-crossing arguments typically require that
$\tau^\pm$-stable objects decompose into at most two pieces at
$\tau_0$, i.e. a so-called ``simple'' wall-crossing, so that {\it
  master space} techniques are applicable. This is not the case for
$\fN = \fN_{(\lambda,\mu,\nu),n}$. The general strategy
(\S\ref{sec:quiver-framed-stacks}) is to construct auxiliary {\it
  quiver-framed} stacks $\tilde\fN$ such that:
\begin{enumerate}
\item there are easily-understood forgetful maps $\Pi_\fN\colon
  \tilde\fN \to \fN$ (for us, flag variety fibrations);
\item there are weak stability conditions $\tilde\tau^\pm$ on
  $\tilde\fN$, compatible with $\tau^\pm$ under $\Pi_\fN$, with no
  strictly $\tilde\tau^\pm$-semistable objects and proper
  $\sT$-fixed stable loci;
\item the symmetric obstruction theory on $\fN$ lifts in a compatible
  way along $\Pi_\fN$ to a symmetric obstruction theory on all
  semistable=stable loci in $\tilde\fN$;
\item the wall-crossing problem between $\tau^-$ and $\tau^+$ lifts to
  a wall-crossing problem between $\tilde\tau^-$ and $\tilde\tau^+$,
  which may (and will, for us) have more walls, such that stable
  objects decompose into at most two pieces at each wall.
\end{enumerate}
The same must be done for the stacks $\fQ_n$. Properties 2 and 3 imply
that enumerative invariants
$\tilde\sN_{(\lambda,\mu,\nu),n}(\tilde\tau^\pm) \coloneqq
\chi(\tilde\fN^{\sst}(\tilde\tau^\pm), \hat\cO^\vir)$ are
well-defined, and property 1 means that they relate to the original
enumerative invariants $\sN_{(\lambda,\mu,\nu),n}(\tau^\pm)$ in
explicit ways.

To get property 2, we give a general argument in
\S\ref{sec:properness} to prove properness of stable loci in moduli
stacks like $\fN^Q$, using the notion of elementary modification
originated in \cite{Langton1975}. To get property 3, which in the
terminology of \cite{Liu23} is asking for a {\it symmetrized pullback}
of the obstruction theory on $\fN_{(\lambda,\mu,\nu),n}$, we will
require a generalization of perfect obstruction theories due to Kiem
and Savvas.

\subsubsection{}
\label{sec:symmetrized-pullback-intro}

In \S\ref{sec:obstruction-theory-definitions}, we review obstruction
theories on Artin stacks, and what it means for them to be {\it
  perfect} and/or {\it symmetric} of $\sT$-weight $\kappa$. All
objects and morphisms are $\sT$-equivariant unless stated otherwise.

The strategies in \cite[\S 2]{Liu23} for constructing symmetrized
pullbacks of obstruction theories are only applicable in very special
settings, e.g. if $\fN_{(\lambda,\mu,\nu),n}$ were affine or a shifted
cotangent bundle. To get symmetrized pullbacks in general, we use Kiem
and Savvas' notion of {\it almost-perfect obstruction theory (APOT)}
\cite{kiem_savvas_apot}; we do not know how to prove part 1 of the
following theorem using only perfect obstruction theories (POTs).
Roughly, an APOT is an \'etale-local collection of POTs with the data
of how to glue the local obstruction sheaves into a global obstruction
sheaf. Every POT on a DM stack is an APOT, and every APOT induces a
virtual structure sheaf.

The following theorem is the technical heart of this paper and should
be of independent interest.

\begin{theorem} \label{thm:symmetrized-pullback-summary}
  Let $f\colon \fX \to \fM$ be a smooth morphism of Artin stacks over
  a base Artin stack $\fB$ which is smooth and of pure dimension.
  Suppose $\fX$ is a DM stack and $\fM$ has an obstruction theory
  $\phi_{\fM/\fB}\colon \bE_{\fM/\fB} \to \bL_{\fM/\fB}$ which is
  symmetric of $\sT$-weight $\kappa$.
  \begin{enumerate}[label=(\roman*)]
  \item \label{it:symm-pullback-exists} (Symmetrized pullback,
    Theorem~\ref{thm:symmetric-pullback}) On $\fX$, there exists a
    symmetric APOT $\phi_{\fX/\fB}$ of weight $\kappa$ which is
    compatible under $f$ with $\phi_{\fM/\fB}$ in the sense of
    Definition~\ref{def:smooth-and-symmetrized-pullback}.
  \item \label{it:symm-pullback-localization} (Virtual torus
    localization, Theorem~\ref{thm:symmetrized-pullback-localization})
    Let $\cO^\vir_{\fX}$ be the virtual structure sheaf constructed
    from $\phi_{\fX/\fB}$. If $\fX$ is quasi-separated, locally of
    finite type, and has the resolution property \footnote{The
      condition that $\fX$ has the resolution property can actually be
      removed with more care; see the arguments of \cite[Theorem
        D]{Aranha2022}.}, then
    \begin{equation} \label{eq:symmetrized-virtual-localization}
      \hat\cO^\vir_{\fX} = \iota_* \frac{\hat\cO^\vir_{\fX^\sT}}{\hat\se(\cN^\vir)} \in K_\sT(\fX)_\loc
    \end{equation}
    where $\iota\colon \fX^\sT \hookrightarrow \fX$ is the $\sT$-fixed
    locus, $\cO^\vir_{\fX^\sT}$ is constructed from the $\sT$-fixed part of
    $\phi_{\fX/\fB}$,
    \[ \cN^\vir \coloneqq \iota^*\left(f^*\bE_{\fM/\fB}^\vee + \Omega_f^\vee - \kappa^{-1} \Omega_f\right)^{\text{mov}} \in K_\sT(\fX^\sT) \]
    is the $\sT$-moving part of what the virtual normal bundle would
    be if it existed globally, and $\hat\se(\cN^\vir)$ is its
    symmetrized K-theoretic Euler class
    (\S\ref{sec:symmetrized-pullback-localization}).
  \end{enumerate}
\end{theorem}

Part 2 of this theorem generalizes Kiem--Savvas \cite[Theorem
5.13]{kiem_savvas_loc} by relaxing their assumption that $\cN^\vir$
really must exist globally as a two-term complex of vector bundles,
instead of just as a K-theory class.

\subsubsection{}

Finally, with the constructions of
\S\ref{sec:general-framing-argument} and
\S\ref{sec:symmetrized-pullback-intro} finished, we can begin the
actual wall-crossing procedure. To relate
$\tilde\sN_{(\lambda,\mu,\nu),n}(\tilde\tau)$ and
$\tilde\sN_{(\lambda,\mu,\nu),n}(\tilde\tau')$ for $\tilde\tau'$ close
by $\tilde\tau$, the strategy is to construct a {\it master space}, or
possibly a sequence of master spaces, which interpolate between
$\tilde\tau$-stable and $\tilde\tau'$-stable objects, and have
appropriate analogues of properties 1-3 from
\S\ref{sec:general-framing-argument}. ``Interpolation'' here means
there is a $\bC^\times$-action whose fixed loci consists of: a locus
of the $\tilde\tau$-stable objects of interest, a locus of
$\tilde\tau'$-stable objects of interest, and an ``interaction'' locus
which contributes most of the complexity. Then we apply the following
to the master space.

\begin{proposition}[{\cite[Prop. 5.2, Lemma 5.5]{Liu23}}] \label{prop:master-space-relation}
  In the situation of Theorem~\ref{thm:symmetrized-pullback-summary}
  with $\bC^\times \times \sT$-equivariance, suppose in addition that:
  \begin{enumerate}
  \item there is an inclusion of fixed loci
    $\fX^{(x,y)} \subset \fX^y$ for any
    $(x, y) \in \bC^\times \times \sT$, and $\fX^\sT$ is proper;
  \item the $\bC^\times$-moving part of
    $f^*\bE_{\fM/\fB}|_{\fX^{\bC^\times}}$ has the form
    $\cF^\vee - \kappa^{-1} \cF^\vee$ for some
    $\cF \in K_{\bC^\times \times \sT}(\fX^{\bC^\times})$.
  \end{enumerate}
  Let $\cF_{>0}$ and $\cF_{< 0}$ be the parts of $\cF$ with positive
  and negative $\bC^\times$-weight respectively, and set
  $\ind \coloneqq \rank \cF_{> 0} - \rank \cF_{< 0}$. Then
  \begin{equation} \label{eq:master-space-relation}
    0 = \chi\bigg(\fX^{\bC^\times}, \hat\cO^\vir_{\fX^{\bC^\times}} \otimes (-1)^{\ind}(\kappa^{\frac{\ind}{2}} - \kappa^{-\frac{\ind}{2}})\bigg) \in K_{\sT}(\pt)_{\loc}
  \end{equation}
  where $\cO^\vir_{\fX^{\bC^\times}}$ is constructed from the
  $\bC^\times$-fixed part of the APOT $\phi_{\fX/\fB}$.
\end{proposition}

We omit the proof because it is identical to that of \cite[Prop. 5.2,
Lemma 5.5]{Liu23}, namely: take the
$\bC^\times \times \sT$-equivariant localization formula on $\fX$,
which exists by
Theorem~\ref{thm:symmetrized-pullback-summary}\ref{it:symm-pullback-localization},
and apply an appropriate residue map in the $\bC^\times$ coordinate.
Symmetry of the (A)POT (and condition 2 above) is crucial here,
because then the contribution $1/\hat\se(\cN^\vir)$ to equivariant
localization \eqref{eq:symmetrized-virtual-localization} is a product
of rational functions of the form
\[ \frac{(\kappa w)^{1/2} - (\kappa w)^{-1/2}}{w^{1/2} - w^{-1/2}}, \]
which have the well-defined limits $\pm \kappa^{\pm 1/2}$ as $w \to 0$
or $w \to \infty$. This elementary but very powerful observation
allows us to explicitly identify the output of the residue map, namely
the right hand side of \eqref{eq:master-space-relation}.

Note that $\kappa$ must be a non-trivial $\sT$-weight in order for
\eqref{eq:master-space-relation} to give a non-trivial relation
between invariants. Additional work is needed to use our wall-crossing
techniques in a setting where $\kappa = 1$; see e.g.
\cite{kiem_li_wall}, who treat the case of a simple wall-crossing in
DT theory using $\bC^\times$-intrinsic blowups and cosection
localization.

\subsubsection{}

After appropriate normalization, the wall-crossing formula
\eqref{eq:master-space-relation} becomes a relation between
enumerative invariants with coefficients which are the (symmetric)
quantum integers
\begin{equation} \label{eq:quantum-integer}
  [n]_\kappa \coloneqq (-1)^{n-1} \frac{\kappa^{\frac{n}{2}} - \kappa^{-\frac{n}{2}}}{\kappa^{\frac{1}{2}} - \kappa^{-\frac{1}{2}}}.
\end{equation}
These include an unconventional sign $(-1)^{n-1}$ in order to save on
signs elsewhere; see
Proposition~\ref{prop:projective-bundle-pushforward} for motivation.
In particular, $\lim_{\kappa \to 1} [n]_\kappa = (-1)^{n-1} n$.

Many previous wall-crossing results in the CY3 setting use Joyce's
machinery of Ringel--Hall algebras and motivic stack functions. We
expect our equivariant techniques to provide quantizations/refinements
of these results, upon working $\kappa$-equivariantly, in the sense
that the original formulas appear in the $\kappa \to 1$ limit. In
particular we expect this to be true for the contents of
\cite{JoyceSong}.

\subsection{Acknowledgements}

We would like to thank J{\o}rgen Rennemo for support and for inviting
H.L. for a very enjoyable visit to Oslo, where this project began.
H.L. would also like to thank Dominic Joyce for stimulating
discussions about \cite{joyce_wc_2021}.

N.K. was supported by the Research Council of Norway grant number
302277 -- ``Orthogonal gauge duality and non-commutative geometry'';
H.L. was supported by the Simons Collaboration on Special Holonomy in
Geometry, Analysis and Physics, and the World Premier International
Research Center Initiative (WPI), MEXT, Japan.

\section{Symmetrized Pullback using APOTs}
\label{sec:symmetrized_pullback}

\subsection{Obstruction theories and pullbacks}
\label{sec:obstruction-theory-definitions}

\subsubsection{}
In this section, we construct a \emph{symmetrized pullback} operation 
for symmetric obstruction theory along a smooth morphism from a 
Deligne--Mumford stack to an Artin stack. The main difficulty is that 
such a symmetrized pullback doesn't generally exist as a perfect 
obstruction theory -- instead we allow the result to be a symmetric 
\emph{almost perfect obstruction theory}, as introduced in 
\cite{kiem_savvas_apot}. We also prove a localization formula for 
almost perfect obstruction theories obtained by symmetrized pullback, 
and establish that the construction is functorial under composition 
of smooth maps.

\subsubsection{}
\begin{definition} \label{def:obstruction-theories}
  Fix a base Artin stack $\fB$ which is smooth and of pure dimension.
  All stacks, morphisms, etc. in this section will be implicitly over
  $\fB$ and also $\sT$-equivariant unless stated otherwise. Let $\fM$
  be an Artin stack and $D^-_{\cat{QCoh},\sT}(\fM)$ be its derived
  category of bounded-above $\sT$-equivariant complexes with
  quasi-coherent cohomology \cite{Olsson2007}. Let $\bL_{\fM/\fB} \in
  D^-_{\cat{QCoh},\sT}(\fM)$ denote the cotangent complex
  \cite{Illusie1971}.

  A $\sT$-equivariant {\it obstruction theory} of $\fM$ is an object
  $\bE \in D^-_{\cat{QCoh},\sT}(\fM)$, together with a morphism
  $\phi\colon \bE\to \bL_{\fM/\fB}$ in $D^-_{\cat{QCoh},\sT}(\fM)$,
  such that $h^{\geq 0}(\phi)$ are isomorphisms and $h^{-1}(\phi)$ is
  surjective. An obstruction theory is:
  \begin{itemize}
  \item {\it perfect} if $\bE$ is a perfect complex of amplitude contained in
    $[-1,1]$;
  \item {\it symmetric}, of $\sT$-weight $\kappa$, if $\bE$ is a perfect complex and there is an
    isomorphism
    $\Theta\colon \bE\xrightarrow{\sim} \kappa \otimes \bE^\vee[1]$
    such that $\Theta^\vee[1]=\Theta$.
  \end{itemize}
\end{definition}

This notion of perfect obstruction theory (POT), and symmetric POT,
coincides with the one in
\cite{behrend_fantechi_intrinsic_normal_cone,
  behrend_fantechi_symmetric_obstruction_theories} if $\fM$ is a
Deligne--Mumford stack, because of the requirement that $h^1(\phi)$ is
an isomorphism.

\subsubsection{}

Symmetric obstruction theories are a hallmark of DT-like theories in
comparison to other enumerative theories of curves. Requiring an
obstruction theory $\bE$ on $\fM$ to be symmetric confers two
important advantages.

First, although $h^{-2}(\bE) \neq 0$ in general, if $\fM$ has a
stability condition with no strictly semistable objects, then $\bE$ is
automatically perfect on the (open) stable locus $\fM^\st = \fM^\sst$.
This is because stable objects have no automorphisms, so
\[ h^1(\bE)\Big|_{\fM^\st} \cong h^1(\bL_\fM)\Big|_{\fM^\st} = h^1(\bL_{\fM^\st}) = 0. \]
But then by symmetry
\[ h^{-2}(\bE)\Big|_{\fM^\st} \cong \kappa \otimes h^1(\bE)^\vee\Big|_{\fM^\st} = 0 \]
as well. Hence $\bE$ induces a virtual cycle on $\fM^\st$, an
important ingredient for enumerative geometry.

Second, the rational functions appearing in the $\sT$-equivariant
localization formula for these virtual cycles are of a very special
form and are amenable to {\it rigidity} arguments. In particular, our
key computational tool, Proposition~\ref{prop:master-space-relation},
would be drastically less effective using a non-symmetric obstruction
theory.

\subsubsection{}

\begin{definition} \label{spb:def:apot}
  Let $\fX$ be a Deligne--Mumford stack of finite presentation over
  $\fB$. Following \cite[Definition 5.1]{kiem_savvas_loc}, a
  $\sT$-equivariant {\it almost-perfect obstruction theory} (APOT) on
  $\fX$ consists of:
  \begin{enumerate}[label=(\alph*)]
  \item a $\sT$-equivariant \'etale cover $\{U_i\to \fX\}_{i \in I}$;
  \item for each $i$, a $\sT$-equivariant POT
    $\phi_i\colon \bE_i \to \bL_{U_i/\fB}$.
  \end{enumerate}
  We often write $\phi = (U_i, \phi_i)_{i \in I}$ to denote this data.
  This data must satisfy the following conditions:
  \begin{enumerate}[label=(\roman*)]
  \item there exist $\sT$-equivariant transition isomorphisms
    \[ \psi_{ij}\colon h^1(\bE_i^\vee)\Big|_{U_{ij}}\to h^1(\bE_j^\vee)\Big|_{U_{ij}} \]
    which give descent data for a sheaf $\cOb_\phi$ on $\fX$, called
    the obstruction sheaf; \footnote{Here, we follow the presentation
      of Kiem--Savvas, although it is slightly inaccurate: the
      obstruction sheaf is really part of the \emph{data} defining the
      almost perfect obstruction theory.}
  \item for each pair $i,j$, there exists a $\sT$-equivariant \'etale
    covering $\{V_k\to U_{ij}\}_{k \in K}$, such that the
    isomorphism $\psi_{ij}$ of the obstruction sheaves lifts to a
    $\sT$-equivariant isomorphism of the perfect obstruction theories
    over each $V_k$.
  \end{enumerate}
  In addition, if the following conditions are satisfied for some
  $\sT$-weight $\kappa$, then the APOT is {\it symmetric} with weight
  $\kappa$:
  \begin{enumerate}[resume,label=(\roman*)]
  \item for each $i$, the POT $\phi_i\colon \bE_i \to \bL_{U_i/\fB}$
    is symmetric with weight $\kappa$;\label{it:sym-apotiii}
  \item there exists a $\sT$-equivariant isomorphism
    $\kappa \otimes \cOb_\phi \cong h^0(\bL_{\fX/\fB})$. \label{it:sym-apotiv}
  \end{enumerate}
\end{definition}

We think of an APOT as an \'etale-local collection of POTs where only
the local obstruction sheaves, not the POTs themselves, glue to form a
global object. By \cite[Theorem 4.3]{alper_hall_rydh_luna_stacks_20},
after possibly passing to a multiple cover of $\sT$, a $\sT$-equivariant \'etale cover $\{U_i \to \fX\}_i$
always exists and moreover the $U_i$ can be taken to be affine.

\subsubsection{}

\begin{theorem}[{\cite[Definition 4.1]{kiem_savvas_apot}}] \label{thm:APOT-virtual-cycle}
  A $\sT$-equivariant APOT on $\fX$ gives rise to an
  $\sT$-equivariant virtual structure sheaf
  $\cO_\fX^\vir\in K_\sT(\fX)$.
\end{theorem}

We recall the main ideas from Kiem--Savvas, for completeness and for
the reader's convenience.

\begin{proof}[Proof sketch.]
  Recall \cite[Definition 2.3]{Lee04} \cite[Definition
  3.2.1]{CioKap09} that the usual construction of $\cO_\fX^\vir$ using
  a POT $\bE$ involves: the cone stack
  $\fE \coloneqq h^1/h^0(\bE^\vee)$, an embedding
  $\fC_{\fX} \hookrightarrow \fE$ of the intrinsic normal cone
  $\fC_{\fX}$, and a K-theoretic Gysin map
  $0^!_{\fE}\colon K_\sT(\fE) \to K_\sT(\fX)$.

  For an APOT $\phi$, Kiem and Savvas' idea is to work with the {\it
    sheaf stack} $\cOb_\phi$, which is approximately the coarse moduli
  stack $h^1(\bE^\vee)$ of $\fE$. To be precise, $\cOb_\phi$ is (an
  abuse of notation for) the stack which to a morphism $\rho\colon W
  \to \fX$ from a scheme $W$ assigns the abelian group $H^0(W, \rho^*
  \cOb_\phi)$. This is not an algebraic stack, but its K-group
  $K_\sT(\cOb_\phi)$ and other objects can be defined in terms of what
  Kiem--Savvas call \emph{affine local charts} of the form
  \[ (U \xrightarrow{\rho} \fX, \cE \xrightarrow{r} \rho^* \cOb_\phi). \]
  Here everything is $\sT$-equivariant, $\rho$ is an \'etale map from
  an affine scheme $U$, and $r$ is a surjection from a vector bundle
  $\cE$ on $U$. Working on affine local charts, homology sheaves of
  Koszul complexes on each $U$ glue by standard descent arguments, and
  this suffices construct the desired Gysin map
  \[ 0^!_{\cOb_\phi}\colon K_\sT(\cOb_\phi) \to K_\sT(\fX), \]
  see \cite[Section 2]{kiem_savvas_apot}. Similarly, there is a sheaf
  stack $\fn_{\fX/\fB} \coloneqq h^1(\bL_{\fX/\fB}^\vee)$ and a closed
  substack $\fc_{\fX/\fB} \subset \fn_{\fX/\fB}$ (cf.
  \cite{chang_li_spot}), which Kiem--Savvas call the {\it coarse
    intrinsic normal sheaf} and {\it coarse intrinsic normal cone}
  respectively. Using the transition maps in
  Definition~\ref{spb:def:apot}(ii), the closed embeddings
  $h^1(\phi_i^\vee)\colon \fn_{U_i/\fB}\hookrightarrow \cOb_{\phi_i}$
  glue to form a global closed embedding
  \[ j_\phi\colon \fn_{\fX/\fB} \hookrightarrow \cOb_\phi \]
  of sheaf stacks, and so $\fc_{\fX/\fB} \subset \fn_{\fX/\fB}$ embeds
  as a closed substack $\fc_\phi \coloneqq j_\phi(\fc_{\fX/\fB})$ of
  $\cOb_\phi$ \cite[Theorem 3.4]{kiem_savvas_apot}. Then, using the
  Gysin map, we define
  $\cO^\vir_\fX \coloneqq 0^!_{\cOb_\phi}\cO_{\fc_\phi}$ as usual.
\end{proof}

\subsubsection{}
\label{sec:smooth-and-symmetrized-pullback-definition}

\begin{definition} \label{def:smooth-and-symmetrized-pullback}
  Let $\fN$ be another Artin stack over the base $\fB$, and let
  $f\colon \fM \to \fN$ be a smooth morphism of Artin stacks over
  $\fB$. Suppose
  $\phi_{\fM/\fB}\colon \bE_{\fM/\fB} \to \bL_{\fM/\fB}$ and
  $\phi_{\fN/\fB}\colon \bE_{\fN/\fB} \to \bL_{\fN/\fB}$ are
  obstruction theories.
  \begin{itemize}
  \item The two obstruction theories $\phi_{\fM/\fB}$ and
    $\phi_{\fN/\fB}$ are {\it compatible under $f$} if there is a
    morphism of exact triangles
    \begin{equation} \label{eq:smooth-pullback}
      \begin{tikzcd}
        \bL_f[-1] \ar[equals]{d} \ar{r}{\delta} & f^*\bE_{\fN/\fB} \ar{d}{f^*\phi_{\fN/\fB}} \ar{r} & \bE_{\fM/\fB} \ar{d}{\phi_{\fM/\fB}} \ar{r}{+1} & {} \\
        \bL_f[-1] \ar{r} & f^*\bL_{\fN/\fB} \ar{r} & \bL_{\fM/\fB} \ar{r}{+1} & {}
      \end{tikzcd}.
    \end{equation}
  \item (\cite[Definition 2.3]{Liu23}, cf. \cite[Theorem
    0.1]{Park2021}) Suppose $\phi_{\fM/\fB}$ and $\phi_{\fN/\fB}$ are
    symmetric of weight $\kappa$. Then the two symmetric obstruction
    theories are {\it compatible under $f$} if there are morphisms of
    exact triangles
    \begin{equation} \label{eq:symmetrized-pullback}
      \begin{tikzcd}
        \bL_f[-1] \ar[equals]{d} \ar{r} & \kappa \otimes \bF^\vee[1] \ar{d}{\eta} \ar{r}{\zeta} & \bE_{\fM/\fB} \ar{d}{\zeta^\vee} \ar{r}{+1} & {} \\
        \bL_f[-1] \ar[equals]{d} \ar{r}{\delta} & f^*\bE_{\fN/\fB} \ar{d}{f^*\phi_{\fN/\fB}} \ar{r}{\eta^\vee} & \bF \ar{d} \ar{r}{+1} & {} \\
        \bL_f[-1] \ar{r} & f^*\bL_{\fN/\fB} \ar{r} & \bL_{\fM/\fB} \ar{r}{+1} & {}
      \end{tikzcd}
    \end{equation}
    such that the third column is $\phi_{\fM/\fB}$.
  \end{itemize}
  From a different perspective, if only the (symmetric) obstruction
  theory $\phi_{\fN/\fB}$ is given, then a (symmetric) object
  $\bE_{\fM/\fB}$ along with relevant morphisms fitting into
  \eqref{eq:smooth-pullback} or \eqref{eq:symmetrized-pullback} is
  called a {\it smooth pullback} or {\it symmetrized (smooth)
    pullback} along $f$ of $\phi_{\fN/\fB}$, respectively. This is
  justified as the resulting map $\phi_{\fM/\fB}$ will automatically
  be a (symmetric) obstruction theory. Moreover, in the case of smooth
  pullback, if $\phi_{\fN/\fB}$ is perfect then so is $\phi_{\fM/\fB}$
  by the five lemma. The symmetry of the upper-right square in
  \eqref{eq:symmetrized-pullback} means that the output of symmetrized
  pullback is a symmetric obstruction theory of the same weight.
\end{definition}

\subsubsection{}

Symmetrized pullback was introduced in \cite{Liu23} in order to equip
certain moduli stacks in Vafa--Witten theory with symmetric
obstruction theories. There, the stack $\fN$ was (the classical
truncation of) a $(-1)$-shifted cotangent bundle $T^*[-1]\bar\fN$, and
so symmetrized pullback on $\fN$ was roughly equivalent to $T^*[-1]$
of a smooth pullback on $\bar\fN$, and so the only task in
constructing a symmetrized pullback was to construct the lift $\delta$
in \eqref{eq:smooth-pullback}. This is not difficult in practice
because $\phi_{\fN/\fB}$ arises from functorial constructions like the
Atiyah class.

However, in general, symmetrized pullbacks are rather difficult to
construct. In our situation, we resort to using APOTs, see
\S\ref{sec:symmetrized-pullback-construction}. The \'etale-locally
affine nature of APOTs forces some vanishings that we otherwise do not
have.

\subsection{Smooth pullback of APOTs}
\label{sec:smooth-pullback-APOT}

\subsubsection{}

As a warm-up and for use in
\S\ref{sec:symmetrized-pullback-localization}, we now construct smooth
pullbacks of APOTs. The key idea is that, by working with
APOTs, we no longer require the existence of a lift $\delta$ in the
compatibility diagram \eqref{eq:smooth-pullback}.

\begin{definition} \label{def:compatibility-APOTs}
  Let $f\colon \fX \to \fY$ be a smooth morphism of DM stacks over
  $\fB$. Since $f$ is smooth and of DM-type, $\bL_f=\Omega_{f}$ is a
  locally free sheaf. An APOT $\phi_\fX$ on $\fX$ and an APOT $\phi_\fY$ on
  $\fY$ are {\it compatible under $f$} if:
  \begin{enumerate}
  \item there exists a $\sT$-equivariant \'etale cover $\{U_i \to
    \fY\}_{i \in I}$ and a $T$-equivariant \'etale cover $\{W_j \to
    \fX\}_{j \in J}$ refining $\{f^{-1}(U_i) \to \fX\}_{i \in I}$,
    such that $\phi_\fX|_{W_j}$ and $\phi_\fY|_{U_i}$ are POTs which
    are compatible under $f|_{W_j}:W_j\to U_i$ for each $j$.
  \item $f^*\cOb_{\phi_\fY} = \cOb_{\phi_\fX}$.
  \end{enumerate}
  As in Definition~\ref{def:smooth-and-symmetrized-pullback}, if only
  $\phi_\fY$ is given, we refer to any APOT $\phi_\fX$ satisfying
  these conditions as a {\it smooth pullback} of $\phi_\fY$ along $f$.
\end{definition}

In fact, after sufficient refinement of the \'etale cover, any smooth
pullback APOT is \'etale-locally equivalent to the one constructed in
the following Proposition~\ref{spb:prop:smooth_pullback}, for an
appropriate notion of equivalence of APOTs.

\subsubsection{}

\begin{proposition} \label{spb:prop:smooth_pullback}
  An APOT on $\fY$ admits a smooth pullback to $\fX$ such that
  $\cO^\vir_{\fX} = f^*\cO^\vir_\fY$.
\end{proposition}

\begin{proof}
  Let $(U_i,\phi_i\coloneqq \bE_i\to \bL_{U_i/\fB})_{i\in I}$ be the
  given APOT on $\fY$. By \cite[Theorem
    4.3]{alper_hall_rydh_luna_stacks_20}, after possibly passing to a
  multiple cover of $\sT$, there exists a $\sT$-equivariant affine
  \'etale cover $\{\tilde U_j \to \fX\}_{j\in J}$ refining
  $\{f^{-1}(U_i) \to \fX\}_{i \in I}$, and moreover the $\tilde U_j$
  can be taken to be affine. As we may refine the cover $\{U_i\}_{i\in
    I}$ and restrict the $\phi_i$, we may assume $J=I$ and every
  $\tilde U_i$ gets mapped into $U_i$. Let $f_{i}\colon \tilde U_i \to
  U_i$ be the natural maps.
  
  By Lemma~\ref{lem:affine-vanishing} the morphism $\Omega_{f_i}\to
  f_i^*\bL_{U_i/\fB}[1]$ in the distinguished triangle
  \begin{equation*}
    f_i^*\bL_{U_i/\fB}\to \bL_{\tilde U_i/\fB}\to \Omega_{f_i} \xrightarrow{+1}
  \end{equation*}
  vanishes, which gives us a $\sT$-equivariant splitting
  \begin{equation} \label{eq:smooth-pullback-cotangent-complex-splits-locally}
    \bL_{\tilde U_i/\fB} \cong f_i^*\bL_{U_i/\fB}\oplus \Omega_{f_i}.
  \end{equation}
  We fix this splitting for each $i$ and define the smooth pullback
  APOT on $\fX$ by taking the cover $\{\tilde U_i\}_{i\in I}$ above and the
  morphisms
  \[ \tilde\phi_i\colon \tilde{\bE}_i \coloneqq f_i^*\bE_i \oplus \Omega_{f_i} \xrightarrow{f_i^*\phi_i\oplus \id} f_i^*\bL_{U_i/\fB}\oplus \Omega_{f_i} \xrightarrow{\sim} \bL_{\tilde U_i/\fB}. \]
  By construction of $\tilde{\bE}_i$, the local obstruction sheaves
  $h^1(\tilde{\bE}_i^\vee)$ are isomorphic to $h^1(f^*\bE_i^\vee)$ and
  glue to give the obstruction sheaf
  $\cOb_{\tilde\phi}=f^*\cOb_{\phi}$.

  \subsubsection{}

  It remains to verify that $\tilde\phi \coloneqq (\tilde U_i,
  \tilde\phi_i)_{i \in I}$ is indeed an APOT, by checking property
  (ii) in the definition. For a pair $i,j$, we may choose a
  $\sT$-equivariant cover by affine \'etale neighborhoods $\tilde
  V_k\to \tilde U_{i} \times_\fX \tilde U_{j}$ such that $\tilde V_{k}$
  factors through a neighborhood $V_{l}$ of $U_{i}\times_\fY U_{j}$.
  We get a compatible isomorphism
  \[ \tilde{\bE}_{i}\Big|_{\tilde V_{k}}\simeq f^*\bE_{i}\Big|_{\tilde V_k}\oplus \Omega_{\tilde V_{k}/V_{l}}\to f^*\bE_{j}\Big|_{\tilde V_k}\oplus \Omega_{\tilde V_{k}/V_{l}}\simeq \tilde{\bE}_{j}\Big|_{\tilde V_{k}}\]
  as the upper triangular map determined as follows. Along the
  diagonal, we take (a lift of) $f^*\phi_{ij}$ and $\id_{\Omega_f}$.
  This guarantees compatibility with the obstruction sheaf.
  Off-diagonal, the map $\Omega_{\tilde V_{k}/V_l} \to
  f^*\bE_{j}|_{\tilde V_k}$ is determined via
  \[ \Omega_{f_i} \hookrightarrow \bL_{\tilde U_i/\fB} \to h^0(\bL_{\tilde U_i/\fB})\simeq h^0(\tilde{\bE}_j)\to h^0(f^*\bE_j) \]
  where the first inclusion comes from the splitting
  \eqref{eq:smooth-pullback-cotangent-complex-splits-locally}. Here we
  use that for a vector bundle $W$ on an affine scheme $U$ and an
  element $C$ of $D^-_{\cat{QCoh},\sT}(U)$, we have $\Hom(W,C)\simeq
  \Hom(W,\tau_{\geq 0}C)$. This choice guarantees that the local
  isomorphism of obstruction theories is compatible with the maps to
  cotangent complexes.

  \subsubsection{}

  Finally, we prove that the induced virtual structure sheaves satisfy
  $\cO^\vir_\fX = f^*\cO^\vir_\fY$. By construction,
  $\cOb_{\tilde\phi} = f^*\cOb_\phi$, and recall that
  $\cO_\fY^\vir \coloneqq 0_{\cOb_{\phi}}^![\cO_{\fc_\phi}]$. So, we
  want to show
  \[ f^*0_{\cOb_\phi}^!\cO_{\fc_{\phi}} = 0_{\cOb_{\tilde\phi}}^!\cO_{\fc_{\tilde\phi}}. \]
  By the functoriality of $0^!$ \cite[Lemma 3.8]{kiem_savvas_loc}, we
  only need to show that
  $f^*\cO_{\fc_{\phi}} = \cO_{\fc_{\tilde\phi}}$. This reduces to
  proving the equality of closed substacks
  $\fc_{\tilde\phi} = f^{-1}\fc_{\phi}$ in $\cOb_{\tilde\phi}$. Consider the distinguished
  triangle
  \[ f^*\bL_{\fY/\fB}\to \bL_{\fX/\fB}\to \bL_f\xrightarrow{+1}. \]
  Using \cite[Prop. 2.7, Prop.
  3.14]{behrend_fantechi_intrinsic_normal_cone} and smoothness of $f$,
  we obtain the vertical morphisms in the commutative
  diagram
  \[ \begin{tikzcd}
      \fC_{\fX/\fB} \arrow[hookrightarrow]{r} \arrow[twoheadrightarrow]{d} & \fN_{\fX/\fB}\arrow{r}\arrow[twoheadrightarrow]{d} & \fn_{\fX/\fB}\arrow[hookrightarrow]{r}{j_{\tilde\phi}}\arrow[sloped]{d}{\sim} & \cOb_{\tilde\phi}\arrow[equals]{d}\\
      f^*\fC_{\fY/\fB} \arrow[hookrightarrow]{r} & f^*\fN_{\fY/\fB} \arrow{r} & f^*\fn_{\fY/\fB}\arrow[hookrightarrow]{r}{f^*j_{\phi}} & f^*\cOb_{\phi}.
    \end{tikzcd} \]
    Here, the leftmost square is cartesian. By definition, $\fc_{\fX/\fB}$ is the image of $\fC_{\fY/\fB}$. For an affine test-scheme $T$, one has that $\fc_{\fX/\fB}(T)\subset\fn_{\fX/\fB}(T)$ is the set of sections which admit a lift to $\fN_{\fX/\fB}$ that factors through $\fC_{\fX/\fB}$ (see also \cite[Lemma 2.3]{chang_li_spot}). Using this, one checks immediately that $f^{-1}\fc_{\fY/\fB}\subseteq f^*\fN_{\fY/\fB}$ is the image of $f^*\fC_{\fY/\fB}$ and that the latter agrees with $\fc_{\fX/\fB}$ under the identification $\fn_{\fX/\fB}\simeq f^*\fn_{\fY/\fB}$. 
    We conclude that
  \[ \fc_{\tilde\phi} = j_{\tilde\phi}(\fc_{\fX/\fB}) = (f^*j_\phi) (f^{-1}\fc_{\fY/\fB}) = f^{-1}j_\phi(\fc_{\fY/\fB}) = f^{-1}\fc_\phi. \qedhere \]
\end{proof}

\subsection{Symmetrized pullback of obstruction theories}
\label{sec:symmetrized-pullback-construction}

\subsubsection{}

\begin{definition} \label{def:symmetrized-pullback-setup}
  Let $f\colon \fX \to \fM$ be a smooth morphism of Artin stacks over
  $\fB$, where $\fX$ is a DM stack. Since $f$ is smooth and of
  DM-type, $\bL_f=\Omega_{f}$ is a locally free sheaf. A symmetric
  APOT $\phi_\fX$ on $\fX$ and a symmetric obstruction theory
  $\phi_\fM$ on $\fM$, both of weight $\kappa$, are {\it compatible
    under $f$} if there exists a $\sT$-equivariant \'etale cover
  $\{U_i \to \fX\}_{i \in I}$ such that $\phi_\fX|_{U_i}$ and
  $\phi_\fM$ are symmetric obstruction theories which are compatible
  under $f$ for each $i$.
  
  This definition should be compared to
  Definition~\ref{def:compatibility-APOTs} for the compatibility of
  two APOTs. We refer to any symmetric APOT $\phi_\fX$ satisfying
  these conditions as a {\it symmetrized pullback} of $\phi_\fM$ along
  $f$.
\end{definition}

\subsubsection{}

In \cite[\S 2.4]{Liu23}, the following strategy for constructing a
symmetrized pullback of obstruction theories (without using APOTs) was
proposed. Let $\phi_{\fM/\fB}\colon \bE_{\fM/\fB} \to \bL_{\fM/\fB}$
be a symmetric obstruction theory of weight $\kappa$. Recall from
Definition~\ref{def:smooth-and-symmetrized-pullback} that, in order to
construct a {\it smooth} pullback of $\phi_{\fM/\fB}$ along $f$, we
would like to find a dashed map $\delta$ making the square commute in
the following diagram of solid arrows:
\begin{equation}\label{eq:symmetrized-pullback-lifting-diagram}
  \begin{tikzcd}[column sep=4em]
    & \Omega_{f}^{\vee}[2]\otimes \kappa & & \\
    \Omega_{f}[-1]\ar[dashed]{r}{\delta} \ar[equals]{d} \ar[dotted]{ur}{\delta^{\vee}[1] \circ \delta = 0} & f^*\bE_{\fM/\fB} \ar{d}{f^*\phi_{\fM/\fB}} \ar{u}{\delta^\vee[1]} \ar[dashed]{r} & \cone(\delta) \ar[dashed]{d} \ar[dotted]{ul}[swap]{\xi} \ar[dashed]{r}{+1} & {} \\
    \Omega_{f}[-1]\ar{r} & f^*\bL_{\fM/\fB} \ar{r} & \bL_{\fX/\fB}\ar{r}{+1} & {}
  \end{tikzcd}
\end{equation}
This induces $\cone(\delta)$ and all the other
dashed maps, non-canonically. If in addition
$\delta^{\vee}[1] \circ \delta = 0$, then the dotted map $\xi$
exists, and the co-cone of $\xi$ would be a candidate for the {\it
  symmetrized} pullback of $\bE_{\fM/\fB}$ along $f$.

% \begin{equation} \label{eq:symmetrized-pullback-construction}
%   \begin{tikzcd}
%     \bL_f[-1] \ar[dotted]{r}{\xi} \ar[equals]{d} & \cocone(\delta^\vee[1]) \ar[dotted]{r}{\zeta} \ar[dashed]{d} & \cocone(\eta) \ar[dashed]{d} \ar[dotted]{r}{+1} & {} \\
%     \bL_f[-1] \ar{r}{\delta} \ar{d} & f^*\bE_{\fN/\fB} \ar[dashed]{r} \ar{d}{\delta^\vee[1]} & \cone(\delta) \ar[dashed]{d}{\eta} \ar[dashed]{r}{+1} & {}  \\
%     0 \ar{r} & t \otimes \bL_f^\vee[2] \ar[equals]{r} & t \otimes \bL_f^\vee[2] \ar{r}{+1} & {}
%   \end{tikzcd}
% \end{equation}
% induces all dashed arrows, making all but the topmost row into exact
% triangles. Furthermore, the octahedral axiom implies that the topmost
% row in \eqref{eq:symmetrized-pullback-construction} can also be
% completed into an exact triangle. By construction,
% $\phi_{\fM/\fB}\colon \cocone(\eta) \to \cone(\delta) \to
% \bL_{\fM/\fB}$ is an obstruction theory, and if one can verify that
% $\xi = \psi^\vee[1]$, then $\phi_{\fM/\fB}$ is symmetric with weight
% $\kappa$ and compatible with $\phi_{\fN/\fB}$, as desired.

\subsubsection{}
\label{sec:symm-obstr}

In general, without using APOTs, we neither have a lift $\delta$, nor
that $\delta^{\vee}[1] \circ \delta = 0$ if $\delta$ exists. However,
observe that:
\begin{enumerate}[label=(\roman*)]
\item the obstruction to the existence of the lift $\delta$ lies in
  $\Hom(\Omega_f[-1], f^*\cone(\phi_{\fM/\fB})) = \Ext^1(\Omega_f,
  f^*\cone(\phi_{\fM/\fB}))$; \label{it:symm-obstri}
\item if $\delta$ exists, then it is unique up to an element of
  $\Hom(\Omega_f[-1], f^*\cone(\phi_{\fM/\fB})[-1]) =
  \Hom(\Omega_f,f^*\cone(\phi_{\fM/\fB}))$; \label{it:symm-obstrii}
\item for a given $\delta$, the element
  $\delta^{\vee}[1] \circ \delta$ lies in
  $\Hom(\Omega_f[-1], \Omega_f^\vee[2] \otimes \kappa) =
  \Ext^3(\Omega_{f}, \Omega_{f}^{\vee}\otimes
  \kappa)$. \label{it:symm-obstriii}
\end{enumerate}
By the following lemma, if we only want to construct the
symmetrized pullback on \'etale-local affine charts, then all these
Ext groups will be zero.

\begin{lemma} \label{lem:affine-vanishing}
  Let $U$ be an affine scheme with $\sT$-action. Let $V$ be a
  $\sT$-equivariant vector bundle on $U$ and $E\in
  D_{\cat{QCoh},\sT}^{\leq a}(U)$ for some $a\in \bZ$. Then
  $\Ext^i(V,E) = 0$ for $i > a$.
\end{lemma}

\begin{proof}
  One can represent $E$ by a complex of $\sT$-equivariant
  quasi-coherent sheaves in degrees $\leq a$. Then
  \[\Ext^i(V,E) = h^{i}(R\Gamma(E\otimes V^{\vee})),\]
  and the latter vanishes for $i>a$ by
  \cite[\href{https://stacks.math.columbia.edu/tag/0G9R}{Lemma
      0G9R}]{stacks-project}.
\end{proof}

\subsubsection{}

\begin{theorem} \label{thm:symmetric-pullback}
  In the setting of Definition~\ref{def:symmetrized-pullback-setup}, a
  symmetric obstruction theory on $\fM$ admits a symmetrized pullback
  to an APOT on $\fX$.
\end{theorem}

\begin{proof}
  Let $\phi\colon \bE_{\fM/\fB} \to \bL_{\fM/\fB}$ be the symmetric
  obstruction theory on $\fM$. Since $\phi$ is an obstruction theory,
  $\cone(\phi)$ is concentrated in degrees $\leq -2$. Thus,
  Lemma~\ref{lem:affine-vanishing} together with the observations
  \ref{sec:symm-obstr}\ref{it:symm-obstri} and
  \ref{sec:symm-obstr}\ref{it:symm-obstriii} respectively lets us
  conclude that the obstructions to finding $\delta$ and $\xi$ vanish
  if $\fX$ is affine. In particular, on each affine $U_i$ of an affine
  \'etale cover $\{U_i \to \fX\}_{i \in I}$, the obstructions vanish
  and we may choose $\delta_i$ and $\xi_i$ filling in the diagram
  \eqref{eq:symmetrized-pullback-lifting-diagram}

  Let $\hat\bE_{i}\coloneqq \cocone(\xi_{i})$. By construction, this
  comes with a map $\hat\bE_i\to \cone(\delta_i)$ which is an
  isomorphism on $h^1$ and $h^0$ and surjective on $h^{-1}$.
  Composition with the chosen map $\cone(\delta_i) \to \bL_{U_i/\fB}$ 
  yields a map $\hat\phi_i\colon \hat\bE_{i}\to \bL_{U_{i}/\fB}$ which
  makes makes $\hat\bE_i$ into a $\sT$-equivariant POT on $U_i$ over
  $\fB$. Next, we construct transition isomorphisms for these POTs.

  \subsubsection{}

  By the definition of a symmetric APOT (of weight $\kappa$), we must
  construct transition data for the sheaves $h^1(\hat\bE_{i}^{\vee})$
  giving the global obstruction sheaf
  $h^0(\bL_{\fX/\fB}) \otimes \kappa^{-1}$, and construct for each $i$
  an isomorphism
  \[ \psi_i\colon h^1(\hat\bE_i^\vee) \to h^0(\bL_{\fX/\fB})\Big|_{U_i} \otimes \kappa^{-1}. \]
  For a given $i$, we may choose a dashed arrow $\zeta_i$ giving a
  morphism of exact triangles
  \begin{equation}\label{eq:obstr-diagram}
    \begin{tikzcd}
      \cocone(\delta_{i}^{\vee}[1]) \ar{r}\ar[dashed]{d}{\zeta_i} & f^*\bE_{\fM/\fB}\Big|_{U_i} \ar{r}{\delta_{i}^{\vee}[1]} \ar{d} & \Omega_f^\vee\Big|_{U_i}[2]\otimes \kappa \ar[equals]{d} \ar{r}{+1} & {} \\
      \widehat{\bE}_{i}\ar{r} & \cone(\delta_i) \ar{r}{\xi_i} & \Omega_{f}^{\vee}\Big|_{U_i}[2]\otimes \kappa \ar{r}{+1} & {}
    \end{tikzcd}
  \end{equation}
  cf. the upper-right square in \eqref{eq:symmetrized-pullback}. Here,
  we make a distinguished choice of the cocone of $\delta_i^{\vee}[1]$
  by dualizing the dashed row of
  \eqref{eq:symmetrized-pullback-lifting-diagram}, namely
  $\cocone(\delta_i^{\vee}) \coloneqq \cone(\delta_i)^\vee[2]\otimes
  \kappa$ with the induced maps. Dualizing \eqref{eq:obstr-diagram}
  and taking $h^1$, we get a morphism of exact sequences
  \[ \begin{tikzcd}[column sep=1.5em]
      0\ar[r] & h^1(\cone(\delta_{i})^{\vee})\ar[r]\ar[d] & h^1(\widehat{\bE}_{i}^{\vee})\ar[r]\ar{d}{h^1(\zeta_i^\vee)} & \Omega_f\Big|_{U_i}\otimes \kappa^{-1}\ar[d, equals]\ar[r] & h^2(\cone(\delta_{i})^{\vee})\ar[r]\ar[d] & 0 \\
      0\ar[r] & h^1(f^*\bE_{\fM/\fB}^{\vee}\Big|_{U_{i}})\ar[r] & h^1\left(\cocone(\delta^{\vee}_{i}[1])^{\vee}\right)\ar[r] & \Omega_f\Big|_{U_i}\otimes \kappa^{-1}\ar[r] & h^2(f^*\bE_{\fM/\fB}^{\vee}\Big|_{U_{i}})\ar[r] & 0
    \end{tikzcd} \]
  Since $h^j(\cone(\delta_i))^{\vee})\to
  h^j(f^*\bE_{\fM/\fB}^{\vee}|_{U_i})$ are isomorphisms outside of
  $j = -1,0$, it follows by the five lemma that $\zeta_i^{\vee}$
  induces an isomorphism on $h^1$. Thus, we get isomorphisms
  \[ h^1(\hat\bE_i^\vee) \simeq h^1\left(\cocone(\delta^\vee_i[1])^{\vee}\right) \simeq h^0(\cone(\delta_i)) \otimes \kappa^{-1} \simeq h^0(\bL_{\fX/\fB})\Big|_{U_i} \otimes \kappa^{-1}. \]
  We define $\psi_i\colon h^1(\hat\bE_i^\vee) \to
  h^0(\bL_{\fX/\fB})|_{U_i} \otimes \kappa^{-1}$ to be this
  composition.

  \subsubsection{}

  To conclude the construction of the symmetrized pullback APOT, we
  need to show that for any pair $i, j$, we can locally find \'etale
  neighborhoods $V_k \to U_{i}\times_\fX U_{j}$ and isomorphisms of
  obstruction theories
  $\eta_{ijk}\colon \hat\bE_i|_{V_k}\to \hat\bE_j|_{V_k}$ such that
  \[ h^1(\eta_{ijk}^{\vee}) = \psi_{i}^{-1}\Big|_{V_k}\circ \psi_{j}\Big|_{V_k}\colon h^1(\hat\bE_{j}\big|_{V_k}^{\vee})\to h^1(\hat\bE_{i}\big|_{V_k}^{\vee}). \]

  Without loss of generality, we may assume that
  $U_{i} = U_{j} = V_{k} \eqqcolon V$. Note that we have
  $\delta_i=\delta_j\eqqcolon \delta$ since the difference
  $\delta_i-\delta_j$ lies in
  $\Hom_V(\Omega_f, f^*\cone(\phi_{\fM/\fB})) = 0$. This might a
  priori be problematic, since the choice of $\cone(\delta)$ and the
  map to $\bL_{\fX/\fB}$ may be incompatible for different $i$ and
  $j$. The following Lemma~\ref{lem:cone-delta-iso} shows that we may
  safely identify $\cone(\delta_i)$ and $\cone(\delta_j)$.

  \subsubsection{}
  
  \begin{lemma} \label{lem:cone-delta-iso}
    If $\fX$ is affine, then the choice of $\cone(\delta)$ together
    with the morphism $\cone(\delta) \to \bL_{\fX/\fB}$ in
    \eqref{eq:symmetrized-pullback-lifting-diagram} is unique up to
    (not necessarily unique) isomorphism.
  \end{lemma}
  
  \begin{proof}
    Since cones are unique up to isomorphism, we need to show the
    following: for any two choices
    $v_1,v_2\colon \cone(\delta)\to \bL_{\fX/\fB}$ that give an
    morphism of exact triangles in
    \eqref{eq:symmetrized-pullback-lifting-diagram}, there is an
    automorphism $s\colon \cone(\delta)\to \cone(\delta)$ such that
    $v_2 = v_1\circ s$ and such that
    \[ \begin{tikzcd}[row sep=tiny]
        &\cone(\delta)\ar[dd,"s"]\ar[dr,"r"]&\\
        f^*\bE_{\fM/\fB}\ar[ur,"q"]\ar[dr,"q"']& & \Omega_f\\
        &\cone(\delta)\ar[ur,"r"']& 
      \end{tikzcd} \]
    is a commutative diagram, i.e. $s$ is an isomorphism of cones. Now
    consider the commutative diagram of solid arrows
    \begin{equation} \label{eq:obstr-diagram-ambiguity}
      \begin{tikzcd}[column sep=4em]
        f^*\bE_{\fM/\fB} \ar{d}{f^*\phi} \ar{r}{q} & \cone(\delta) \ar[shift left]{d}{v_2} \ar[shift right]{d}[swap]{v_1} \ar{r}[swap]{r} & \Omega_f \ar[equals]{d} \ar[dashed]{dl}{u_0} \ar[dashed, bend right]{l}{u_1} \\
        f^*\bL_{\fM/\fB} \ar{r} & \bL_{\fX/\fB} \ar{r}{r'} & \Omega_f
      \end{tikzcd}
    \end{equation}
    Let $u_v \coloneqq v_2 - v_1$. Since both squares commute, $u_v
    \circ q = 0$ and $r' \circ u_v = 0$. The first vanishing implies
    there exists some $u_0$ such that $u_0 \circ r = u_v$. Since each
    $v_i$ is itself an obstruction theory, we have $h^{\geq -1}(\cone
    v_1) = 0$, and therefore $\Hom(\Omega_f,\cone v_1) = 0$. Hence,
    post-composition with $v_1$ induces an isomorphism
    \[ v_1 \circ\colon \Hom(\Omega_f, \cone(\delta)) \xrightarrow{\sim} \Hom(\Omega_f, \bL_{\fX/\fB}). \]
    This implies that $u_0 = v_1 \circ u_1$ for a unique $u_1$. Let
    $u_2 \coloneqq u_1 \circ r$. We claim that $s \coloneqq \id + u_2$
    has the desired properties. By construction,
    $v_1 \circ s = v_1 + u_v = v_2$. Moreover,
    $s \circ q = q + u_1 \circ r \circ q = q$. Finally,
    $r \circ s = r + r \circ u_2 = r$ because
    $r \circ u_2 = r' \circ v_1 \circ u_1 \circ r = r' \circ u_0 \circ
    r = r' \circ u_v = 0$ as observed earlier.
  \end{proof}

  \subsubsection{}
  \label{sec:compare_xi}

  After using Lemma~\ref{lem:cone-delta-iso} to identify the choice of
  $\cone(\delta)$ for $i$ and $j$, we may assume that all dashed
  arrows in \eqref{eq:symmetrized-pullback-lifting-diagram} are chosen
  identically for $i$ and $j$. Now we must run the ``dualized''
  version of the same argument again for
  $\hat\bE = \cocone(\xi)$. Namely, the lift $\xi$ in
  \eqref{eq:symmetrized-pullback-lifting-diagram} is unique up to an
  element of $\Ext^2(\Omega_f,\Omega_f^{\vee}\otimes \kappa)=0$, just
  like in the observation \ref{sec:symm-obstr} \ref{it:symm-obstrii},
  so we may also assume that $\xi_{i} = \xi_{j} \eqqcolon \xi$. Again,
  this might a priori be problematic, since the choice of
  $\cocone(\xi)$ and the map $\zeta$ in \eqref{eq:obstr-diagram} may
  be incompatible for different $i$ and $j$. For instance, the
  isomorphisms $\psi_i$ depend on the choice of $\zeta_i$. The
  following Lemma~\ref{lem:choose_gamma} shows that we may safely
  identify $\cocone(\xi_i)$ and $\cocone(\xi_j)$.
  \subsubsection{}
  
  \begin{lemma} \label{lem:choose_gamma}
    There exists an isomorphism
    $s\colon \hat\bE_{i}\to \hat\bE_{j}$ that is compatible with
    the composition to $\cone(\delta)$ and such that
    $s \circ \zeta_{i} = \zeta_{j}$.
  \end{lemma}
  
  \begin{proof}
    This is essentially the dual of the argument in
    Lemma~\ref{lem:cone-delta-iso}. Choose an isomorphism of
    $\hat\bE \coloneqq \hat\bE_{i}$ with $\hat\bE_{j}$. By rotating
    \eqref{eq:obstr-diagram}, consider the commutative diagram of
    solid arrows
    \[ \begin{tikzcd}
        \Omega_{f}^{\vee}\Big|_{V}[1]\otimes \kappa \ar{r}{h'} \ar[equals]{d} & \cocone(\delta^{\vee}[1]) \ar{r} \ar[dashed]{dl}{u_0} \ar[shift left]{d}{\zeta_{j}}\ar[shift right]{d}[swap]{\zeta_{i}} & f^*\bE_{\fM/\fB}\Big|_{V} \ar{r}{\delta^{\vee}[1]} \ar{d}  & {} \\
        \Omega_{f}^{\vee}\Big|_{V}[1]\otimes \kappa \ar{r}{h} & \hat\bE \ar[dashed, bend left]{l}[swap]{u_1} \ar{r}{g} & \cone(\delta) \ar{r}{+1} & {}
      \end{tikzcd} \]    cf. \eqref{eq:obstr-diagram-ambiguity}. Let
    $u_\zeta \coloneqq \zeta_j - \zeta_i$. Since $g \circ u_\zeta = 0$,
    there exists some $u_0$ such that $h \circ u_0 = u_\zeta$. On the
    other hand $\cocone(\zeta_1) \simeq \Omega_f|_V[-1]$, so
    pre-composition with $\zeta_1$ induces an isomorphism
    \[ \circ \zeta_1\colon \Hom(\hat\bE_i, \Omega_f^\vee\big|_V[1]) \xrightarrow{\sim} \Hom(\cocone(\delta^\vee[1]), \Omega_f^\vee\big|_V[1]). \]
    This implies that $u_0 = u_1 \circ \zeta_1$ for a unique $u_1$.
    Let $u_2 \coloneqq h \circ u_1$. The desired isomorphism is
    $s \coloneqq \id + u_2$. By construction,
    $g \circ s = g + g \circ h \circ u_1 = g$. Moreover,
    $u_2 \circ \zeta_1 = h \circ u_1 = u_\zeta$ so that
    $s \circ \zeta_1 = \zeta_1 + u_\zeta = \zeta_2$. This finishes the
    proof.
  \end{proof}

  \subsubsection{}

  Finally, we show that each $\hat\bE_i$ is a {\it symmetric} POT. For
  this, we may assume that $\fX = U_i$ is affine and that we have an
  obstruction theory $\hat \bE$ obtained following the strategy of
  filling in \eqref{eq:symmetrized-pullback-lifting-diagram} and
  taking the co-cone of $\xi$. By what we have argued so far, this
  strategy works and the resulting $\hat \bE$ is unique up to
  isomorphism. In particular, after choosing $\delta$, $\cone(\delta)$
  and $\xi$, we may construct $\hat{\bE}$ in the following way: as a
  consequence of the octahedral axiom, we can expand only the upper
  left corner of the following diagram to obtain the full commutative
  $3\times 3$ diagram whose rows and columns are distinguished
  triangles
  \begin{equation}\label{eq:flip-diagram}
  	\begin{tikzcd} 
  		\Omega_f[-1]\ar[r, "\xi^{\vee}"]\ar[d, equals]& \cone(\delta)^{\vee}[-1] \ar[r,"h"]\ar[d]& \hat{\bE}\ar[d, "g"]\ar[r,"+1"]& \,\\
  		\Omega_f[-1]\ar[r, "\delta"] \ar[d] & f^*\bE_{\fM/\fB}\ar[r]\ar[d, "\delta^{\vee}"]& \cone(\delta)\ar[d, "\xi"]\ar[r,"+1"]&\,\\
  		0\ar[r] \ar[d, "+1"] & \Omega_f^{\vee}[2]\otimes \kappa\ar[r, equals]\ar[d,"+1"]& \Omega_f^{\vee}[2]\otimes \kappa\ar[d, "+1"]\ar[r,"+1"]&\,  \\
  			\, & \, &\, 		
  	\end{tikzcd}
  \end{equation}
  Here, we identified the middle right and lower middle entry by
  uniqueness of cones (up to isomorphism), which also identifies the
  lower right entry as $\Omega_f^{\vee}[2]\otimes \kappa$. We also
  recover $\xi$ as the \emph{unique} lift of $\delta^{\vee}$ to a map
  from $\cone(\delta)$. We conclude that the upper right corner of the
  diagram is indeed (isomorphic to) $\hat\bE$.
  
  Applying the functor $- \mapsto (-)^{\vee}[1]\otimes \kappa$ and
  mirroring the diagram along the anti-diagonal transforms the diagram
  into itself, except possibly for the upper right corner.
  
  By the same argument as in Lemma \ref{lem:cone-delta-iso}, we have
  \begin{lemma}
    Consider only the upper two rows in \eqref{eq:flip-diagram}, and
    suppose everything except for $\hat \bE$ is already fixed. The
    choice of $\hat{\bE}$ together with morphisms making the upper two
    rows into a morphism of exact triangles is unique up to
    isomorphism.
  \end{lemma}
  In our situation, this shows that there is a unique isomorphism
  $\Theta:\hat{\bE}\to \hat{\bE}^{\vee}[1]\otimes \kappa $, which
  identifies $g^{\vee}$ with $h$ and $h^{\vee}$ with $g$, and with
  respect to which the whole diagram \eqref{eq:flip-diagram} is
  invariant under $-^{\vee}[1]\otimes \kappa$. Indeed, the remaining
  maps involving $\hat \bE$ are induced from $g$ and $h$ by
  composition.
  
  Furthermore, we see that $\Theta^{\vee}[1]\otimes \kappa$ satisfies
  the same defining property, and so does $\Theta^{s}:=\Theta +
  \Theta^{\vee}[1]\otimes \kappa$. Moreover, the latter is again an
  isomorphism, since $g$ and $h^{\vee}$ are isomorphisms in degrees
  $\geq 0$, while $g^{\vee}$ and $h$ are isomorphisms in degree $\leq
  -1$.
    
  This concludes the proof of Theorem~\ref{thm:symmetric-pullback}.
\end{proof}

\subsubsection{}
\label{sec:symmetrized-pullback-apots}

Theorem~\ref{thm:symmetric-pullback} states that we can pull back a
symmetric obstruction theory along a smooth morphism from a DM stack.
If the base is already DM, then this also works for APOTs. Let
$g\colon \fX\to \fY$ be a smooth morphism of DM stacks over the base
$\fB$. We assume that everything is $\sT$-equivariant.

\begin{proposition} \label{prop:symmetrized-pullback-APOTs}
  A symmetric APOT on $\fY$ admits a symmetric pullback to an APOT on
  $\fX$, such that
  \[ \cO_{\fX}^{\vir} = g^* \cO_{\fY}^{\vir} \otimes \se\left(\Omega_{\fX/\fY}\otimes \kappa^{-1}\right). \]
\end{proposition}

\begin{proof}
  We construct the APOT $\hat\phi$ on $\fX$ as follows: take
  $\cOb_{\hat\phi} \coloneqq \Omega_{\fX/\fB} \otimes \kappa^{-1}$,
  and define local charts by applying
  Theorem~\ref{thm:symmetric-pullback} locally on each \'etale chart on
  $\fY$. Then it remains to glue everything together, using the same
  strategy as for Proposition~\ref{spb:prop:smooth_pullback}.

  Explicitly, given a chart $U_i \to \fY$ with symmetric POT
  $\phi_i\colon \bE_i \to \bL_{U_i/\fB}$, we can find charts
  $V_j\to \fX$ that factor through $U_i \times_{\fY} \fX$ on which
  $\bL_{V_j/\fB} \simeq \Omega_{V_j/U_i}\oplus g^*\bL_{U_i/\fB}$. Then
  set
  \[ \hat{\bE}_j \coloneqq \Omega_{V_j/U_i} \oplus g^*\bE_i \oplus \Omega_{V_j/U_i}^{\vee}\otimes \kappa[1], \]
  and take $\hat\phi_j\coloneqq \id \oplus g^*\phi \oplus 0$. Given
  two charts $V_{j_1}$ and $V_{j_2}$ over $U_i$, let $W_k$ be an
  affine \'etale chart factoring through $V_{j_1}\times_\fX V_{j_2}$.
  Let $\eta\colon \Omega_{\fX/\fY}|_{W_k}\to h^0(g^*\bE_i)$ be the
  morphism defined as the composition
  \[ \Omega_{\fX/\fY}\Big|_{V_{j_1}} \to h^0(\hat\bE_{j_1}) \simeq h^0(\bL_{\fX/\fB}) \simeq h^0(\hat\bE_{j_2})\to h^0(g^*\bE_i), \]
  where the first and last morphism are the natural inclusion and
  projection coming from the definition of $\hat\bE_{j_1}$ and
  $\hat\bE_{j_2}$ respectively. Then the transition map between
  $\hat\bE_{j_1}|_{W_k}$ and $\hat\bE_{j_2}|_{W_k}$ is given
  by the matrix
  \[ \begin{bmatrix}
      \id_{\Omega_{W_k/U_i}} & 0 & 0\\
      \eta & \id_{g^*\bE_i} & 0 \\
      0 &  \eta^{\vee}[1]\otimes \kappa & \id_{\Omega^{\vee}_{W_k/U_i}[1]\otimes \kappa}
    \end{bmatrix} \]

  \subsubsection{}
  
  We prove the claimed identity for virtual structure sheaves. By the
  exact sequence of K\"ahler differentials, we have an exact sequence
  of coherent sheaves on $\fX$
  \[ 0\to g^*\cOb_{\phi} \xrightarrow{\iota} \cOb_{\hat\phi} \to \Omega_{\fX/\fY}\otimes \kappa^{-1} \to 0. \]
  We claim that the induced pushforward
  $\iota_*\colon K_\sT(g^*\cOb_{\phi}) \to K_\sT(\cOb_{\hat\phi})$
  \cite[(3.16)]{kiem_savvas_loc} sends $g^*\cO_{\fc_\phi}$ to
  $\cO_{\fc_{\hat\phi}}$. Assuming this is true,
  we have that
  \[ \cO_{\fX}^{\vir} = 0^!_{\cOb_{\hat\phi}}\cO_{\fc_\phi} = 0^!_{\cOb_{\phi}}\iota^*\iota_*\cO_{\fc_\phi} \]
  by Lemmas 3.8 and 3.7 in \cite{kiem_savvas_loc} respectively.
  Moreover, from \cite[(3.10)]{kiem_savvas_loc} one can see that
  $\iota^*\iota_*$ equals multiplication with
  $\se(\Omega_{\fX/\fY}\otimes \kappa^{-1})$ as desired.

  Now to prove our claim, it is enough to show that the closed
  substack $g^{-1}\fc_{\phi}\subseteq g^*\cOb_{\phi}$ is identified
  with $\fc_{\hat\phi}\subseteq \cOb_{\hat\phi}$ when viewed as a
  closed substack of $\cOb_{\bar\phi}$ via $\iota$ (which, as a map of
  sheaf stacks, is a closed immersion). This can be checked locally.
  In particular, by the explicit local description of $\hat\phi$ in
  \S\ref{sec:symmetrized-pullback-apots}, we may assume that the
  obstruction theory on $\fX$ is obtained by first taking the smooth
  pullback of $\phi$ and then modifying it by adding
  $\Omega_{\fX/\fY}\otimes \kappa^{-1}$ to the obstruction sheaf. The
  result follows from this.
\end{proof}

\subsubsection{}

Suppose that, in the setting of
Definition~\ref{def:symmetrized-pullback-setup}, we have a diagram
\[ \fX \xrightarrow{g} \fY \xrightarrow{f} \fM\]
over $\fB$, where $g$ is a smooth morphism of DM stacks. Then we have
the APOTs $\phi_\fX$ on $\fX$ and $\phi_\fY$ on $\fY$ obtained from
symmetrized pullback along $f\circ g$ and $f$ respectively. The
following functoriality statement is very useful and follows from the
same arguments as in the construction of symmetrized pullback.

\begin{lemma} \label{prop:symmetrized-pullback-functoriality}
  The symmetrized pullback of $\phi_\fY$ along $g$, as in
  Proposition~\ref{prop:symmetrized-pullback-APOTs}, agrees with the
  symmetrized pullback $\phi_\fX$ of the symmetric obstruction theory
  on $\fM$ along $f\circ g$.
\end{lemma}

\subsection{Virtual torus localization}
\label{sec:symmetrized-pullback-localization}

% \todo{Rewrite this slightly to refer to Henry VW where appropriate.}
% We continue in our setup $f:\bar{\fM}\xrightarrow[\fB]{} \fM$ of Situation \ref{sec:symmetrized-pullback-setup}.

% \subsubsection{}
%  In our wall-crossing applications, $\bar{\fM}$ will be an auxilliary moduli space, the so-called master space, with an additional $\bC^*$-action which is compatible with the trivial action on $\fM$. The components of the fixed locus correspond to the various moduli appearing in the wall-crossing formula. The desired wall-crossing formulae are then proven by virtual torus localization on the master space. To make this work in our setup, we now prove a virtual localization formula for the virtual structure sheaf defined by the symmetrized pullback of a symmetric obstruction theory. 

\subsubsection{}

Let $\fX$ be a DM stack with $\sT$-action, and let
$\iota\colon \fX^\sT \hookrightarrow \fX$ be the $\sT$-fixed locus.
Suppose $\fX$ has an (A)POT, and let $\cO_\fX^\vir$ be the induced
virtual structure sheaf. Recall that virtual $\sT$-localization
requires the following two ingredients. 

First, $\iota_*\colon K_\sT(\fX^\sT) \to K_\sT(\fX)$ must become an
isomorphism of $K_\sT(\pt)_\loc$-modules after base change to
$K_\sT(\pt)_\loc$. This holds if $\fX$ is DM of finite type with a
$\sT$-equivariant \'etale atlas \cite[Proposition
5.12]{kiem_savvas_loc}. By \cite[Theorem
4.3]{alper_hall_rydh_luna_stacks_20}, such an atlas will always exist
after possibly passing to a multiple cover of $\sT$.

Second, for locally free sheaves $\cE$ on $\fX^\sT$ of non-trivial
$\sT$-weight, the {\it K-theoretic Euler class}
\[ \se(\cE) \coloneqq \sum_i (-1)^i \wedge^i(\cE^\vee) \in K_\sT(\fX^\sT) \]
must be invertible in $K_\sT(\fX^\sT)_{\loc}$, which is automatic for
for DM stacks. For two such locally free sheaves $\cE_1$ and $\cE_2$,
we write $\se(\cE_1 - \cE_2) \coloneqq \se(\cE_1) / \se(\cE_2)$.
Generally we will only apply $\se$ to the (some notion of the) virtual
normal bundle of $\iota$.

\subsubsection{}

Kiem and Savvas prove a torus localization formula for APOTs on a DM
stack $\fX$ under the assumption that the virtual normal bundle
$\cN_\iota^\vir$ exists globally on $\fX^\sT$ \cite[Assumption
5.3]{kiem_savvas_loc}. We were not able to confirm this assumption in
our setting, and it likely does not hold in general for the
symmetrized pullbacks constructed in
\S\ref{sec:symmetrized-pullback-construction}. However, if a virtual
normal bundle existed in the situation of
Definition~\ref{def:symmetrized-pullback-setup}, it would have
K-theory class
\begin{equation} \label{eq:APOT-Nvir}
  \cN^\vir \coloneqq \left(f^*\bE_\fM^\vee + \Omega_f^\vee - \Omega_f \otimes \kappa^{-1}\right)\Big|_{\fX^\sT}^{\text{mov}}
\end{equation}
where the superscript $\text{mov}$ denotes the $\sT$-moving part. In
this situation, we can show that the localization formula still holds
using $\cN^\vir$, at least if $\fX$ has the resolution property
\cite{totaro_resolution_property}.

\subsubsection{}

\begin{theorem} \label{thm:symmetrized-pullback-localization}
  In the situation of Definition~\ref{def:symmetrized-pullback-setup},
  suppose that $\fX$ has a symmetric APOT $\phi_{\fX/\fB}$ given by
  symmetrized pullback along $f$ of a symmetric obstruction theory on
  $\fM$. If $\fX$ has the resolution property, then
  \begin{equation}\label{eq:symmetrized_pullback_localization}
    \cO_{\fX}^{\vir} = \iota_*\frac{\cO_{\fX^\sT}^{\vir}}{\se(\cN^\vir)} \in K_\sT(\fX)_\loc
  \end{equation}
  where $\cO_{\fX}^{\vir}$ is the virtual structure sheaf for
  $\phi_{\fX/\fB}$, and $\cO_{\fX^\sT}^{\vir}$ is the virtual
  structure sheaf for the $\sT$-fixed part of
  $\phi_{\fX/\fB}|_{\fX^\sT}$.
\end{theorem}

In this setting, let
$\cK_\vir \coloneqq \det(f^*\bE_\fM + \Omega_f - \Omega_f^\vee \otimes \kappa)$
denote the (would-be) virtual canonical bundle on $\fX$. If the
virtual canonical $\det \bE_\fM$ for $\fM$ admits a square root,
possibly up to passing to a multiple cover of $\sT$, then the
Nekrasov--Okounkov {\it symmetrized virtual structure sheaf} is
\[ \hat\cO_\fX^\vir \coloneqq \cO_\fX^\vir \otimes (\det \bE_\fX)^{1/2} \]
and the {\it symmetrized Euler class} is $\hat\se(-) \coloneqq \se(-)
\otimes \det(-)^{1/2}$. Then
Theorem~\ref{thm:symmetrized-pullback-localization} immediately
implies
\[ \hat\cO_\fX^\vir = \iota_*\frac{\hat\cO_{\fX^\sT}^\vir}{\hat\se(\cN^\vir)}. \]

\subsubsection{}

\begin{remark}
  In \cite{kiem_savvas_loc}, the localization formula is only stated
  for the action of a rank one torus. One deduces the general case by
  the following standard strategy: The argument in
  \cite{kiem_savvas_loc} goes through when considering another torus
  $\sT_1$ acting compatibly with the $\bC^*$-action. For general $\sT$,
  one can always choose a generic rank one sub-torus $\sT_0$, such that
  $\fX^T = \fX^{\sT_0}$, and such that $\sT$ splits as $\sT_0\times \sT_1$.
  The localization formula for $\sT$ then follows from the one for $\sT_0$
  while remembering the $\sT_1$-equivariance.
\end{remark}

\subsubsection{}

\begin{proof}[Proof of Theorem~\ref{thm:symmetrized-pullback-localization}.]
  Let $\phi_{\fM/\fB}\colon \bE_{\fM/\fB} \to \bL_{\fM/\fB}$ be the
  original symmetric obstruction theory on $\fM$ over $\fB$. We divide
  the proof of this theorem into three steps. Throughout, we work
  $\sT$-equivariantly.
  \begin{enumerate}
  \item Construct an affine bundle $a\colon \fA \to \fX$ such that, on
    $\fA$, there is a lift
    $\delta\colon a^*\Omega_f[-1] \to a^*f^*\bE_{\fM/\fB}$ for which
    the composition
    $\delta^\vee[1] \circ \delta\colon a^*\Omega_{f}[-1] \to
    a^*f^*\bE_{\fM/\fB} \to a^*\Omega_{f}^{\vee}[2]\otimes \kappa$
    vanishes, as in \eqref{eq:symmetrized-pullback-lifting-diagram}.
    
  \item Show that the $\sT$-localization formula on $\fA$ holds for
    the APOT on $\fA$ obtained from smooth (non-symmetric) pullback
    along $a$ of the APOT on $\fX$.
    
  \item Use the Thom isomorphism
    $a^*\colon K_\sT(\fX) \xrightarrow{\sim} K_\sT(\fA)$ \cite[Theorem
    5.4.17]{ChrGin97} to obtain the desired localization formula on
    $\fX$ from the one on $\fA$.
  \end{enumerate}
	
  \subsubsection{Step 1}
	
  \begin{lemma} \label{lem:splitting-via-affine-fibration}
    Let $\fX$ be a stack with the resolution property. Let
    $\cE \in D^{\leq k_0}\cat{Coh}(\fX)$, let $\cV$ be a vector bundle
    on $\fX$ and let $e \in \Ext^k(\cV, \cE)$ for $k>k_0$. Then there
    exists a surjection of vector bundles
    $p\colon \cV_0 \twoheadrightarrow \cV$ with the following property: if
    $a\colon \fA \to \fX$ denotes the fiber over $\id$ in the induced
    map $\cHom(\cV,\cV_0)\to \cHom(\cV,\cV)$, then $a^*e=0$.
  \end{lemma}

  Note that $a$ is therefore a $\cHom(\cV, \ker p)$-torsor over $\fX$.
  As $\ker p$ is a vector bundle, this is an affine bundle.

  \begin{proof}
    Up to a shift, we may assume that $k_0 = 0$ and $k=1$. Interpret
    $e$ as a morphism $e\colon \cV \to \cE[1]$ and let
    $\cF \coloneqq \cocone(e)$. Then $e=0$ if and only if the natural
    truncation map $\cF \to h^0(\cF) \simeq V$ splits in the derived
    category. By the resolution property, $\cF$ can be represented by
    a complex of vector bundles $\cdots \to \cV_1 \to \cV_0$ in
    non-positive degree, so that the truncation corresponds to a
    surjective map $p\colon \cV_0 \twoheadrightarrow \cV$. By
    construction, on the affine bundle $\fA$ there is a universal
    splitting of this surjection, which induces a splitting of
    $\cF \to \cV$.
  \end{proof}

  \subsubsection{}
  
  Apply Lemma~\ref{lem:splitting-via-affine-fibration} on $\fX$, to
  the class in $\Ext^1(\Omega_f, f^*\cone(\phi_{\fM/\fB}))$ that
  obstructs the existence of a lift $\delta$. This yields an affine
  bundle $a'\colon \fA' \to \fX$ on which the desired lift
  $\delta\colon (a')^*\Omega_f[-1] \to (a')^*f^*\bE_{\fM/\fB}$ exists.

  Apply Lemma~\ref{lem:splitting-via-affine-fibration} again on $\fA'$
  to the class $\delta^{\vee} \circ \delta \in \Ext^3((a')^*\Omega_f,
  (a')^*\Omega_f^{\vee}\otimes \kappa)$. This yields an affine bundle
  $\fA \to \fA'$ on which the pullback of this class vanishes. Let
  $a\colon \fA \to \fA' \to \fX$ denote the induced projection, and
  use $\delta$ again to denote the pullback of $\delta$ to $\fA$.

  \subsubsection{Step 2}\label{proof:apot_vir_loc_step_2}

  Recall the construction of the APOT on $\fX$ from symmetrized
  pullback: on affine \'etale charts $U_i$, fill in the morphisms
  $\delta_i\colon \Omega_f|_{U_i}[-1] \to f^*\bE_{\fM/\fB}$ and
  $\xi_i\colon \cone(\delta_i)|_{U_i} \to
  \Omega_f^\vee|_{U_i}[2] \otimes \kappa$ of
  \eqref{eq:symmetrized-pullback-lifting-diagram}, and take
  $\hat\bE_i \coloneqq \cocone(\xi_i)$. On the other hand, on $\fA$,
  the lift
  $\delta\colon a^*\Omega_f|_{U_i}[-1] \to a^*f^*\bE_{\fM/\fB}$ is
  globally defined, and the condition
  $\delta^\vee[1] \circ \delta = 0$ is satisfied globally and induces
  a global
  $\xi\colon \cone(\delta) \to a^*\Omega_f^\vee[2] \otimes \kappa$.
  Let $\hat\bE_\fA\coloneqq \cocone(\xi)$.

  By the same arguments as in \S\ref{sec:compare_xi}, there are local
  isomorphisms of $\xi$ with the pullbacks of the $\xi_i$ which induce
  isomorphisms of co-cones. In particular, the natural map
  $\hat\bE_{\fA}\to a^*\bL_{\fX}$ is an isomorphism on $h^0$ and
  surjective on $h^{-1}$.

  Now consider the diagram of solid arrows in $D^-_{\cat{QCoh},\sT}(\fA)$
  \[ \begin{tikzcd}
      \Omega_{a}[-1] \ar[r, dashed] \ar[d,equal] & \hat\bE_\fA\ar[d]& \\
      \Omega_{a}[-1]\ar[r] & a^*\bL_{\fX/\fB} \ar[r]&\bL_{\fA/\fB}.
    \end{tikzcd} \]
  If we could find a dashed arrow that makes the diagram commute, then
  the cone would give a perfect obstruction theory on $\fA$. Since
  $a^*\bL_{\fX/\fB}$ is in cohomological degree $\le 0$ and
  $\Omega_{a}[-1]$ is a vector bundle in degree $1$, affine
  \'etale-locally the connecting map $\Omega_a[-1]\to a^*\bL_{\fX}$ is
  zero by Lemma~\ref{lem:affine-vanishing}. Thus, just like in
  \S\ref{sec:smooth-pullback-APOT}, we may choose POTs of the form
  $\Omega_a\oplus \hat\bE_\fA \to \bL_{\fA/\fB}$ \'etale-locally
  on $\fA$, which fit together to form an APOT
  $(U_j, \hat\bF_j\to \bL_{U_j/\fB})_j$ on $\fA$. More precisely, on
  each $U_j$, we choose a section of $\bL_{\fA/\fB} \to \Omega_a$. It
  is straightforward to check that the APOT we obtain in this way
  agrees with the smooth pullback of the APOT on $\fX$ to $\fA$ using
  Proposition~\ref{spb:prop:smooth_pullback}.
	
  With this APOT on $\fA$, we claim that the $\sT$-fixed locus
  $\iota_{\fA}\colon \fA^{\sT} \hookrightarrow \fA$ possesses a global
  virtual normal bundle in the sense of \cite[Assumption
  5.4]{kiem_savvas_loc}. In fact, the virtual normal bundle is
  \[ \cN^{\vir}_{\iota_\fA} \coloneqq \left( \cT_{\fA/\fX} \oplus \hat\bE_\fA \right)\Big|_{\fA^\sT}^{\text{mov}}. \]
  Namely, restricted to each $U_i$, this is isomorphic with the dual
  of the moving part of $\hat\bF_i$, and moreover, on a common
  refinement $V$ of $U_i$ and $U_j$, the isomorphism
  $\hat\bF_i|_V \to \hat\bF_j|_V$ defining the obstruction theory is
  the identity on the $\hat\bE_{\fA}$ summand, and therefore respects
  the obstruction sheaf.

  It follows that the localization formula of \cite[Theorem
  5.15]{kiem_savvas_loc} is applicable to the APOT on $\fA$.

  \subsubsection{Step 3}

  The localization formula \cite[Theorem 5.15]{kiem_savvas_loc} on $\fA$
  gives
  \begin{equation}\label{eq:localize_pullback}
    \cO_{\fA}^{\vir} = (\iota_{\fA})_* \frac{\cO_{\fA^\sT}^{\vir}}{\se(\cN_{\iota_\fA}^{\vir})} = (\iota_{\fA})_*\frac{\cO_{\fA^\sT}^{\vir}}{\se(\cT_{\fA/\fX}\big|_{\fA^\sT}^{\text{mov}}) \, \se(\hat\bE_\fA\big|_{\fA^\sT}^{\text{mov}})}.
  \end{equation}
  Let $\fF_{\fA} \coloneqq a^{-1}(\fX^\sT)$ and factor
  $\iota_\fA = \iota_{\fF} \circ \iota_\fA'$ where
  $\iota_{\fA}'\colon \fA^\sT \hookrightarrow \fF_\fA$ and
  $\iota_\fF\colon \fF_\fA \hookrightarrow \fA$ are the natural closed
  immersions. Note that
  $\se(\hat\bE_{\fA}|_{\fF_\fA}^{\text{mov}}) = a^*\se(\cN^{\vir})$. Using this,
  and the projection formula, we rewrite \eqref{eq:localize_pullback}
  as
  \[ \cO_{\fA}^{\vir} = (\iota_\fF)_*\left( \left((\iota_\fA')_*\frac{\cO_{\fA^\sT}^{\vir}}{\se(\cT_{\fA/\fX}\big|_{\fA^\sT}^{\text{mov}})}\right) \frac{1}{a^*\se(\cN^{\vir})}\right). \]
  On the other hand, one can check that $\iota_\fA'$ is a regular
  closed immersion with normal bundle
  $\cT_{\fA/\fX}|_{\fA^\sT}^{\text{mov}}$. Applying the torus
  localization formula on $\fF_\fA$ gives
  \[ \cO_{\fF_\fA}^{\vir} = (\iota_\fA')_*\frac{\cO_{\fA^\sT}^{\vir}}{\se(\cT_{\fA/\fX}\big|_{\fA^\sT}^{\text{mov}})}. \]
  Combining the last two equalities, and using that the APOTs on $\fA$
  and $\fF_\fA$ are the smooth pullbacks of the APOTs on $\fX$
  and $\fX^\sT$ respectively, we conclude that
  \[ a^*\cO_\fX^{\vir} = \cO_{\fA}^{\vir} = (\iota_\fF)_* \frac{\cO_{\fF_\fA}^\vir}{a^*\se(\cN^{\vir})} = a^* \iota_*\frac{\cO_{\fX^\sT}^{\vir}}{\se(\cN^{\vir})}. \]
  Since $a^*$ is an isomorphism on K-theory, this proves
  \eqref{eq:symmetrized_pullback_localization}.
\end{proof}

\subsection{Comparison of virtual cycles}

\subsubsection{}

By \cite{thomas2020ktheoretic}, the virtual structure sheaf associated
to a POT depends only on the {\it K-theory class} of the POT. In other
words, if a space $Z$ has two possibly-different POTs, it is
sufficient to check that their K-theory classes are equal in order to
conclude that they induce the same $\hat\cO^\vir$. The same should be
true for APOTs and their associated virtual structure sheaves (from
Theorem~\ref{thm:APOT-virtual-cycle}).

In particular, in wall-crossing applications, we will have a master
space $\bM$ with symmetric APOT from symmetrized pullback, and we will
need to identify the induced APOT on various torus-fixed loci $Z
\subset \bM$ (which arise in the virtual localization of
Theorem~\ref{thm:symmetrized-pullback-localization}) with symmetric
APOTs constructed by symmetrized pullback along maps to the actual
moduli spaces of interest.

\subsubsection{}

\begin{proposition}\label{prop:master_space_vir_class_comparison}
  Let $\fM$ and $\fN$ be an Artin stacks with symmetric obstruction
  theories $\phi_\fM\colon \bE_\fM\to \bL_\fM$ and $\phi_\fN\colon
  \bE_\fN\to \bL_\fN$. Consider the following setup
  \begin{equation*}
    \begin{tikzcd}
      Z \ar[d, "f"]\ar[r, hookrightarrow] & \bM \ar[loop right, "\bC^*"]\ar[d, "g"]\\
      \fN & \fM, 
    \end{tikzcd}
  \end{equation*}
  where $f$ and $g$ are smooth morphisms, $\bM$ is an algebraic space
  with $\bC^\times$-action compatible with the $\sT$-action, and $Z
  \subset \bM$ is a $\bC^\times$-fixed component. Then $Z$ has two
  $\sT$-equivariant symmetric APOTs (by
  Theorem~\ref{thm:symmetric-pullback}):
  \begin{enumerate}
  \item the one obtained by symmetrized pullback along $f$;
  \item the $\bC^\times$-fixed part of the restriction of the
    symmetric APOT on $\bM$, obtained by symmetrized pullback along
    $g$.
  \end{enumerate}
  Assume $\bM$ has the $(\sT \times \bC^\times)$-equivariant
  resolution property \cite{totaro_resolution_property}. If we have an
  identity of K-theory classes
  \begin{equation} \label{eq:APOT-k-class-matching}
    f^*\bE_\fN + \Omega_f - \kappa \Omega_f^\vee = g^*\bE_\fM\Big|^f_Z + \Omega_g\Big|_Z^f - \kappa \Omega_g^\vee\Big|_Z^f \in K_\sT(Z),
  \end{equation}
  then the virtual structure sheaves of the two different APOTs on $Z$
  coincide.
\end{proposition}

A POT is in particular an APOT, by choosing any \'etale atlas and
restricting to each \'etale chart, so this proposition can also be
applied when $Z$ has a POT whose K-theory class equals the right hand
side of \eqref{eq:APOT-k-class-matching}.

Both sides of \eqref{eq:APOT-k-class-matching}, particularly the right
hand side, should be compared with the expression \eqref{eq:APOT-Nvir}
for the virtual normal bundle of an APOT.

\subsubsection{}

\begin{proof}
  The idea of the proof is similar to the one for the localization
  formula. We pass to a $(\sT \times \bC^\times)$-equivariant affine
  bundle $b\colon B \to \bM$, where all obstructions to the
  construction of a {\it global} symmetrized pullback vanish. In fact,
  we can use that $\bM$ is an algebraic space to arrange for $B$ to be
  affine, and for smooth pullback along $b$ to produce a {\it POT} on
  $B$; note that in the more general case of
  Theorem~\ref{thm:symmetrized-pullback-localization}, the smooth
  pullback of an APOT to $B$ only yielded another {\it APOT}. There is
  an induced commutative diagram
  \begin{equation}\label{sym:eq:vir_class_comp_thm_jouanolou_setup}
    \begin{tikzcd}
      A \ar[dr,"a"]\ar[r,hookrightarrow] & B_Z \ar[d]\ar[r, hookrightarrow] & B \ar[loop right, "\sT \times \bC^*"]\ar[d, "b"]\\
      & Z \ar[r, hookrightarrow] & \bM \ar[loop right, "\sT \times \bC^*"]
    \end{tikzcd}
  \end{equation}
  where the square is Cartesian and $A \subset B_Z$ is the
  $\bC^\times$-fixed locus. So $a\colon A \to Z$ is a
  $\sT$-equivariant affine bundle and $A$ is affine. On $A$, we will
  use \cite{thomas2020ktheoretic} to show that the virtual structure
  sheaf of the pullback POTs to $A$ only depend on the K-theory class,
  and pass back to $Z$ via the Thom isomorphism
  $K_\sT(Z)\xrightarrow{\sim} K_\sT(A)$.
  
  The equivariant Jouanolou trick of \cite[\S
    3]{totaro_resolution_property}, on $\bM$, produces the desired $b$
  once we find a $(\sT \times \bC^\times)$-equivariant isomorphism
  $\bM \simeq [\tilde W / \GL(n)]$ where $\tilde W$ is a quasi-affine
  scheme. Because $\bM$ has the $(\sT \times \bC^\times)$-equivariant
  resolution property, by \cite[Theorem A]{gross2015tensor} $[\bM/(\sT
    \times \bC^\times)] \simeq [W / \GL(n)]$ for some quasi-affine
  scheme $W$ and some integer $n \ge 0$. Hence we can construct a
  Cartesian square
  \[ \begin{tikzcd}
    \tilde W \ar{r} \ar{d} & \bM \ar{d} \\
    W \ar{r} & {[\bM/(\sT \times \bC^\times)]}
  \end{tikzcd} \]
  where horizontal (resp. vertical) arrows are torsors for $\GL(n)$
  (resp. $\sT \times \bC^\times$). The upper row gives the desired
  isomorphism $\bM \simeq [\tilde W / \GL(n)]$.

\subsubsection{}

  Now, using $a\colon A \to Z$, we argue as in
  \S\ref{proof:apot_vir_loc_step_2}. Denote the resulting virtual
  structure sheaves of the two APOTs by $\cO^\vir_Z$ and
  $\cO^\vir_{Z\subset\bM}$ respectively.

  First, for the symmetrized pullback along $f$, the obstructions to
  constructing a global symmetrized pullback vanish on $A$. As in
  \S\ref{proof:apot_vir_loc_step_2}, this lets us construct a global
  $\hat\bE^Z_A\to a^*\bL_Z$ which is an isomorphism on $h^0$ and
  surjective on $h^{-1}$. Moreover, because $A$ is affine, any map
  $\Omega_a[-1]\to a^*\bL_Z$ vanishes by degree reasons. So, the
  smooth pullback to $A$ of the symmetrized pullback APOT on $Z$ is
  the POT
  \begin{equation} \label{eq:APOT-comparison-symmetric-pullback}
    \phi_A^Z\colon \hat\bE^Z_A \oplus \Omega_a \to \bL_A.
  \end{equation}

  Analogously, the smooth pullback along $b$ to $B$ of the symmetrized
  pullback APOT on $\bM$ is a POT $\hat\bE^\bM_B \oplus \Omega_b \to
  \bL_B$. Restricting to $A$ and taking the $\bC^\times$-fixed part
  give the POT
  \begin{equation} \label{eq:APOT-comparison-fixed-part}
    \phi_A^B\colon \hat\bE^\bM_B\Big|_A^f \oplus \Omega_b\Big|_A^f \to \bL_A
  \end{equation}
  on $A$. We need to show that $\phi_A^B$ agrees with the APOT $a^* \phi^{\bM}_Z$
  obtained by going the other way from $\bM$ to $A$ in the diagram
  \eqref{sym:eq:vir_class_comp_thm_jouanolou_setup}, namely: restrict
  to $Z$ and take the $\bC^\times$-fixed part of the APOT on $\bM$,
  and then the smooth pullback along $a$ to $A$.

\subsubsection{}

  An affine \'etale cover $\{U_i\}_{i \in I}$ of $B$ induces an affine \'etale cover
  $\{U_i^A \}_{i \in I}$ of $A$ by taking fiber products. There
  are equivariant splittings $\bL_{U_i} \cong b_i^* \bL_{U_i} \oplus
  \Omega_{b_i}$ where $b_i \coloneqq b\big|_{U_i}$, and similarly for
  $U_i^A $ and $a_i \coloneqq a\big|_{U_i^A}$. After possibly
  refining the cover, $a^*\phi^{\bM}_Z$ and $\phi_A^B$ are given on $U_i$ by
  \[ a_i^*\left(\hat \bE_{\fM,i}\big|_Z^f\right) \oplus \Omega_{a_i}, \qquad \left(b_i^* \hat \bE_{\fM,i}\right)\big|_Z^f \oplus \Omega_{b_i}\big|_Z^f \]
  respectively, where $(\hat \bE_{\fM,i})_{i \in I}$ is the
  symmetrized pullback APOT on $\bM$. But these two are equal because
  $A$ is the $\bC^\times$-fixed locus.

\subsubsection{}

  It remains to show that the two POTs $\phi_A^Z$ and $\phi_A^B$, from
  \eqref{eq:APOT-comparison-symmetric-pullback} and
  \eqref{eq:APOT-comparison-fixed-part} respectively, induce the same
  virtual structure sheaf. From the above global description, their
  K-theory classes are
  \[ a^*\left( f^*\bE_\fN + \Omega_f - \kappa \Omega_f^\vee\right) + \Omega_a \]
  and
  \begin{align*}
    & b^*\left(g^*\bE_\fM + \Omega_g - \kappa \Omega_g^\vee \right)\Big|_A^f + \Omega_b\Big|_A^f \\
    & \; = a^*\left(g^*\bE_\fM\big|_Z^f + \Omega_g\big|_Z^f - \kappa\Omega_g^\vee\big|_Z^f\right) + \Omega_a
  \end{align*}
  respectively. These two K-theory classes coincide by hypothesis. By
  \cite{thomas2020ktheoretic}, on a quasi-projective scheme, the
  virtual structure sheaf induced by a POT only depends on the
  K-theory class of the POT. By
  Proposition~\ref{spb:prop:smooth_pullback}, we get
  \[ a^*\cO^\vir_Z = a^*\cO^\vir_{Z\subset\bM}. \]
  The proposition follows as $a^*$ is an isomorphism on K-theory.
\end{proof}

\section{The DT/PT setup}
\label{sec:setup}

\subsection{Moduli stacks}
\label{sec:moduli-stacks}

\subsubsection{}

In this section we set up the moduli stacks, stability conditions, and
finally the wall-crossing problem for the DT/PT vertex correspondence.
For the wall-crossing, we follow \cite{toda_dtpt}, where a
triangulated category with weak stability conditions is defined, such
that the DT and PT stability conditions have exactly one wall (where
there are strictly semistables) separating them. We combine this with
the specific quasi-projective geometry from \cite{Maulik2011}, to give
us a wall-crossing problem for DT and PT {\it vertices} with specified
triples of partitions $(\lambda,\mu,\nu)$.

\subsubsection{}
\label{sec:abelian-category}

\begin{definition}
  Continue with the geometry and notation of
  \S\ref{sec:DT-PT-threefold-geometry}. Let
  $\cat{Coh}_{\le 1}(\bar X)$ be the abelian category of coherent
  sheaves on $\bar X$ of (support of) dimension $\le 1$. Define the
  abelian (resp. triangulated) category
  \begin{align*}
    \cat{A} &\coloneqq \inner*{\cO_{\bar{X}},\cat{Coh}_{\le 1}(\bar{X})[-1]}_{\mathrm{ex}} \subset \cat{D} \\
    \cat{D} &\coloneqq \inner*{\cO_{\bar{X}},\cat{Coh}_{\le 1}(\bar{X})[-1]}_{\mathrm{tr}} \subset D^b\cat{Coh}(\bar{X})
  \end{align*}
  to be the smallest extension-closed (resp. triangulated) subcategory
  containing $\cO_{\bar X}$ and $\cat{Coh}_{\le 1}(\bar{X})[-1]$. It
  is known that $\cat{A}$ is a heart of a bounded t-structure on
  $\cat{D}$ \cite[Lemma 3.5]{toda_dtpt}. Objects $I \in \cat{D}$ have
  {\it numerical class}
  \[ \cl{I} \coloneqq (\rk(I), \ch_2(I), \ch_3(I)) \in \bZ \oplus H_2(\bar{X}) \oplus \bZ. \]
  Let $\fM_\alpha = \fM_{r,\beta_C,n}$ denote the moduli stack of
  objects in $\cat{A}$ of numerical class
  $\alpha = (r, -\beta_C, -n)$. For any $\beta_C$ and $n$, the stacks
  $\fM_{1,\beta_C,n}$ and $\fM_{0,0,n}$ are known to be Artin and
  locally of finite type \cite[Lemma 3.15]{toda_dtpt}. Any given
  instance of DT/PT wall-crossing involves a fixed curve class
  $\beta_C$, with $r \in \{0, 1\}$ and varying $n \in \bZ$. For short,
  let $\fM_{\rank=r} \coloneqq \bigsqcup_{\beta_C,n} \fM_{r,\beta_C,n}$.
\end{definition}

\subsubsection{}

\begin{lemma}[{\cite[Lemma 3.11]{toda_dtpt}}] \label{lem:ambient-moduli-stack}
  Let $[I] \in \fM_{\rank=1}$. Then there is an exact sequence in
  $\cat{A}$,
  \[ 0 \to \cI_C \to I \to \cQ[-1] \to 0, \]
  where $\cI_C$ is the ideal sheaf of a $1$-dimensional subscheme
  $C \subset \bar X$ and $\dim \cQ \le 1$. There is
  an isomorphism
  \begin{equation} \label{eq:pair-as-two-term-complex}
    I \cong [\cO_{\bar X} \xrightarrow{s} \cE],
  \end{equation}
  where $\dim \cE \le 1$ and $\dim \coker s = 0$, if and only if
  $\dim \cQ = 0$.
\end{lemma}

\begin{proof}
  The proof of \cite[Lemma 3.11]{toda_dtpt} applies verbatim without
  the semistability assumption on $I$.
\end{proof}

\subsubsection{}
\label{sec:ambient-moduli-stack}

\begin{definition}
  In the setting of Lemma~\ref{lem:ambient-moduli-stack}, let
  \[ \fM_{1,\beta_C,n}^\circ \subset \fM_{1,\beta_C,n} \]
  denote the open substack where indeed $\dim \cQ = 0$, and, as
  before, write $\fM_{\rank=1}^\circ \coloneqq \bigsqcup_{\beta_C,n}
  \fM_{1,\beta_C,n}^\circ$. We view points in $\fM_{\rank=1}^\circ$ as
  two-term complexes \eqref{eq:pair-as-two-term-complex}. It is
  $\fM_{\rank=1}^\circ$, not $\fM_{\rank=1}$, which is most relevant
  to us.
\end{definition}

\subsubsection{}
\label{sec:rigidification}

Since $\cat{A}$ is a $\bC$-linear category, every object in
$\fM_\alpha$ has a group $\bC^\times$ of automorphisms by scaling. Let
$\fM_\alpha^\pl$ denote the $\bC^\times$-rigidification of
$\fM_\alpha$ \cite{Abramovich2008}. Roughly, this means to quotient
away the group $\bC^\times$ of scalar automorphisms from all
automorphism groups. Then the canonical map
\[ \Pi^\pl_\alpha\colon \fM_\alpha \to \fM_\alpha^\pl \]
is a principal $[\pt/\bC^\times]$-bundle for all $\alpha \neq 0$. The
notation $\pl$ stands for {\it projective linear}, exemplified by the
moduli stack $[\pt/\GL(n)]^\pl = [\pt/\PGL(n)]$ of vector spaces up to
projective automorphisms. Later, stable loci will be substacks of
$\fM_\alpha^\pl$ instead of $\fM_\alpha$.

\subsubsection{}
\label{sec:rigidification-of-dim-1-pairs}

\begin{lemma} \label{lem:rigidification-of-dim-1-pairs}
  The $\bC^\times$-rigidification map
  \[ \Pi^\pl_1\colon \fM_{\rank=1}^\circ \cong \fM_{\rank=1}^{\circ,\pl} \times [\pt/\bC^\times] \to \fM_{\rank=1}^{\circ,\pl} \]
  is a {\it trivial} $[\pt/\bC^\times]$-bundle.
\end{lemma}

\begin{proof}
  By definition, all objects in $\fM_{\rank=1}^\circ$ are pairs of the form
  \[ [\cL \xrightarrow{s} \cE], \qquad \cL \cong \cO_{\bar X}, \; \cE \in \cat{Coh}_{\le 1}(\bar X). \]
  Since $\Aut(\cL) \cong \bC^\times$, the $\bC^\times$-rigidification of
  such a pair is equivalent to fixing the extra data of an isomorphism
  $\phi\colon \cL \xrightarrow{\sim} \cO_X$. Hence $\Pi^\pl_1$ has a
  section given by forgetting this extra data, and this section
  trivializes the bundle $\Pi^\pl_1$ \cite[Lemma 3.21]{Laumon2000}.
\end{proof}

This argument applies equally well to Hilbert schemes. For instance,
$\Hilb(D_i)$ is the $\bC^\times$-rigidification of the associated
Hilbert {\it stack} $\fHilb(D_i) \cong \Hilb(D_i) \times
[\pt/\bC^\times]$ of objects of the form $[\cO_{D_i}
  \twoheadrightarrow \cO_Z]$.

\subsubsection{}
\label{sec:def-open}
\begin{definition}
  For a curve class $\beta_C = (\beta_1,\beta_2,\beta_3) \in
  H_2(\bar{X})$ and an integer $n \in \bZ$, let
  \[ \fN_{\beta_C,n} \subset \fM^\circ_{1,\beta_C,n} \]
  be the open locus of two-term complexes $I$ such that, for each $i =
  1, 2, 3$:
  \begin{enumerate}
  \item $L^k\iota_i^* I = 0$ for $k > 0$, so that $L\iota_i^* =
    \iota_i^*$;
  \item the {\it evaluation map}
    \begin{align*}
      \ev_i\colon \fN_{\beta_C,n} &\to \fHilb(D_i \cap X, \beta_i) \\
      [I] &\mapsto [\iota_i^*I],
    \end{align*}
    lands in the Hilbert stack (see
    \S\ref{sec:rigidification-of-dim-1-pairs}) of $\beta_i$ points on
    $D_i \cap X$.
  \end{enumerate}
  Similarly, viewing $\fM_{0,0,n}$ as a moduli stack of
  zero-dimensional sheaves, let
  \[ \fQ_n \subset \fM_{0,0,n} \]
  be the open locus of sheaves whose restriction to $D$ is zero, so
  that the analogous evaluation maps are zero.
\end{definition}

Condition 1 is an open condition by upper semi-continuity of
cohomology. Condition 2 is also an open condition, since $X \subset
\bar X$ is open and the Hilbert stack is a semistable (and stable)
locus in the moduli stack of {\it all} coherent sheaves on $D_i
\subset \bar X$ of Chern character $(1, 0, -\beta_i) \in H^0 \oplus
H^2 \oplus H^4$. It should be compared with the {\it admissibility}
condition for relative ideal sheaves \cite{Maulik2006a, Li2015}.

\subsubsection{}

To avoid dealing with each $\ev_i$ separately, it is convenient to
introduce the open locus
\[ \fHilb^\circ(D) \subset \fHilb(D) \]
consisting of rank-$1$ torsion-free sheaves on $D$ which are locally
free along the singular locus of $D$. Note the isomorphism
\[ \Hilb^{\circ}(D)\coloneqq (\fHilb^{\circ}(D))^{\pl} \cong \Hilb(D_1) \times \Hilb(D_2) \times \Hilb(D_3). \]
We continue to use the same names to denote the $\pl$ version of all
evaluation maps. The total evaluation map
\[ \ev \coloneqq \ev_1 \times \ev_2 \times \ev_3\colon \fN_{\beta_C,n} \to \fHilb^\circ(D) \]
and its $\pl$ version are compatible with the rigidification maps.

\subsubsection{}\label{sec:def_moduli_subspaces}

\begin{definition}
  For points $p_i \in \Hilb(D_i \cap X)$ for $i = 1, 2, 3$, let
  $|p_i|$ be the lengths of the subschemes they define, set
  $\beta_C = (|p_1|, |p_2|, |p_3|)$, and let
  \[ \fN_{(p_1,p_2,p_3),n}^\pl = \ev^{-1}(p_1, p_2, p_3) \subset \fN_{\beta_C,n}^\pl. \]
  be the closed substack of objects with the prescribed intersection
  with the $D_i$. Set
  \[ \fN_{(p_1,p_2,p_3),n} \coloneqq [\pt/\bC^\times] \times \fN_{(p_1,p_2,p_3),n}^\pl \subset \fN_{\beta_C,n}, \]
  where the $\bC^\times$ is the group of scaling automorphisms, so
  that the $\pl$ superscript makes sense because by
  Lemma~\ref{lem:rigidification-of-dim-1-pairs}. In other words,
  automorphisms of objects $[I] \in \fN_{(p_1,p_2,p_3),n}^\pl$ (resp.
  $\fN_{(p_1,p_2,p_3),n}$) must restrict to the {\it identity} (resp.
  any scaling automorphism) on $I|_D$.
\end{definition}

In what follows, $(p_1,p_2,p_3) = (\lambda, \mu, \nu)$ is a triple of
$\sT$-fixed points in $\Hilb(\bC^2)$, i.e. a triple of integer
partitions. The moduli substacks $\fN_{(\lambda,\mu,\nu),n}$ and
$\fQ_n$ will be the relevant moduli stacks involved in wall-crossing.
Later in \S\ref{sec:enumerative-invariants} we will equip them with
obstruction theories and enumerative invariants.

\subsection{Stability conditions}
\label{sec:stability-conditions}

\subsubsection{}
\label{sec:stab_choices}

We endow $\cat{A}$ with the same weak stability condition as Toda
\cite[\S 3.2]{toda_dtpt}. Recall that a {\it weak stability condition}
on an abelian category, in Joyce's sense \cite[Definition
  3.5]{JoyceSong}, is a function $\tau$ taking numerical classes into
some poset, such that for any short exact sequence $0 \to F \to E \to
G \to 0$, either
\[ \tau(F)\leq \tau(E)\leq \tau(G) \quad \text{or} \quad \tau(F)\geq \tau(E)\geq \tau(G). \]
An object $E$ is {\it $\tau$-semistable} (resp. {\it $\tau$-stable})
if for any short exact sequence $0 \to F \to E \to G \to 0$, we have
$\tau(F) \le \tau(G)$ (resp. $\tau(F) < \tau(G)$). We caution that,
for stable objects $E$, it is possible to have inequalities like
$\tau(F) = \tau(E) < \tau(G)$.

\begin{definition} \label{def:stability-conditions}
  For a parameter $\xi \coloneqq (z_1, z_0) \in \bC^2$ where
  $\arg(z_i) \in (\pi/2, \pi)$, let
  \begin{align*}
    \tau_\xi\colon \bZ \oplus H_2(\bar X) \oplus \bZ &\to \bR \\
    (r, -\beta_C, -n) &\mapsto \begin{cases} \arg z_1 & r \neq 0 \\ \pi/2 & r = 0, \; \beta_C \neq 0 \\ \arg z_0 & r = 0, \; \beta_C = 0. \end{cases}
  \end{align*}
  Let $\fM_\alpha^{\circ,\sst}(\tau_\xi) \subset
  \fM_\alpha^{\circ,\pl}$ be the substack of $\tau_\xi$-semistable
  objects, and similarly for $\fN_{\beta_C,n}^\sst(\tau_\xi)$ and
  $\fQ_n^\sst(\tau_\xi)$.
\end{definition}

The choice of the range $(\pi/2, \pi)$ is purely for the sake of
agreement with Toda's more general construction of weak stability
conditions on triangulated categories, and will not be important for
this paper.

\subsubsection{}

\begin{lemma}\label{setup:lemma_wall_crossing_moduli}
  Consider objects in $\fM_{1,\beta_C,n}^{\circ,\sst}(\tau_\xi)$.
  \begin{enumerate}[label=(\roman*)]
  \item If $\arg(z_0) < \arg(z_1)$, then all $\tau_\xi$-semistable
    objects are $\tau_\xi$-stable, and
    \[ \fM^{\circ,\sst}_{1,\beta_C,n}(\tau_\xi) = \DT_{\beta_C,n}(\bar{X}). \]
    In particular, $\fN_{(\lambda,\mu,\nu),n}^{\sst}(\tau_\xi) =
    \DT_{(\lambda,\mu,\nu),n}$ from
    \eqref{eq:DT-and-PT-moduli-spaces}.
  \item If $\arg(z_0) > \arg(z_1)$, then all $\tau_\xi$-semistable
    objects are $\tau_\xi$-stable, and
    \[ \fM^{\circ,\sst}_{1,\beta_C,n}(\tau_\xi) = \PT_{\beta_C, n}(\bar{X}). \]
    In particular, $\fN_{(\lambda,\mu,\nu),n}^{\sst}(\tau_\xi) =
    \PT_{(\lambda,\mu,\nu),n}$ from
    \eqref{eq:DT-and-PT-moduli-spaces}.
  \item If $\arg(z_0) = \arg(z_1)$, then all objects are semistable,
    i.e. $\fM^{\circ,\sst}_{1,\beta_C,n}(\tau_\xi) =
    \fM^{\circ,\pl}_{1,\beta_C,n}$.
  \end{enumerate}
\end{lemma}

Let $\tau^-$, $\tau^+$, and $\tau_0$ denote $\tau_\xi$ for some choice
of $\xi$ in cases (i), (ii), and (iii) above, respectively. Then this
lemma describes a wall in a space of weak stability conditions, and
the moduli spaces on both sides of the wall.

\begin{proof}
  The third statement follows immediately from the definition of the
  stability conditions $\tau_\xi$. The first two statements about
  $\fM^{\circ,\sst}_{1,\beta_C,n}(\tau_\xi)$ were proven in
  \cite[Proposition 3.12]{toda_dtpt}. Since the conditions defining
  $\fN_{(\lambda,\mu,\nu),n}^{\sst}(\tau_\xi)$ inside
  $\fM^{\circ,\sst}_{1,\beta_C,n}(\tau_\xi)$ are the same conditions as
  the ones defining $M_{(\lambda,\mu,\nu),n}$ inside $M_{\beta_C,
    n}(\bar{X})$ in \eqref{eq:DT-and-PT-moduli-spaces}, the
  identifications of $\fM^{\circ,\sst}_{1,\beta_C,n}(\tau_\xi)$ with the
  DT and PT moduli restricts to these loci.
\end{proof}

\subsubsection{}

\begin{lemma}
  For any $\xi$, the semistable loci
  $\fN^\sst_{(p_1,p_2,p_3),n}(\tau_\xi)$ and $\fQ^\sst_n(\tau_\xi)$
  are finite type.
\end{lemma}

\begin{proof}
  It is well-known that the moduli stack $\fQ$ of zero-dimensional
  sheaves is finite type, so we focus on $\fN$.

  From \cite[Lemma 3.15]{toda_dtpt}, $\fM^\circ_{1,\beta_C,n}(\tau_\xi)$
  is of finite type. Since
  \[ \fN^\sst_{(p_1,p_2,p_3),n}(\tau_\xi) \subset \fN^\sst_{\beta_C,n}(\tau_\xi) \subset \fM^{\circ,\pl}_{1,\beta_C,n}(\tau_\xi) \] 
  is a closed followed by open immersion, we are done.
\end{proof}

\subsubsection{}
\label{lem:val-crit-semistable-base}

A one-parameter family of our rank-$1$ semistable objects always has a
limit. This will be important to show properness of auxiliary moduli
stacks later.

\begin{lemma}
  For any $\xi$, the stack $\fM_{1,\beta_C,n}^{\circ,\sst}(\tau_\xi)$,
  as well as its pre-image in $\fM_{1,\beta_C,n}^\circ$ under the
  rigidification map, satisfy the existence part of the valuative
  criterion for properness.
\end{lemma}

\begin{proof}
  It is enough to show the statement for $\fM_{1,\beta_C,n}^{\circ}$.

  Let $R$ be a DVR with generic point $\eta$ and closed point $\xi$.
  Suppose we are given any family of pairs
  $\mathcal{O}_{\bar{X}_\eta}\xrightarrow{\phi_{\eta}}
  \mathcal{E}_{\eta}$ over the generic point of $R$, which is
  semi-stable with respect to $\tau$. Then one can find an $R$-flat
  coherent sheaf $\mathcal{E}_R$ extending $\mathcal{E}_{\eta}$ and an
  extension of the section of $\mathcal{E}_{\eta}$ to a section of
  $\mathcal{E}_{R}$, which gives a pair
  $\mathcal{O}_{\bar{X}_{R}}\xrightarrow{\phi_R} \mathcal{E}_R$ in
  $\fM_{1,\beta_C,n}$.

  We will modify this pair so that fiber over $\xi$ lies in
  $\fM_{1,\beta_C,n}^{\circ}$. Indeed, if it isn't already the case,
  then $\mathcal{F} \coloneqq \coker(\phi_R)|_{\xi}$ has
  one-dimensional support along. We can find a quotient
  $\mathcal{F}\to \mathcal{F}'$, such that $\mathcal{F}'|_{\eta} = 0$,
  and such that the kernel has zero-dimensional support along
  $\xi$.Let $\mathcal{E}_R'$ be the kernel of the composition
  $\mathcal{E}_R\to \mathcal{F}\to \mathcal{F}'$. Then the pair
  $\mathcal{O}_R\to \mathcal{E}_R'$ lies in
  $\fM_{1,\beta_C,n}^{\circ}$.
\end{proof}

%% \subsubsection{}

%% We address the choices we make in \S\ref{sec:quiver-stab} for our
%% situation, to lift these weak stability conditions to the
%% quiver-framed stacks. Concretely, this is a choice of the parameters
%% $\lambda_0\colon H_{\pe} \to \bR$ and
%% $r\colon H_{\pe} \to \bZ_{\ge 0}$ in
%% \eqref{eq:quiver-framed-stack-stability-condition-parameters}. Recall
%% that we fixed a curve class $\beta$ for $H_{\pe}$.

%% \begin{definition} \label{def:quiver_moduli_our}
%%   To prescribe the function $\lambda_0$ is the same as giving the
%%   value $\lambda_0(0,0,-1)$ and the value $\lambda_0(1,-\beta,-n)$ for
%%   a single choice of $n$. We choose these depending on the situation.
%%   For $r$, by boundedness of the moduli stacks,
%%   \[ \fM^{\circ}_{1,\beta,n} = \emptyset \qquad \forall n \ll 0. \]
%%   Fix $n_0$ such that this vanishing holds for all $n \leq n_0$, and
%%   set
%%   \[ r(\alpha) \coloneqq \begin{cases}
%%       n,& \alpha = (0,0,-n);\\
%%       n-n_0,& \alpha = (1,-\beta, -n). 
%%     \end{cases} \]
%%   Given a quiver $Q$, this suffices to define the family of stability
%%   conditions $(\tau^Q_{s,x})_{s, x \in \bR}$, the semistable loci
%%   $\fM^{\circ,Q,\sst}_{1,\beta,n,\vec d}(\tau^Q_{s,x})$ and
%%   $\fM^{\circ,Q,\sst}_{0,0,n,\vec d}(\tau^Q_{s,x})$, and their base
%%   changes
%%   \[ \fN_{\beta,n,\vec d}^Q(\tau^Q_{s,x}), \quad \fN_{(\lambda,\mu,\nu),n,\vec d}^Q(\tau^Q_{s,x}), \quad \fQ_n^Q(\tau^Q_{s,x}) \]
%%   along the obvious inclusions of $\pl$ stacks. Similarly, use a
%%   superscript $\st$ to denote stable loci.
%% \end{definition}

\subsection{Enumerative invariants}
\label{sec:enumerative-invariants}

\subsubsection{}

Recall that $\sT = (\bC^\times)^3$ act on $\bar X$ by scaling with
weights denoted $t_1, t_2, t_3 \in K_\sT(\pt)$. Let
\[ \kappa \coloneqq t_1t_2t_3. \]
The goal of this subsection is to establish the following definition
and explain why it produces DT and PT vertices for the corresponding
stability chambers.

\begin{definition}
  Fix integer partitions $\lambda, \mu, \nu$. For stability conditions
  $\tau$ on $\fN_{(\lambda,\mu,\nu),n}^\pl$ with no strictly
  semistables, let
  \[ \sN_{(\lambda,\mu,\nu),n}(\tau) \coloneqq \chi\left(\fN_{(\lambda,\mu,\nu),n}^{\sst}(\tau), \hat\cO^\vir\right) \in K_\sT(\pt)_\loc \]
  where $\cO^\vir$ is defined using the symmetric obstruction theory
  of Proposition~\ref{prop:obstruction-theory-rank-1}, and the
  symmetrized $\hat\cO^\vir$ is well-defined by
  Lemma~\ref{lem:obstruction-theory-square-root}. For the definition
  of (symmetric) obstruction theories on Artin stacks, see
  \S\ref{sec:obstruction-theory-definitions}.

  By Lemma~\ref{lem:obstruction-theory-rank-1-is-vertex}, the series
  $\sum_n Q^n \sN_{(\lambda,\mu,\nu),n}(\tau^\pm)$ are exactly the
  DT and PT vertices \eqref{eq:DT-and-PT-vertices} of interest.
\end{definition}

\subsubsection{}

\begin{proposition} \label{prop:obstruction-theory-rank-1}
  The moduli stack $\fN_{(\lambda,\mu,\nu),n}^\pl$ has an obstruction
  theory given by $(\bF[1])^\vee$, for
  \begin{equation} \label{eq:obstruction-theory-rank-1}
    \bF \coloneqq R\pi_*\cExt(\scI, \scI(-D)) \in D^b\cat{Coh}_\sT(\fN_{(\lambda,\mu,\nu),n}^\pl)
  \end{equation}
  where $\scI$ is the universal family on the source of
  $\pi\colon \fN_{(\lambda,\mu,\nu),n}^\pl \times \bar X \to \fN_{(\lambda,\mu,\nu),n}^\pl$.
\end{proposition}

To be precise, a universal family
$\scI \in D^b\cat{Coh}_\sT(\fM_{\rank=1} \times \bar X)$ certainly exists,
and it descends to a universal family on $\fM_{\rank=1}^\pl$ after a
twist by the weight-$1$ generator of the scaling $\bC^\times$, by
Lemma~\ref{lem:rigidification-of-dim-1-pairs}. Restriction to
$\fN_{(\lambda,\mu,\nu),n}^\pl$ produces the universal family $\scI$.

For any $I \in D^b\cat{Coh}(\bar X)$, relative Serre duality applied
to $R\pi_*\cExt$, along with the $\sT$-equivariant identification
$\cK_{\bar X} \cong \kappa \otimes \cO_{\bar X}(-2D)$, shows that
\[ \Ext_{\bar X}^*(I, I(-D))^\vee \cong \Ext_{\bar X}^{3-*}(I(-D), I \otimes \cK_{\bar X}) = \kappa \otimes \Ext_{\bar X}^{3-*}(I, I(-D)). \]
Hence $\bF$ defines a {\it symmetric} obstruction theory.

\subsubsection{}

\begin{proof}[Proof of Proposition~\ref{prop:obstruction-theory-rank-1}.]
  Recall that
  $\fN_{\beta_C,n} \subset \fM_{1,\beta_C,n}^\circ \subset \fM_{1,\beta_C,n}$
  are open, and $\fM_{1,\beta_C,n}$ is an open substack of a
  moduli stack of perfect complexes on $\bar{X}$ by \cite[\S
  6.3]{toda_dtpt}. By \cite[Theorem 1.4]{kuhn_ac_2023}, the Atiyah
  class of the universal complex gives an obstruction theory
  \[ R\pi_*\cExt(\scI, \scI)^{\vee}[-1] \to \bL_{\fN_{\beta,n}}. \]
  In the same way, the obstruction theory on $\fHilb^{\circ}(D)$ is
  given by the Atiyah class of the universal sheaf. Since the Atiyah
  class is compatible with pullbacks, the evaluation map induces a
  commutative diagram
  \begin{equation} \label{eq:obstruction-theory-rank-1-triangle}
    \begin{tikzcd}
      \ev^*R\pi_*\cExt(\scI|_D, \scI|_D)^{\vee}[-1] \ar{r}{\Xi} \ar{d} & R\pi_*\cExt(\scI, \scI)^{\vee}[-1] \ar{d} \ar[dotted]{r} & ?? \ar[dotted]{r}{[1]} \ar[dotted]{d} & {} \\
      \ev^*\bL_{\fHilb^{\circ}(D)} \ar{r} & \bL_{\fN_{\beta,n}} \ar{r} & \bL_{\fN_{\beta,n}/\fHilb(D)} \ar{r}{[1]} & {}.
    \end{tikzcd}
  \end{equation}
  By standard considerations, $(\bF[1])^\vee \coloneqq \cone(\Xi)$
  fills in $??$ and the dotted arrows and gives a relative obstruction
  theory for $\fN_{\beta,n}$ over $\fHilb^{\circ}(D)$, which pulls
  back to a relative obstruction theory for $\fN_{\beta,n}^{\pl}$ over
  $\Hilb^{\circ}(D)$. To determine the explicit formula
  \eqref{eq:obstruction-theory-rank-1} for $\bF$, we need to
  determine the co-cone for the restriction morphism
  \[ R\pi_*\cExt(\scI, \scI)\to R\pi_*\cExt(\scI\big|_D, \scI\big|_D). \]
  By adjunction, this map is identified with the result of applying
  $R\pi_*\cExt(\scI,-)$ to the co-unit of adjunction
  $\scI\to \iota_*\iota^*\scI$. Thus, the co-cone is precisely
  $R\pi_*\cExt(\scI, \scI(-D))$, as claimed.
\end{proof}

\subsubsection{}

\begin{remark}
  Alternatively, it should be possible to construct the obstruction
  theory \eqref{eq:obstruction-theory-rank-1} by directly
  twisting the construction \cite{Huybrechts2010} of the truncated
  Atiyah class by $\cO_{\bar X}(-D)$ to disallow deformations along
  $D$. More generally, we expect that the technology of
  \cite{Huybrechts1995,Huybrechts2010,Sala2012} can be combined to
  show that enforcing a framing isomorphism $L\iota^*I = F$ for a
  given $F$ changes the obstruction theory from
  $\Ext_{\bar X}(I, I)_0$ into the hyper-Ext group
  $\Ext_{\bar X}(I, I \to F)$. For us, in full agreement with
  Proposition~\ref{prop:obstruction-theory-rank-1},
  \[ \Ext_{\bar X}(I, I \to I|_D) \cong \Ext_{\bar X}(I, I(-D)). \]
  The presence of the framing $F$ removes the need for the
  traceless-ness condition.
\end{remark}

\subsubsection{}

The cohomology long exact sequence of the top row in
\eqref{eq:obstruction-theory-rank-1-triangle} begins as
\begin{equation} \label{eq:obstruction-theory-rank-1-LES}
  0 \to \Hom_{\bar X}(I, I(-D)) \to \Hom_{\bar X}(I, I) \to \Hom_D(I|_D, I|_D) \to \cdots.
\end{equation}
Clearly the first term $\Hom_{\bar X}(I, I(-D))$ is the group of
automorphisms of $I$ which are the identity on $I|_D$. When $I$ is a
stable object, the latter two terms here consist only of scalar
automorphisms and are identified with each other, so
$\Hom_{\bar X}(I, I(-D)) = 0$ as expected. Then Serre duality gives
$\Ext_{\bar X}^3(I, I(-D)) = 0$ as well. Hence
Proposition~\ref{prop:obstruction-theory-rank-1} induces a {\it
  perfect} obstruction theory on stable loci of
$\fN^\pl_{(\lambda,\mu,\nu),n}$.

\subsubsection{}

\begin{lemma} \label{lem:obstruction-theory-rank-1-is-vertex}
  Let $I \in \fN^{\sst}_{(\lambda,\mu,\nu),n}(\tau^\pm)$ be a
  $\sT$-fixed point. Then the virtual tangent space
  \[ T^{\vir}_I \fN^{\sst}_{(\lambda,\mu,\nu),n}(\tau^\pm) \]
  agrees with the vertex contribution $\mathsf{V}_\alpha$ of \cite[\S
  4.9]{Maulik2006} or \cite[\S 4.6]{Pandharipande2009a}.
\end{lemma}

\begin{proof}
  The following computation is independent of stability condition and
  applies to the whole stack. For $\bar X$, the original vertex
  contribution is by definition the Laurent polynomial $V$ in
  \begin{equation} \label{eq:Tvir-old-vertex-edge-redistribution}
    T^\vir_I\fN^\pl_{\beta,n} = V + \sum_{i=1}^3 \left(\frac{F_i}{1 - t_i} + \frac{F_i}{1 - t_i^{-1}}\right) = V + F_1 + F_2 + F_3
  \end{equation}
  where $F_i = T_{p_i}\Hilb(D_i) \in K_\sT(\pt)$; the terms in the
  brackets are edge contributions and correspond to deformations of
  $I|_D$. (Note that the equalities in
  \eqref{eq:Tvir-old-vertex-edge-redistribution} take place in
  $K_\sT(\pt)_{\loc}$ but clearly both the left and right hand sides
  lie in $K_\sT(\pt)$.) On the other hand, the top row in
  \eqref{eq:obstruction-theory-rank-1-triangle} gives
  \[ T^\vir_I\fN^\pl_{(\lambda,\mu,\nu),n} = T^\vir\fN^\pl_{\beta,n} - \sum_i T_{p_i}\Hilb(D_i) \]
  in K-theory, thereby removing the edge contributions. Alternatively,
  the same \v Cech cohomology computation which gives
  \eqref{eq:Tvir-old-vertex-edge-redistribution} also gives
  \begin{equation} \label{eq:Tvir-new-vertex-edge-redistribution}
    T^\vir_I\fN^\pl_{(\lambda,\mu,\nu),n} = V + \sum_{i=1}^3 \left(\frac{F_i}{1 - t_i} + \frac{t_i^{-1} F_i}{1 - t_i^{-1}}\right) = V
  \end{equation}
  where now the edge terms in the brackets are zero. Note that the
  twist by $D$ does not affect the $\bC^3$ chart where terms in $V$
  originate, so $V$ is unchanged between
  \eqref{eq:Tvir-old-vertex-edge-redistribution} and
  \eqref{eq:Tvir-new-vertex-edge-redistribution}.
\end{proof}

\subsubsection{}

\begin{lemma}[{\cite[\S 6]{Nekrasov2016}}] \label{lem:obstruction-theory-square-root}
  $\det \bF$ admits a square root, after a base change
  $\kappa \leadsto \kappa^{1/2}$.
\end{lemma}

Throughout, we implicitly take a double cover of $\sT$ so that
$\kappa^{1/2}$ exists. The Nekrasov--Okounkov symmetrization
$\hat\cO^\vir \coloneqq \cO^\vir \otimes (\det \bF)^{-1/2}$ is
therefore well-defined on any locus where $\cO^\vir$ can be
constructed.

\begin{proof}[Proof sketch.]
  On
  $\fN^\pl_{(\lambda,\mu,\nu),n} \times \fN^\pl_{(\lambda,\mu,\nu),n}
  \times \bar X$, let $\scI_1$ and $\scI_2$ be universal
  families for the first and second factors, and consider
  \[ \bL \coloneqq \det R\pi_*\cExt(\scI_1, \scI_2(-D)) \]
  where $\pi$ is projection onto the first two factors. The desired
  $\det \bF$ is the restriction of $\bL$ to the diagonal $\Delta$. By
  the argument in \cite[Lemma 6.1]{Nekrasov2016}, it suffices to show
  that $c_1(\bL\big|_\Delta)$ is divisible by two. This follows from
  \[ (12)^*\bL = \det R\pi_*\cExt(\scI_2, \scI_1(-D)) = \kappa^{\rank} R\pi_*\cExt(\scI_1, \scI_2(-D)) = \bL, \]
  where $(12)$ swaps the two copies of
  $\fN^\pl_{(\lambda,\mu,\nu),n}$, and $\rank = 0$ is the rank of
  $R\pi_*\cExt(\scI_2, \scI_1(-D))$.
\end{proof}

\subsubsection{}
\label{sec:moduli-stack-of-dim-0-sheaves}

For completeness, we mention here that the moduli stacks $\fQ_n$
are well-known to carry a symmetric obstruction theory given by
$(\bF[1])^\vee$ where
\begin{equation} \label{eq:obstruction-theory-rank-0}
  \bF \coloneqq R\pi_*\cExt(\scE, \scE) \in D^b\cat{Coh}_\sT(\fQ_n).
\end{equation}
Here, since Lemma~\ref{lem:rigidification-of-dim-1-pairs} is
unavailable, the universal family $\scE$ on $\fQ_n$ does not descend
to $\fQ_n^\pl$, but $\cExt(\scE, \scE)$ has trivial
$\bC^\times$-weight and therefore does descend. The symmetric
obstruction theory on $\fQ_n^\pl$ which is compatible with the one on
$\fQ_n$, in the sense of
Definition~\ref{def:smooth-and-symmetrized-pullback}, is therefore
given by $(\bF_{\perp}[1])^\vee$ where
\begin{equation} \label{eq:obstruction-theory-rank-0-pl}
  \bF_\perp \coloneqq R\pi_*\cExt(\scE, \scE)_\perp \in D^b\cat{Coh}_\sT(\fQ_n^\pl)
\end{equation}
is the ``symmetrically traceless'' version of $\bF$ in the sense of
\cite[Theorem 6.1]{Tanaka2020}. It is obtained from $\bF$ by removing a copy of $\bC$ in
$H^0$ and, by Serre duality, a copy of $\bC \otimes \kappa$ in $H^3$.

Viewing objects in $\fQ_n$ or $\fQ_n^\pl$ as pairs $[0 \to \cE]$, and
recalling that they have zero intersection with $D \subset \bar X$,
\eqref{eq:obstruction-theory-rank-0} is exactly the same formula as
\eqref{eq:obstruction-theory-rank-1}. For instance,
Lemma~\ref{lem:obstruction-theory-square-root} applies equally well
for $\fQ_n^\pl$.

Note that, unlike $\fN^{\sst}_{(\lambda,\mu,\nu),m}(\tau^\pm)$, the
moduli stacks $\fQ^{\sst}_m(\tau_\xi)$ have strictly semistable
objects for any $\xi$ and any $m > 1$, and so there is no direct
analogue of the enumerative invariants
$\fN_{(\lambda,\mu,\nu),m}(\tau^\pm)$. One would have to follow the
strategy of \cite[Theorem 5.7]{joyce_wc_2021} \cite[\S 6]{Liu23} in
order to define {\it semistable invariants} $\sQ_m(\tau_\xi)$
intrinsic to the stack $\fQ^{\sst}_m(\tau_\xi)$.

\section{Ingredients for wall-crossing}

\subsection{Quiver-framed stacks}
\label{sec:quiver-framed-stacks}

\subsubsection{}

We first build auxiliary moduli stacks following the general strategy
outlined in \S\ref{sec:general-framing-argument}, for our moduli
stacks $\fN_{(\lambda,\mu,\nu),n}$ of pairs and $\fQ_n$ of
zero-dimensional sheaves as described in
\S\ref{sec:def_moduli_subspaces}. This involves the following framing
functors.

\begin{definition} \label{def:framing-functor}
  Let $\cO_{\bar X}(1)$ denote $\cO_{\bP^1}(1)^{\boxtimes 3}$, or any
  other ample line bundle on $\bar X$. For $k \in \bZ_{> 0}$, let
  ${}_k \fN_{(\lambda,\mu,\nu),n} \subset \fN_{(\lambda,\mu,\nu),n}$
  be the open locus of pairs $[\cO_{\bar X} \otimes L \to \cE]$ where
  $\cE$ is $k$-regular \cite[\S 1.7]{Huybrechts2010}, i.e. such that
  $H^i(\cE(k-i)) = 0$ for all $i > 0$. For $p \in \bZ_{> 0}$, define
  the functor
  \[ F_{k,p}\colon [\cO_{\bar X} \otimes L \to \cE] \mapsto \chi(\cE(k)) \oplus L^{\oplus p} \]
  which sends $\bC$-valued objects of ${}_k \fN_{(\lambda,\mu,\nu),n} \subset \fN_{(\lambda,\mu,\nu),n}$ to vector spaces.
  The sheaves $\cE$ parameterized by $\fQ_n$ are zero-dimensional and
  therefore automatically $k$-regular; we extend $F_{k,p}$ to objects
  in $\fQ_n$ as
  \[ F_{k,p}\colon [0 \to \cE] \mapsto \chi(\cE(k)) = \chi(\cE), \]
  which is independent of both $k$ and $p$. Then $F_{k,p}$ becomes a
  functor which is exact on short exact sequences involving only
  objects in $\fQ$ and ${}_k \fN_{(\lambda,\mu,\nu)}$. Set $f_{k,p}(I)
  \coloneqq \dim F_{k,p}(I)$ and note that it only depends on the
  numerical class $\cl I$.

  The term $L^{\oplus p}$ ensures that the induced map $\Hom(-, -) \to
  \Hom(F_{k,p}(-), F_{k,p}(-))$ is an inclusion on ${}_k
  \fN_{(\lambda,\mu,\nu)}$, especially when $\lambda = \mu = \nu =
  \emptyset$. For most of the paper, we fix $p = 1$, but in
  \S\ref{sec:joyce-WCF-generating-functions} we take advantage of the
  freedom in choosing $p$ in order to get a nice factorization of some
  generating series.
\end{definition}

\subsubsection{}

\begin{definition} \label{def:auxiliary-stacks}
  Let $(\lambda, \mu, \nu)$ be integer partitions, which specify
  $\beta_C = (\abs{\lambda}, \abs{\mu}, \abs{\nu})$, and let $\alpha =
  (1, -\beta_C, -n)$ be a numerical class. In the wall-crossing
  arguments of \S\ref{sec:mochizuki-WCF} and \S\ref{sec:joyce-WCF}, we
  always begin by fixing $(\lambda, \mu, \nu)$ and $n$, which
  determines the class $\alpha$; choices of parameters below will
  depend on $\alpha$. Choose $k \gg 0$ such that ${}_k\fN_{\beta_C,m}
  = \fN_{\beta_C,m}$ for all $m \le n$. Choose $p > 0$ and set
  \[ N \coloneqq f_{k,p}(\alpha). \]
  For $m \le n$ and a dimension vector $\vec d = (d_1, \ldots, d_N)$,
  let $\fN_{(\lambda,\mu,\nu),m,\vec d}^{Q(N)}$ be the moduli stack of
  triples $(I, \vec V, \vec \rho)$ where:
  \begin{itemize}
  \item $I$ is an object in ${}_k\fN_{(\lambda,\mu,\nu),m} =
    \fN_{(\lambda,\mu,\nu),m}$;
  \item $\vec V = (V_i)_{i=1}^N$ are vector spaces with
    $\dim V_i = d_i$; set $V_{N+1} \coloneqq F_{k,p}(I)$; 
  \item $\vec \rho = (\rho_i)_{i=1}^N$ are linear maps $\rho_i\colon
    V_i \to V_{i+1}$.
  \end{itemize}
  Similarly define $\fQ_{m,\vec d}^{Q(N)}$ using objects in
  ${}_k \fQ_{m,\vec d}$. In other words, these moduli stacks
  parameterize objects of the underlying moduli stack
  $\fN_{(\lambda,\mu,\nu),m}$ or $\fQ_m$ along with quiver
  representations
  \[ \begin{tikzcd}
      \overset{V_1}{\blacksquare} \ar{r} & \overset{V_2}{\blacksquare} \ar{r} & \cdots \ar{r} & \overset{V_{N-1}}{\blacksquare} \ar{r} & \overset{V_N}{\blacksquare} \ar{r} & \overset{V_{N+1} = F_{k,p}(I)}{\bigbullet}.
    \end{tikzcd} \]
  Denote this quiver by $Q(N)$. Consequently, the forgetful maps
  \begin{equation} \label{eq:framed-stack-forgetful-maps}
    \Pi_\fN\colon \fN_{(\lambda,\mu,\nu),m,\vec d}^{Q(N)} \to \fN_{(\lambda,\mu,\nu),m}, \qquad \Pi_\fQ\colon \fQ_{m,\vec d}^{Q(N)} \to \fQ_m
  \end{equation}
  are smooth with fiber $[\bigoplus_{i=1}^N \Hom(V_i, V_{i+1}) /
    \prod_{i=1}^N \GL(V_i)]$. Write $(\beta, \vec d)$ for the
  numerical class of $(I, \vec V, \vec \rho)$ where $\beta = \cl(I)$
  and $\vec d \coloneqq \dim \vec V$.
\end{definition}

\subsubsection{}
\label{sec:quiver-spaces-compactified}
The same construction works to give relative quiver spaces $\Pi_{\fM}:
\fM_{1,\beta_C,m,\vec d}^{\circ,Q(N)}\to \, \fM_{1,\beta_C,m}^{\circ}$
and $\Pi_{\fM}: \fM_{0,0,m,\vec d}^{\circ,Q(N)}\to \,
\fM_{0,0,m}^{\circ}$. The maps $\Pi_{\fM}$ base change to $\Pi_{\fN}$
over $\fN_{(\lambda,\mu,\nu), m}\subseteq \fM_{1,\beta_C,m}^{\circ}$
to $\Pi_{\fQ}$ over $\fQ_{m}\subseteq \fM_{0,0,m}$ respectively.

\subsubsection{}
\label{sec:framing-quiver-notation}

Let $\scI$ and $(\scV_i)_{i=1}^{N+1}$ denote the universal families
associated to $I$ and $(V_i)_{i=1}^{N+1}$ in \ref{def:auxiliary-stacks} respectively, on both
$\fN_{(\lambda,\mu,\nu),n,\vec d}^{Q(N)}$ and $\fQ_{n,\vec d}^{Q(N)}$.
The relative cotangent complexes $\bL_{\Pi_\fN}$ and
$\bL_{\Pi_\fQ}$ are both given by the restriction
to the diagonal of $\bF^\vee$, where
\begin{equation} \label{eq:quiver-bilinear-obstruction}
  \bF \coloneqq \bigg[\bigoplus_{i=1}^N \scV_i^\vee \boxtimes \scV_i \to \bigoplus_{i=1}^N \scV_i^\vee \boxtimes \scV_{i+1}\bigg]
\end{equation}
with second term in degree $0$. The morphism is the derivative of the
natural $\prod_{i=1}^N \GL(V_i)$ action. Note that the universal
bundles $\scV_i$ may not descend to
$\fN^{Q(N),\pl}_{(\lambda,\mu,\nu),n,\vec d}$ or
$\fQ^{Q(N),\pl}_{n,\vec d}$, but bilinear quantities like $\bF$ do.
For later use, denote the rank of $\bF$ by
\[ \chi_Q\left(\vec e, \vec f\right) \coloneqq \sum_{i=1}^N (e_i f_{i+1} - e_i f_i), \]
where $f_{N+1} \coloneqq \rank \scV_{N+1}$. Similarly let
$\chi(\gamma, \delta) \coloneqq \rank \Ext_{\bar X}(I_\gamma,
I_\delta(-D))$ be the analogous Euler pairing on the original moduli
stacks $\fN$ and $\fQ$, where $I_\gamma$ and $I_\delta$ denote any
element $I$ of numerical class $\gamma$ and $\delta$ respectively.
Later in wall-crossing formulas,
\[ c_Q(\vec e,\vec f) \coloneqq \chi_Q(\vec e,\vec f) - \chi_Q(\vec f,\vec e) = \sum_{i=1}^N (e_i f_{i+1} - f_i e_{i+1}) \]
will appear as part of the quantity $\ind$ in
Proposition~\ref{prop:master-space-relation} for the complicated
$\bC^\times$-fixed locus in the master space.

\subsubsection{}

\begin{remark} \label{rem:quiver-framing-isomorphisms}
  The length of the quiver $Q(N)$ may actually be chosen to be greater
  than $N$: there is obviously an isomorphism of stacks (with
  obstruction theory)
  \[ \fQ^{Q(N)}_{m,\vec d} \cong \fQ^{Q(N+1)}_{m,(0,\vec d)}. \]
  Similarly but less obviously, if $d_a = d_{a+1}$ for some $a$, then
  there is an isomorphism of stacks (with obstruction theory)
  \[ \fQ^{Q(N)}_{m,\vec d} \supset \{\rho_a\colon V_a \xrightarrow{\sim} V_{a+1}\} \cong \fQ^{Q(N-1)}_{m, (d_1, \ldots, d_a, d_{a+2}, \ldots, d_N)} \]
  for the open substack where $\rho_a$ is an isomorphism. One can
  check that the isomorphism $\rho_a$ splits off a two-term complex
  $[\scV_a^\vee \boxtimes \scV_a^\vee \to \scV_a^\vee \boxtimes \scV_{a+1}^\vee]$
  in \eqref{eq:quiver-bilinear-obstruction}, and this complex is
  quasi-isomorphic to $0$. 
\end{remark}

\subsection{Full flags and words}

\subsubsection{}

\begin{definition} \label{def:full-flag}
  A class $(\beta, \vec d)$ is a {\it full flag} if:
  \begin{itemize}
  \item $\beta \neq 0$ and $\vec d \neq \vec 0$;
  \item $d_1 \le 1$, and $d_N = d_{N+1} \coloneqq f_{k,p}(\beta)$, and
    $d_a \leq d_{a+1}\leq d_a +1$ for all $1 \leq a < N$.
  \end{itemize}
  Given a quiver-framed stack $\fX^{Q(N)}_{\beta,\vec d}$, let
  \[ \fX^{\fl}_{\beta, \vec d} \subset \fX^{Q(N), \pl}_{\beta,\vec d}\]
  be the open substack for which all maps in the quiver representation
  are injective. As in Remark~\ref{rem:quiver-framing-isomorphisms},
  we have a canonical isomorphism
  \[ \fX^{\fl}_{\beta, \vec d} \simeq \fX^{\fl}_{\beta, (1, 2, \ldots, d_N)}. \]
  
  In both wall-crossing approaches, the only relevant classes which
  appear are full flags with $\beta = (1, -\beta_C, -m)$ and $\beta =
  (0, 0, -m)$. At walls, a full flag splits into two smaller full
  flags. We view objects of $\fN_{(\lambda,\mu,\nu),m,\vec d}^{\fl}$
  or $\fQ_{m,\vec d}^{\fl}$ as objects $I$ in
  $\fN_{(\lambda,\mu,\nu),m}$ or $\fQ_m$ together with a full flag of
  sub-spaces in $F_{k,p}(I)$.
\end{definition}

\subsubsection{}

\begin{proposition} \label{prop:projective-bundle-pushforward}
  Let $\pi\colon \fX \to \fY$ be a $\bP^N$-bundle. Suppose both $\fX$ and
  $\fY$ have symmetric APOTs, related by symmetrized pullback. Then the
  associated virtual cycles satisfy
  \[ \chi(\hat\cO^\vir_{\fX}) = [N+1]_\kappa \cdot \chi(\hat\cO^\vir_{\fY}). \]
\end{proposition}

In particular, let $(\beta, \vec d)$ be a full flag, and take
$\pi\colon \fX^{\fl}_{\beta,\vec d} \to \fX_\beta$ to be the forgetful
map. Supposing that it lands in a stable locus $\fX_\beta^\st \subset
\fX_\beta$, since full flag bundles are iterated projective bundles,
\begin{equation} \label{eq:full-flag-bundle-pushforward}
  \chi(\fX^{\fl}_{\beta,\vec d}, \hat\cO^\vir) = [d_N]_\kappa! \cdot \chi(\fX^{\st}_\beta, \hat\cO^\vir).
\end{equation}

\begin{proof}
  Apply Proposition~\ref{prop:symmetrized-pullback-APOTs}. It remains
  to compute the pushforward
  \[ \pi_*(\se(\Omega_\pi \otimes \kappa^{-1})) \otimes \kappa^{-\frac{\dim \pi}{2}} K_\pi \]
  where $K_\pi \coloneqq \det \Omega_\pi$ is the relative canonical.
  This is a standard computation \cite[Prop. 5.5]{Thomas2020}
  \cite[Prop. 4.6]{Liu23} which we briefly review. The pushforward is
  some combination of relative cohomology bundles
  $R^i\pi_*\Omega_\pi^j$, which are all canonically
  trivialized by powers of the hyperplane class. So it is enough to
  compute on fibers:
  \[ \chi\left(\bP^N, \se\left(\Omega_{\bP^N} \otimes \kappa^{-1}\right) \otimes K_{\bP^N}\right) \kappa^{-\frac{N}{2}} = (-1)^N \kappa^{-\frac{N}{2}} \sum_{i,j=0}^N (-1)^{i+j} \dim H^{i,j}(\bP^N) t^j = [N+1]_t \]
  using that the modules $H^{i,j}(\bP^N) \subset H^{i+j}(\cO)$ are
  $\sT$-equivariantly trivial by Hodge theory. This computation
  motivates the sign $(-1)^{N-1}$ in the definition
  \eqref{eq:quantum-integer} of $[N]_t$.
\end{proof}

\subsubsection{}
\label{sec:on-wall-invariants-Q}

Let $\vec 1$ be the framing dimension vector $(1, 1, \ldots, 1)$, and
\[ \fQ^{\fl,\st}_{m, \vec 1} \subset \fQ^{\fl}_{m, \vec 1} \]
be the open substack of triples $([0 \to \cE], \vec V, \vec \rho)$
where the composition $V_1 \to \cdots \to V_N \to F_{k,p}(\cE)
= H^0(\cE)$ induces a surjection $V_1 \otimes \cO_{\bar X}
\twoheadrightarrow \cE$. We call such triples {\it stable}.

\begin{proposition} \label{prop:on-wall-invariants-Q}
  There is an isomorphism
  \[ \fQ^{\fl,\st}_{m, \vec 1} \cong \Hilb(\bC^3, m) \]
  which identifies the (symmetrized) virtual structure sheaves.
\end{proposition}

Here, we use Theorem~\ref{thm:symmetric-pullback} to equip the left
hand side with a symmetric APOT by symmetrized pullback along the
forgetful map to $\fQ_m$, which is clearly smooth. We view
$\Hilb(\bC^3, m)$ as the moduli scheme of pairs $[\cO_{\bar X}
  \twoheadrightarrow \cO_Z]$ where $Z \subset \bC^3 = \bar X \setminus
D$ is a zero-dimensional subscheme of length $m$ and the map is the
natural surjection.

More generally, if $((0, 0, -m), \vec e)$ is a full flag, then it is
clear that $\fQ_{m,\vec{e}}^{\fl,\st}$ is naturally isomorphic to the full
flag bundle over $\Hilb(\bC^3, m)$ associated to the vector bundle
$H^0(\cE)/\cO_{\Hilb(\bC^3, m)}$, where the quotient is taken with
respect to the universal section. Consequently
\[ \chi(\fQ_{m,\vec{e}}^{\fl,\st}, \hat\cO^\vir) = [m-1]_\kappa! \cdot \chi(\Hilb(\bC^3, m), \hat\cO^\vir). \]

\subsubsection{}

\begin{proof}[Proof of Proposition~\ref{prop:on-wall-invariants-Q}.]
  It is clear from the moduli description of $\fQ^{\fl,\st}_{m, \vec
    1}$ that the desired isomorphism exists, namely, if $N > 1$ then
  use Remark~\ref{rem:quiver-framing-isomorphisms} to shorten the
  chain of isomorphisms $V_1 \xrightarrow{\sim} \cdots
  \xrightarrow{\sim} V_N \to H^0(\cE)$ to the single map $V_1 \to
  H^0(\cE)$. To match virtual cycles, recall that virtual tangent
  space of $\Hilb(\bC^3)$, i.e. the K-theory class of the dual of the
  symmetric obstruction theory, is
  \begin{align*}
    &\Ext_{\bar X}(\cO_{\bar X}, \cO_{\bar X}) - \Ext_{\bar X}([\cO_{\bar X} \twoheadrightarrow \cO_Z], [\cO_{\bar X} \twoheadrightarrow \cO_Z]) \\
    &= \Ext_{\bar X}(\cO_{\bar X}, \cO_Z) + \Ext_{\bar X}(\cO_Z, \cO_{\bar X}) - \Ext_{\bar X}(\cO_Z, \cO_Z)
  \end{align*}
  where $\Ext_{\bar X}(-, -) \coloneqq \sum_i (-1)^i \Ext^i_{\bar X}(-, -)$. On
  the other hand, the APOT on stable loci of $\fQ^{Q(N),\pl}_{m,\vec
    1}$ comes from symmetrized pullback, with virtual tangent space
  \[ -\Ext_{\bar X}(\cO_Z, \cO_Z)_\perp + \left(\Hom_{\bar X}(\cO_{\bar X}, \cO_Z) - \bC\right) - \left(\Hom_{\bar X}(\cO_{\bar X}, \cO_Z) - \bC \right)^\vee \kappa^{-1} \]
  coming from \eqref{eq:obstruction-theory-rank-0-pl} and
  \eqref{eq:quiver-bilinear-obstruction}. The $\bC = \Ext^0_{\bar
    X}(\cO_{\bar X}, \cO_{\bar X})$ and $\bC^\vee \otimes \kappa^{-1}$
  terms remove the $\perp$ from the first term; see
  \S\ref{sec:moduli-stack-of-dim-0-sheaves}. Since $Z$ is affine,
  $\Hom_{\bar X}(\cO_{\bar X}, \cO_Z) = \Ext_{\bar X}(\cO_{\bar X},
  \cO_Z)$. So the two virtual tangent spaces match in K-theory. By
  Proposition~\ref{prop:master_space_vir_class_comparison} (with $X =
  \bM$ and a trivial $\bC^\times$-action), the virtual cycles match
  too.
\end{proof}

\subsubsection{}%quiver notation

For later use, we collect here some notation, which make handling the
combinatorial expressions in the upcoming wall-crossing formulas of
\S\ref{sec:mochizuki-WCF} and \S\ref{sec:joyce-WCF} a bit simpler.
By iterated wall-crossing, we obtain decompositions of the full flag dimension
vector $(1, 2, \ldots, N)$, of the form
\begin{equation} \label{eq:decomposition-of-full-flag}
  (1, 2, \ldots, N) = \sum_{i=1}^k \vec e_i
\end{equation}
where $e_{i,a} \le e_{i,a+1} \le e_{i,a}+1$ for all $i$ and
$1 \le a < N$. This makes contact with the combinatorial subject of
{\it word rearrangements}.

\begin{definition} \label{def:word-rearrangements}
  For $\ell \ge 1$ and $\vec m \in \bZ_{> 0}^\ell$ such that
  $N = |\vec m| \coloneqq \sum_{i=1}^\ell m_i$, let
  \[ R(\vec m) \coloneqq \{\text{rearrangement of the word } 1^{m_1} 2^{m_2} \cdots \ell^{m_\ell}\}. \]
  Given a word $w \in R(\vec m)$, let
  $O_i(w) \subset \{1, 2, \ldots, N\}$ be the set of indices where the
  letter $i$ occurs in $w$, and let $o_i(w) \coloneqq \min O_i(w)$ be
  the index of the first occurrence. There is a bijection between
  $R(\vec m)$ and decompositions \eqref{eq:decomposition-of-full-flag}
  such that $m_i = e_{i,N}$, where
  \[ e_{i,a} \coloneqq \#\{b \in O_i(w) : b \le a\}. \]
  Under this bijection,
  \begin{equation} \label{eq:word-symmetrized-inversion-number}
    \begin{aligned}
      c_{i,j}(w) \coloneqq c_Q(\vec e_i, \vec e_j)
      = &\;\#\left\{(a, b) \in O_i(w) \times O_j(w) : a < b\right\} \\
      &- \#\left\{(a, b) \in O_i(w) \times O_j(w) : a > b\right\}.
    \end{aligned}
  \end{equation}
  The equality is easy to check: it holds for
  $w = 1^{m_1} 2^{m_2} \cdots \ell^{m_\ell}$, and then each inversion
  reduces both sides by $2$. So $\sum_{j>i} c_Q(i, j)$ is related to
  the inversion number for $i$. It can be used to quantize the
  combinatorial identity $|R(\vec m)| = \binom{N}{\vec m}$, see
  Remark~\ref{rem:q-multinomial}.
\end{definition}

\subsection{Properness of moduli spaces}
\label{sec:properness}
We describe a general strategy for proving the valuative criteria for properness and separatedness for moduli spaces of (semi-) stable objects originally due to Langton \cite{Langton1975} and nicely explained in \cite{huybrechts_lehn_geom}. This will be applied in \S \ref{sec:mochizuki-WCF} and \S \ref{sec:joyce-WCF}. 

Throughout, we let $R$ denote an arbitrary DVR over $\bC$ with generic point $\eta$, closed point $\xi$, fraction field $K$ and uniformizer $\pi$. 
\subsubsection{}
We consider the following general (although slightly imprecise) set-up.

	Let $\fM_{\alpha}$ be a moduli stack of objects in an abelian category $\cat{A}$ with a given numerical class $\alpha$.
	We let $\tau$ be a weak stability condition on $\cat{A}$, and let $\fM_{\alpha}^{\st}(\tau)\subseteq \fM_{\alpha}^{\sst}(\tau)\subseteq \fM_{\alpha}^{\pl}$ be the moduli stacks of stable and semi-stable objects. We assume that $\cat{A}$ is noetherian and  \emph{$\tau$-artinian} \cite[3.1]{joyce_wc_2021} and that $\fM_{\alpha}^{\sst}$ is of finite type. 
\begin{remark}
For us, we will take the spaces $\fM^{\circ,Q}_{\alpha,\vec{d}}$ parametrizing pairs with quiver representations of \ref{sec:quiver-spaces-compactified}, as well as some variants of them. 
\end{remark}
\subsubsection{}
Quite generally, one has the following consequences of the properties of stable objects:
\begin{lemma}\label{lem:separated-algebraic-space}
  The stack $\fM_{\alpha}^{\st}(\tau)$ is a separated algebraic space. 
\end{lemma}
\begin{proof}
The automorphism group of a stable object is given by the scaling automorphisms, which are removed by rigidification. To see separatedness, let $\mathcal{E}_R$ and $\mathcal{E}_R'$ be two families of stable objects over the DVR $R$ and $f:\mathcal{E}_{\eta}\to \mathcal{E}'_{\eta}$ an isomorphism over the generic point. Then there exists a unique integer $k$ such that $\pi^k f$ extends to an morphism over $\Spec R$ whose restriction to the special fiber is non-zero.  In particular, $\pi^k f$ induces a non-zero morphism $\mathcal{E}_{\xi}\to \mathcal{E}'_{\xi}$ of stable objects over the closed point of $R$. Since any morphism of stable objects in the same numerical class is an isomorphism, this shows that $\pi^k f$ extends to an isomorphism of the families $\mathcal{E}_{R}$. 

Now suppose we are given two $R$-valued points $x_R,y_R$ of $\fM^{\st}_{\alpha}(\tau)$, and an isomorphism $[f]: x_{\eta} \to y_{\eta}$. We want to check that $[f]$ extends uniquely to an isomorphism $x_R\to y_R$. After possibly passing to an \'etale cover of $R$, we may assume that $x_R$ and $y_R$ are represented by genuine families of stable objects $\mathcal{E}_R$ and $\mathcal{E}'_R$, and that $[f]$ is represented by a generic isomorphism $f:\mathcal{E}_{\eta}\to \mathcal{E}'_{\eta}$. 
By what we said above, there is a unique $k$ for which $\pi^k f$ extends to an isomorphism $\mathcal{E}_R \to \mathcal{E}'_R$ (again uniquely). Passing to the projective linear stack exactly has the effect of identifying the generic isomorphisms $f$ and $\pi^k f$. Thus $[\pi^k f]$ gives the desired extension of $[f]$ to an isomorphism over $R$.   
\end{proof}

\subsubsection{}
We recall the notion of elementary modification originated in \cite{Langton1975}. 
We will use the following definition, which is somewhat imprecise in general, but makes sense when working with some abelian category of quasi-coherent sheaves or complexes with extra structure.

Let $\cat{A}$ be an abelian category of sheaves or complexes, possibly with extra structure. Suppose we are given a family of objects $A_R$ in some over the DVR $R$. Let $j:\xi \to \Spec R$ denote the inclusion of the closed point. 
\begin{definition}
	Let $A_{\xi}$ denote the restriction of $A_R$ to the closed fiber. Let $0\to B \to A_{\xi} \to C \to 0$ be an exact sequence in $\mathcal{A}$, defined either by the subobject $B$ or the quotient $C$ of $A_{\xi}$. Then we call the family of objects $A'_R$ obtained as the co-cone of the composition
	\[A_{R}\to j_*A_{\xi}\to j_*C\]
	the \emph{elementary modification of $A_R$ along $B$ or $C$} respectively.
\end{definition}

\subsubsection{}
\label{sec:properness-strategy}
We describe a strategy to check the existence part of the valuative criterion for properness for $\fM_{\alpha}^{\sst}$, and we record under which assumptions this strategy works. Let $\mathcal{E}_K$ be a semi-stable object of $\fM_{\alpha}$ over $\Spec K$. We make the first assumption

\begin{itemize}
\item[(A0)] There is an extension of $\mathcal{E}_K$ to a family $\mathcal{E}_R^{(0)}$ in $\fM_{\alpha}$, possibly after a base change on $R$ and possibly with additionally desirable properties.  
\end{itemize}

Now, suppose one has a family $\mathcal{E}_R^{(n)}$ extending $\mathcal{E}_K$ over $R$. If the central fiber $\mathcal{E}^{(n)}_{\xi}$ is semi-stable, we are finished. If not, by \cite[Theorem 3.3]{joyce_wc_2021} there exists a Harder-Narasimhan filtration of $\mathcal{E}^{(n)}_{\xi}$:
\[0 = F_0 \subset F_1 \subset \cdots \subset F_{L-1} \subset F_L = \mathcal{E}^{(n)}_{\xi},\]
such that the subquotients $\bar{F}_i := F_{i}/F_{i-1}$ are semi-stable for $1\leq i\leq L$ with $\tau(\bar{F}_1) > \cdots >\tau(\bar{F}_L)$. Let $\tau_{\max}^{n}:=\tau_{\max}(\mathcal{E}^{(n)}_{\xi}):=\tau(\bar{F}_1)$ and $\tau_{\min}^{n}:=\tau_{\min}(\mathcal{E}^{(n)}_{\xi}) := \tau(\bar{F}_L)$. Then $\tau_{\max}^{n} \geq \tau(\alpha) = \tau(\mathcal{E}^{(n)}_{\xi}) \geq \tau_{\min}^{n}$, and at least one of the inequalities is strict. 

If $\tau_{\max}^{n} > \tau(\alpha)$, let $G^{(n)} \coloneqq \mathcal{E}_\xi^{(n)}/F_1$. Otherwise, let $G^{(n)} \coloneqq \mathcal{E}_\xi^{(n)}/F_{L-1}$. We let $\mathcal{E}^{(n+1)}_R$ be the elementary modification along $G^{(n)}$. In other words, this is either the elementary modification along the maximal destabilizing sub-object of $\mathcal{E}^{(n)}_{\xi}$, or along a maximally destabilizing quotient. 

This inductively defines a sequence of objects $\mathcal{E}_R^{(n)}$ with restrictions $\mathcal{E}_{\xi}^{(n)}$. We also have a natural map $\varphi^{(n)}: G^{(n)}\to G^{(n+1)}$ given by the composition $G^{(n)}\to \mathcal{E}_{\xi}^{(n+1)}\to G^{(n+1)}$. One checks the following 
\begin{lemma}
	Unless $\mathcal{E}^{(n)}_{\xi}$ is already semi-stable, we have: $\tau_{\max}^{n}\geq \tau_{\max}^{n+1}$ and $\tau_{\min}^n\leq \tau_{\min}^{n+1}$. 
\end{lemma} 

It follows that we are in exactly one of the following situations:
\begin{enumerate}[label = \arabic*)]
\item We have $[\tau_{\min}^{n+1},\tau_{\max}^{n+1} ] \subsetneq [\tau_{\min}^n,\tau_{\max}^n]$
\item We have $[\tau_{\min}^{n+1},\tau_{\max}^{n+1} ]  =  [\tau_{\min}^n,\tau_{\max}^n ]$ and $\varphi^{(n)}: G^{(n)}\to G^{(n+1)}$ \emph{is not} an isomorphism, i.e. it has a non-trivial kernel or co-kernel.
\item We have $[\tau_{\min}^{n+1},\tau_{\max}^{n+1} ]  =  [\tau_{\min}^n,\tau_{\max}^n ]$ and $\varphi^{(n)}: G^{(n)}\to G^{(n+1)}$ \emph{is} an isomorphism.
\end{enumerate}

Consequently, one needs to check the following assumptions for the existence part of the valuative criterion to hold:
\begin{itemize}
	\item[(A1)] There is no infinite sequence of consecutive steps in which situation 3) occurs.
	\item[(A2)] There is no infinite sequence of consecutive steps in which either only situation 2) or 3) occur, and in which 2) occurs infinitely many times.
	\item[(A3)] The ascending sequence of values $\tau_{\min}^n$ has to stabilize. So does the  descending sequence $\tau_{\max}^n$.   
\end{itemize}
\subsubsection{}
\begin{remark}\label{rem:properness-assumptions}
	\begin{enumerate}
	\item In practice, assumption (A3) is often automatically satisfied, e.g. if the set of possible values that $\tau$ takes in the interval $[\tau_{\min}^{(0)}, \tau_{\max}^0] = [\tau_{\min}(\mathcal{E}_{\xi}),\tau_{\max}(\mathcal{E}_{\xi})]$ when evaluated on classes which appear in effective decompositions of $\alpha$ is a-priori finite.  
	\item Similarly, assumption (A1) is often automatic when working with a moduli space parametrizing configurations of coherent sheaves as we do here. The reason is (exactly as in \cite[Theorem 2.B.1]{Huybrechts1995}) that an infinite sequence of steps 3) gives rise to a section of a relative Quot-scheme in a formal neighborhood of $\xi$, which can be algebraized, contradicting the semi-stability of the generic fiber. 
	\item If $\tau$ is a (non-weak) stability condition, then (A2) holds. This is not clear if one has only a weak stability condition, and requires an additional argument.  
	\end{enumerate} 
\end{remark}

\subsubsection{}

\begin{remark} \label{rem:fixed-loci-are-components-in-proper}
  For us, we only want to show properness of $\sT$-fixed loci of
  moduli spaces $\fN_{\alpha}^{\sst}$ when there are no properly
  semi-stable objects. We also have an open embedding
  $\fN_{\alpha}^{\sst}\subseteq \fM_{\alpha}^{\sst}$ (cf.
  \ref{sec:def-open}). In our situation, it is clear that when
  semi-stability equals stability for $\alpha$ (so that $\fM_{\alpha}$
  is a separated algebraic space), then the inclusion on $\sT$-fixed
  loci $\left(\fN_{\alpha}^{\sst}\right)^\sT\subseteq
  \left(\fM_{\alpha}^{\sst}\right)^\sT$ is an inclusion of connected
  components. Thus, properness of $\sT$-fixed loci follows from
  properness of $\fM_{\alpha}$.
\end{remark}

\section{DT/PT via Mochizuki-style wall-crossing}
\label{sec:mochizuki-WCF}

\subsection{Flag-framed invariants}

\subsubsection{}

In this section, we give a proof of the K-theoretic DT/PT vertex
correspondence, using a somewhat ad-hoc wall-crossing set-up similar
to the one used in \cite{Nakajima2011, Kuhn2021}. Since this sort of
set-up ultimately stems from work of Mochizuki \cite{Mochizuki2009},
we refer to it as {\it Mochizuki-style} wall-crossing. In contrast, in
\S\ref{sec:joyce-WCF}, we provide a different proof of the K-theoretic
DT/PT vertex correspondence using a vastly more general and
complicated wall-crossing setup due to Joyce. Both proofs crucially
require the techniques developed in \S\ref{sec:symmetrized_pullback}. 
We also give a formulation of the Mochizuki-style approach in terms 
of a weak stability condition on the quiver which, as far as we know, 
is new. 

\subsubsection{}

Throughout this section, fix the integer partitions
$(\lambda,\mu,\nu)$ and therefore the class $\beta_C = (\abs{\lambda},
\abs{\mu}, \abs{\nu})$. Furthermore, fix a class $\alpha = (1,
-\beta_C, -n)$, for $n \in \bZ$, and denote by
\[ \fN_m \coloneqq \fN_{(\lambda,\mu,\nu), m} \]
the moduli stack parametrizing pairs $[\cO_{\bar{X}}\to \cE]$ with
prescribed restrictions to the boundary divisors. For short, let
\begin{align*}
  \fN^{\DT}_m &\coloneqq \fN^{\sst}_{(\lambda,\mu,\nu),m}(\tau^-) \\
  \fN^{\PT}_m &\coloneqq \fN^{\sst}_{(\lambda,\mu,\nu),m}(\tau^+)
\end{align*}
denote the moduli stacks parametrizing $\PT$ and $\DT$-stable pairs
respectively, where $\tau^\pm$ are as in
Lemma~\ref{setup:lemma_wall_crossing_moduli}.

Fix $k\gg 0$ large enough (depending on $\alpha$) as described in
\ref{def:auxiliary-stacks}, and $p\geq 1$ arbitrary. The choice of $p$
is unimportant in this section; it will only matter in
\S\ref{sec:joyce-WCF}. Let $N \coloneqq f_{k,p}(\alpha)$. We will
consider the quiver framed stacks $\fN_{m,\vec{d}}^{Q(N)}$ and
$\fQ_{m, \vec{d}}^{Q(N)}$ of Definition~\ref{def:auxiliary-stacks} for
various $m \in \bZ$. The wall-crossing will begin with a framing
dimension vector $\vec d = (d_i)_{i=1}^N$ such that $(\alpha, \vec d)$
is a full flag, in the sense of Definition~\ref{def:full-flag}.

\subsubsection{}

We will set up the wall-crossing step such that the full flag
$(\alpha, \vec d)$ splits into smaller full flags. Hence, for the
remainder of this subsection, fix a class $\beta = (1, -\beta_C, -m)$
with $m \le n$, and a dimension vector $\vec e \le \vec d$ such that
$(\beta, \vec e)$ is a full flag. Let $M \coloneqq f_{k,p}(\beta)$,
and for $\ell\in \{0,1,\ldots,M\}$, let $0 \le i_{\vec e}(\ell) \le N$ be the
minimal index for which $e_{i_{\vec e}(\ell)} = \ell$, with the convention that
$V_0 = 0$ and $e_0 = 0$. Finally, let
\[ \fN^{\fl,\DT}_{m,\vec{e}}, \; \fN^{\fl,\PT}_{m,\vec{e}} \subset \fN^{\fl}_{m, \vec e} \]
denote the preimages of $\fN^{\DT}_m$ and $\fN^{\PT}_m$ under the
forgetful projection $\Pi_{\fN}\colon \fN_{m,\vec{e}}^{\fl}\to
\fN_{m}^\pl$. Objects in these stacks are triples $(I, \vec V, \vec
\rho)$ where every map in $\vec\rho = (\rho_i)_i$ is injective, and $I
= [\cO_{\bar X} \to \cE]$ is a DT- or PT-stable pair.

\subsubsection{}

We introduce stability conditions on the stacks
$\fN_{m,\vec{e}}^{\fl}$.

\begin{definition} \label{def:ellstable}
  Let $\fN_{m,\vec{e}}^{\fl, \ell\text{-}\st} \subseteq
  \fN_{m,\vec{e}}^{\fl}$ be the open substack of objects $(I, \vec V,
  \vec \rho)$ satisfying
  \begin{enumerate}[label = (\roman*)]
  \item \label{it:ellstablei} for a non-zero injection of pairs
    $[0 \to \cZ] \hookrightarrow I$ with $\cZ$ zero-dimensional, we have
    \[ V_{i_{\vec e}(\ell)} \cap F_{k,p}(\cZ) = \{0\} \subset F_{k,p}(I), \]
    where we identify $V_{i_{\vec e}(\ell)}$ with its image in $F_{k,p}(I)$;
  \item \label{it:ellstableii} for a non-zero surjection of pairs $I
    \twoheadrightarrow [0\to \cZ]$ with $\cZ$ zero-dimensional, the
    composition
    \[ V_{i_{\vec e}(\ell)} \to F_{k,p}(I) \twoheadrightarrow F_{k,p}(\cZ)\]
    is non-zero.
  \end{enumerate}
  We say the objects in $\fN_{m,\vec{e}}^{\fl, \ell\text{-}\st}$ are
  {\it $\ell$-stable}. The name is justified by
  Proposition~\ref{prop:flag-stacks-are-algebraic}. By construction,
  note that
  \[\fN_{m,\vec{e}}^{\fl, 0\text{-}\st} = \fN^{\fl,\DT}_{m,\vec{e}} \qquad \fN_{m,\vec{e}}^{\fl, M\text{-}\st} = \fN^{\fl,\PT}_{m,\vec{e}}.\]
\end{definition}

\subsubsection{}

\begin{proposition} \label{prop:flag-stacks-are-algebraic}
  The stacks $\fN_{m,\vec{e}}^{\fl, \ell\text{-}\st}$ are separated
  algebraic spaces with proper $\sT$-fixed loci.
\end{proposition}

Consequently, Theorem~\ref{thm:symmetric-pullback} endows each
$\ell$-stable locus with a symmetric APOT, by symmetrized pullback
along the forgetful map to $\fN_m$.

\begin{proof}
  Let $(I = [\cO_{\bar{X}} \xrightarrow{s} \cE], \vec V, \vec \rho)$
  be an object of $\fN_{m,\vec{e}}^{\fl,\ell\text{-}\st}$. We first
  show that it has no non-trivial automorphisms.

  Consider the zero-dimensional sheaf $\cY \coloneqq \coker s$, and
  let $\cZ \subset \cE$ be the largest zero-dimensional sub-sheaf.
  Since we are on the rigidified stack, any automorphism $\varphi$ of
  the pair $I$ induces the identity on $\cO_{\bar{X}}$ (see the proof
  of Lemma~\ref{lem:rigidification-of-dim-1-pairs}), and therefore is
  of the form $\id_I +\,  \delta_1[-1]$, where $\delta_1\colon \cE \to \cE$
  factors through a map $\cY \to \cZ$ via the natural compositions. In
  particular, $\im(\delta_1)$ is a zero-dimensional sheaf.

  Suppose that $\im \delta_1 \neq 0$, so that $I \to [0 \to \im \delta_1]$
  is a non-zero morphism. By
  Definition~\ref{def:ellstable}\ref{it:ellstableii}, the image of
  $V_{i_{\vec e}(\ell)}$ in $F_{k,p}(\im(\delta_1))$ is non-zero. But the
  automorphism of $F_{k,p}(I)$ induced by $\varphi$ must respect the
  full flag $(\vec V, \vec \rho)$, so in particular $V_{i_{\vec e}(\ell)}$ is
  mapped to itself. This contradicts
  Definition~\ref{def:ellstable}\ref{it:ellstablei}. Hence $\delta_1=0$
  and therefore $\varphi = \id_I$ is a trivial automorphism. Since
  every object has only trivial automorphisms, we see that
  $\fN_{m,\vec{e}}^{\fl,\ell\text{-}\st}$ is an algebraic space.

  The statements about separatedness and properness can be checked via
  the strategy of \S\ref{sec:properness-strategy} using the
  description via weak stability conditions in
  Lemma~\ref{lem:weak-stability}. This is the content of
  Proposition~\ref{prop:valuative-criterion-existence}, in view of
  Remark~\ref{rem:fixed-loci-are-components-in-proper}.
\end{proof}

\subsubsection{}

The spaces $\fN_{m,\vec{e}}^{\fl,\ell\text{-}\st}$ for varying $\ell$
are connected by spaces of ``semi-stable'' objects, as follows.

\begin{definition} \label{def:ellsemistable}
  For $\ell \in \{0, 1, \ldots, M-1\}$, let
  $\fN_{m,\vec{e}}^{\fl,(\ell,\ell+1)\text{-}\sst} \subseteq
  \fN_{m,\vec{e}}^{\fl}$ denote the open sub-stack parametrizing
  objects $(I, \vec V, \vec \rho)$ satisfying
  \begin{enumerate}[label = (\roman*)]
  \item \label{it:ellsemistablei} for a non-zero injection of pairs
    $[0 \to \cZ] \hookrightarrow I$ with $\cZ$ zero-dimensional, we
    have
    \[ V_{i_{\vec e}(\ell)} \cap F_{k,p}(\cZ) = \{0\} \subset F_{k,p}(I); \]
  \item \label{it:ellsemistableii} for a non-zero surjection of pairs
    $I \twoheadrightarrow [0\to \cZ]$ with $\cZ$ zero-dimensional, the
    composition
    \[ V_{i_{\vec e}(\ell+1)} \to F_{k,p}(I) \twoheadrightarrow F_{k,p}(\cZ)\]
    is non-zero.
  \end{enumerate}
  In other words, (i) imposes the corresponding condition from
  $\fN_{m,\vec e}^{\fl,\ell\text{-}\st}$ and (ii) the condition from
  $\fN_{m,\vec e}^{\fl,(\ell+1)\text{-}\st}$. This leads to a
  stability condition which is weaker than both $\ell$- and
  $(\ell+1)$-stability. Consequently, the stack $\fN_{m,\vec{e}}^{\fl,
    (\ell,\ell+1)\text{-}\sst}$ contains both $\fN_{m,\vec
    e}^{\fl,\ell\text{-}\st}$ and $\fN_{m,\vec
    e}^{\fl,(\ell+1)\text{-}\st}$ as open sub-stacks. We say the
  objects in $\fN_{m,\vec{e}}^{\fl, (\ell,\ell+1)\text{-}\sst}$ are
  {\it $(\ell, \ell+1)$-semistable}.
\end{definition}

\subsubsection{}

Through the stacks $\fN_{m,\vec{e}}^{\fl,(\ell,\ell+1)\text{-}\sst}$,
we have reduced to a simple wall-crossing, meaning that semistable
objects split into at most two pieces and the stabilizer dimension of
semistable objects is at most one. Recall the locus $\fQ_{m,\vec
  e}^{\fl,\st}$ from \S\ref{sec:on-wall-invariants-Q}.

\begin{lemma} \label{lem:pos-stab}
  An object $(I, \vec V, \vec \rho)$ in $\fN_{m,\vec
    e}^{\fl,(\ell,\ell+1)\text{-}\sst}$ has non-trivial automorphisms
  if and only if
  \[ (I, \vec V, \vec \rho) = (I', \vec V', \vec \rho') \oplus (\cZ[-1], \vec V'', \vec \rho'') \]
  where, for some full flag $((1, -\beta_C, -m'), \vec e')$ such that
  the first non-zero entry in $\vec e - \vec e'$ occurs at $i_{\vec e}(\ell+1)$,
  \[ [(I', \vec V', \vec \rho')] \in \fN_{m',\vec{e}'}^{\fl, \ell\text{-}\st}, \qquad (\cZ[-1], \vec V'', \vec \rho'') \in \fQ_{m-m', \vec e - \vec e'}^{\fl,\st}. \]
  Each such ``strictly poly-stable'' object has stabilizer group
  $\bC^\times$ given by scaling $\cZ$.
\end{lemma}

\begin{proof}
  It is straightforward to check using the definition of
  $\ell$-stability that any direct sum of the indicated form has
  automorphisms given by the scaling action on $\cZ$.

  The only thing left to prove is that given an object $(I, \vec V,
  \vec \rho)$ with non-trivial automorphism group, it is of the
  claimed form. Indeed, let $\varphi\colon I\to I$ be a non-trivial
  automorphism compatible with the flag structure, and let $\delta
  \coloneqq \varphi-\id_I$. We claim that this induces the desired
  decomposition of $I$ with
  \begin{alignat*}{3}
    I' &\coloneqq \ker \delta \qquad &&\cZ[-1] &&\coloneqq \im \delta, \\
    V_i' &\coloneqq V_i \cap F_{k,p}(I') \qquad &&V_i'' &&\coloneqq \im\left(V_i \to F_{k,p}(I) \to F_{k,p}(\cZ)\right)
  \end{alignat*}
  and the obvious induced quiver maps $\vec\rho'$ and $\vec\rho''$. We
  showed in the proof of
  Proposition~\ref{prop:flag-stacks-are-algebraic} that $\cZ$ has
  zero-dimensional support. One then concludes that:
  \begin{itemize}
  \item the composition $V_{i_{\vec e}(\ell+1)} \to
    F_{k,p}(I)\twoheadrightarrow F_{k,p}(\cZ)$ is non-zero (by
    Definition~\ref{def:ellsemistable}\ref{it:ellsemistableii});
  \item $V_{i_{\vec e}(\ell)}$ factors through $F_{k,p}(I')$ (by
    Definition~\ref{def:ellsemistable}\ref{it:ellsemistablei});
  \item $V_{i_{\vec e}(\ell+1)} \cap F_{k,p}(\cZ)\neq 0$ (since
    $\delta(V_{i_{\vec e}(\ell+1)})\subseteq F_{k,p}(\cZ)$).
  \end{itemize}
  Thus, $W \coloneqq V_{i_{\vec e}(\ell+1)}\cap F_{k,p}(\cZ)$ is
  one-dimensional, and $\delta$ induces an isomorphism on $W$ which is
  scaling by some $t \in \Aut(W) = \bC^\times$. Since $t$ is a
  non-zero element of $\End(W)$ and $\varphi$ was arbitrary, the
  stabilizer group of $(I, \vec V, \vec \rho)$ injects into $\End(W)$.
  Now $\delta/t$ is an idempotent in $\Aut(W)$, so $\delta/t$ must be
  an idempotent projection onto $\cZ$ which gives the desired
  splitting. To conclude:
  \begin{itemize}
  \item $(\cZ[-1], \vec V'', \vec \rho'')$ is stable in the sense of
    \S\ref{sec:on-wall-invariants-Q}, namely the composition
    $V_{i_{\vec e}(\ell+1)}/V_{i_{\vec e}(\ell)} \otimes \cO_{\bar X} \to
    V_{i_{\vec e}(\ell+1)+1}/V_{i_{\vec e}(\ell)} \otimes \cO_{\bar X} \to \cdots \to
    \cZ$ is surjective (by applying
    Definition~\ref{def:ellsemistable}\ref{it:ellsemistableii} to
    quotients of $\cZ$);
  \item $(I', \vec V', \vec \rho')$ is $\ell$-stable (by comparing
    Definitions~\ref{def:ellstable} and \ref{def:ellsemistable}, using
    the splitting). \qedhere
  \end{itemize}
\end{proof}

\subsubsection{}
\label{sec:ellstability-as-weak-stability}

We can characterize $\ell$- and $(\ell,\ell+1)$-stability as weak
stability conditions. Given a non-zero class $(\gamma, \vec{f})$ which is a full flag, let
$i_{\min}(\vec{f})$ be the minimal index for which $f_i > 0$. For a
parameter $\mu \in \bR$, define
\[ \tau_\mu^{\fl}(\gamma,\vec{f}) \coloneqq \begin{cases}
  \left(\tau_0(\gamma), \mu - i_{\min}(\vec{f}) \right) & \gamma = (0,0,-\star), \\
  \left(\tau_0(\gamma),\,  0\right) & \text{otherwise}
\end{cases} \]
where $\star$ denotes a positive integer. One checks that
$\tau_\mu^{\fl}$ is a weak stability condition for any $\mu \in \bR$
using that $i_{\min}(\vec f' + \vec f'') = \min(i_{\min}(\vec f'),
i_{\min}(\vec f''))$. The following is straightforward to verify from
the definitions.

\begin{lemma} \label{lem:weak-stability}
  Let $(I, \vec V, \vec \rho)$ be an object of
  $\fN^{\fl}_{m,\vec{e}}$.
  \begin{enumerate}[label = \roman*)]
  \item Let $0 \leq \ell \leq M$. Choose $i_{\vec e}(\ell) < \mu <
    i_{\vec e}(\ell+1)$. Then $(I, \vec V, \vec \rho)$ is
    $\tau_{\mu}^{\fl}$-stable if and only if it is $\ell$-stable.
  \item Let $0 \leq \ell \leq M - 1$. Choose $\mu = i_{\vec
    e}(\ell+1)$. Then $(I, \vec V, \vec \rho)$ is
    $\tau_\mu^{\fl}$-stable if and only if it is $(\ell,
    \ell+1)$-stable.
  \end{enumerate}
\end{lemma}

\subsubsection{}
\label{sec:mochizuki-properness-flag-stacks}

From Remark~\ref{rem:fixed-loci-are-components-in-proper}, properness
of the $\sT$-fixed loci in the (semi)stable loci of $\fN^{\fl}$
follows immediately from properness of the same (semi)stable loci in
the ambient stack $\fM^{\circ,\fl}$ (see
\S\ref{sec:ambient-moduli-stack}). Recall that $\fM^\circ$ is the
analogue of $\fN$ except one does not impose any conditions on the
restriction of objects to the boundary of $\bar{X}$.
Definitions~\ref{def:ellstable} and \ref{def:ellsemistable} apply
equally well to $\fM^{\circ,\fl}$.

\begin{proposition}\label{prop:valuative-criterion-existence}
  The stacks $\fM_{\beta, \vec{e}}^{\circ,\fl,\ell\text{-}\sst}$ and
  $\fM_{\beta, \vec{e}}^{\circ,\fl,(\ell,\ell+1)\text{-}\sst}$ satisfy
  the existence part of the valuative criterion of properness. In
  particular, the stacks $\fM_{\beta,
    \vec{e}}^{\circ,\fl,\ell\text{-}\sst}$ are proper.
\end{proposition}

\begin{proof}
  Apply Lemma~\ref{lem:weak-stability} to express $\ell$- and
  $(\ell,\ell+1)$-(semi)stability using a weak stability condition
  $\tau_\mu^{\fl}$ for some $\mu$, and follow the general strategy
  outlined in \S\ref{sec:properness-strategy}.

  We first note that for any object in
  $\fM_{\beta,\vec{e}}^{\circ,\fl}$, the sub-quotients in a
  Harder--Narasimhan filtration with respect to $\tau_\mu^{\fl}$ have
  numerical classes which are again full flags, and lie in either the
  $\tau_\mu^{\fl}$-semistable locus of
  $\fM_{\gamma,\vec{f}}^{\circ,\fl}$ or in
  $\fQ_{m',\vec{f}}^{\fl,\st}$.

  Suppose that $(I_K, \vec V_K, \vec \rho_K)$ is a one-parameter
  family of $\ell$- or $(\ell,\ell+1)$-(semi)stable objects over the
  generic point $\Spec K$ of a DVR $R$. By
  Lemma~\ref{lem:val-crit-semistable-base}, we can extend $I_K$ to a
  family of objects $I_R$ in $\fM_\beta^{\circ,\sst}(\tau_0) =
  \fM_\beta^{\circ,\pl}$ over $R$. The quiver data is equivalent to a
  full flag in $F_{k,p}(I_K)$, which we can extend to a full flag in
  $F_{k,p}(I_R)$. This guarantees that the slopes of a
  Harder--Narasimhan filtration of $(I_{\xi}, \vec V_{\xi}, \vec
  \rho_\xi)$ are of the form $(\tau_0(\beta), a)$ for $a \in
  \{\mu-i\}_{0\leq i\leq N}\cup \{0\}$. In particular, we only need to
  check assumption (A2) in \S\ref{sec:properness-strategy}; (A1) and
  (A3) are automatic.

  Thus, suppose we have constructed a sequence $(I_R^{(n)}, \vec
  V_R^{(n)}, \vec \rho_R^{(n)})$ as in
  \S\ref{sec:properness-strategy}. Say, to go from $n$ to $n+1$, we
  are doing an elementary modification along a (semi)stable sub-object
  \[ (\cZ[-1], \vec W, \vec\sigma) \subset (I_{\xi}^{(n)}, \vec V_\xi^{(n)}, \vec\rho_\xi^{(n)}) \text{ such that } \tau_{\mu}^{\fl}((\cZ[-1], \vec W, \vec\sigma)) > (\tau_0(\beta), 0), \]
  with $\cZ$ zero-dimensional, and say that we are in situation 2),
  i.e. the maximally-destabilizing sub-object $(\cZ'[-1], \vec W',
  \vec \sigma')$ of $(I_{\xi}^{(n+1)}, \vec V_\xi^{(n+1)}, \vec
  \rho_\xi^{(n+1)})$ has the same slope as $(\cZ[-1], \vec W,
  \vec\sigma)$. By the definition of $\tau_\mu$, if $i_{0}$ is the
  minimal index such that $W_{i_{0}}\neq 0$, it must also be the
  minimal index $i$ such that $W'_i\neq 0$. In particular $W_{i_0}$
  and $W'_{i_0}$ both have rank one. The induced map $(\cZ'[-1], \vec
  W', \vec\sigma') \to (\cZ[-1], \vec W, \vec\sigma)$ induces an
  isomorphism $W'_{i_0}\to W_{i_0}$. Hence:
  \begin{itemize}
  \item The map $\cZ' \to \cZ$
  must be surjective, since the induced maps $W_{i_0} \otimes
    \cO_{\bar X} \to \cZ$ and $W'_{i_0} \otimes \cO_{\bar X} \to \cZ'$
    are both surjective by (semi)stability;
  \item if $\cZ' \to \cZ$ is an isomorphism, then all induced maps $W'_{i}\to W_i$ are injective.
\end{itemize}
Consequently, exactly one of the two possibilites occurs:
\begin{itemize}
	\item The map $\cZ'\to \cZ$ has non-trivial kernel, and $\length(\cZ')> \length(\cZ)$.
  \item The map $\cZ'\to \cZ$ is an isomorphism, and there exists $i> i_0$, such that $\dim W'_{i} < \dim W_i$. 
 \end{itemize}
  Hence the quantity $\length(\cZ)- \sum_i \dim W_i$ increases at each such step. But it is a-priori
  upper-bounded by $N$, hence such a situation can only occur finitely
  many times, proving (A2) in this case. The case of a quotient object
  with slope lower than $(\tau_0(\beta), 0)$ is analogous.
\end{proof}

\subsection{Wall-crossing formula}

\subsubsection{}

Recall that we fixed a full flag $(\beta, \vec e)$ with $\beta = (1,
-\beta_C, -m)$, and let $M \coloneqq f_{k,p}(\beta)$. Fix $0 \leq \ell
\leq M-1$. We construct a master space $\bM^{\ell,\ell+1}$ relating
$\fN_{m,\vec{e}}^{\fl,\ell\text{-}\st}$ and
$\fN_{m,\vec{e}}^{\fl,(\ell+1)\text{-}\st}$.

\begin{definition} \label{def:mochizuki-master-space-ambient-stack}
  For a full flag $(\gamma, \vec f)$ with $\gamma = (1, -\beta_C,
  -m')$, and an integer $f_{-1} \ge 0$, consider the (unrigidified)
  moduli stack parametrizing tuples $(I, (\hat V_{-1}, \vec V),
  (\hat\rho_{-1}, \hat\rho_0, \vec\rho))$, where:
  \begin{itemize}
  \item $(I, \vec V, \vec \rho)$ is an object of the unrigidified
    stack $\fN_{m',\vec{f}}^{\fl,(\ell,\ell+1)\text{-}\sst} \times
    [\pt/\bC^\times]$ (see
    Lemma~\ref{lem:rigidification-of-dim-1-pairs});
  \item $\hat V_{-1}$ is an $f_{-1}$-dimensional vector space;
  \item $\hat\rho_{-1}$ and $\hat\rho_0$ are linear maps
    \[ V_{i_{\vec e}(\ell+1)}/V_{i_{\vec e}(\ell)} \xrightarrow{\hat\rho_{-1}} \hat V_{-1} \xleftarrow{\hat\rho_0} L \]
    where $I = [\cO_{\bar X} \otimes L \to \cE]$.
  \end{itemize}
  The group $\bC^\times$ of scaling automorphisms acts diagonally on
  $I$, $\vec V$, and $\hat V_{-1}$. Write $(\beta, \vec f, f_{-1})$
  for the numerical class of $(I, (\hat V_{-1}, \vec V),
  (\hat\rho_{-1}, \hat\rho_0, \vec\rho))$ where $\beta \coloneqq \cl
  I$ and $\vec f \coloneqq \dim \vec V$.

  Let $\hat\fN_{m',\vec f, f_{-1}}^{\fl,\ell,\ell+1}$ be the
  $\bC^\times$-rigidified moduli stack, and let $\bM^{\ell,\ell+1}
  \subset \hat\fN_{m,\vec e, 1}^{\fl,\ell,\ell+1}$ denote the open
  substack of objects satisfying the following conditions:
  \begin{enumerate}[label = \roman*)]
  \item if $\hat\rho_{-1} = 0$, then $(I, \vec V, \vec \rho)$ belongs to
    $\fN_{m,\vec{e}}^{\fl,(\ell+1)\text{-}\st}$;
  \item if $\hat\rho_0 = 0$, then $(I, \vec V, \vec \rho)$ belongs to
    $\fN_{m,\vec{e}}^{\fl,\ell\text{-}\st}$;
  \item at least one of $\hat\rho_{-1}$ and $\hat\rho_0$ is non-zero.
  \end{enumerate}
  In the literature, this construction of $\bM^{\ell,\ell+1}$ is also
  called an {\it enhanced master space} \cite[\S\S 1.6.3,
    4.3]{Mochizuki2009}.
\end{definition}

\subsubsection{}
\label{sec:masterspace_alt}

Here is an alternate description of $\bM^{\ell,\ell+1}$. Let
$\scL \coloneqq \scV_{i_{\vec e}(\ell+1)}/\scV_{i_{\vec e}(\ell)}$ on
$\fN_{m,\vec{e}}^{\fl,(\ell,\ell+1)\text{-}\sst}$, and let
$\scL^\times$ denote the associated principal $\bC^\times$-bundle. The
line bundle $\scL$ has the following useful properties.
\begin{enumerate}[label = \roman*)]
\item The stabilizer groups at the stacky points of
  $\fN_{m,\vec{e}}^{\fl,(\ell,\ell+1)\text{-}\sst}$ act non-trivially
  (in fact freely) on $\scL$. In particular, the total space of
  $\scL^\times$ is a Deligne--Mumford stack (in fact an algebraic
  space).
\item Let $x$ be a point in $\scL^\times$ lying over an object $(I,
  \vec V, \vec \rho)$, normalized so that $I = [\cO_{\bar X} \to
    \cE]$. Then $\lim_{t \to 0} tx$ exists if and only if there exists
  a zero-dimensional quotient $\mathcal{E}\twoheadrightarrow \mathcal{Z}$, such that $V_{i_{\vec e}(\ell)}$ lies in the kernel of the map $V_{i_{\vec e}(\ell+1)}\to F_{k,p}(\mathcal{Z})$, i.e. if and only if the
  pair $(I, \vec V, \vec\rho)$ does \emph{not} lie in $\fN_{m,\vec{e}}^{\fl,
    \ell\text{-}\st}$.
\item Similarly, $\lim_{t \to \infty} tx$ exists if and only if $(I,
  \vec V, \vec\rho)$ does not lie in
  $\fN_{m,\vec{e}}^{\fl,(\ell + 1)\text{-}\st}$.
\end{enumerate}
Then $\bM^{\ell,\ell+1}$ is constructed from $\scL^\times$ by gluing
in the total space of
$\scL\big|_{\fN_{m,\vec{e}}^{\fl,\ell\text{-}\st}}$ and the total
space of $\scL^{\vee}\big|_{\fN_{m,\vec{e}}^{\fl,(\ell + 1)\text{-}\st}}$ to
$\scL^{\times}$, thereby adding in the missing limits of $tx$ for
$t \to 0$ and for $t\to \infty$ respectively.

Via this description, we see that $\bM^{\ell,\ell+1}$ is smooth over
$\fN_{m,\vec{e}}^{\fl,(\ell,\ell+1)\text{-}\sst}$ and hence over $\fN_m$.

\subsubsection{}
\label{sec:masterspace-stability}

Here is yet another description of $\bM^{\ell,\ell+1}$ in terms of a
weak stability condition, building on
\S\ref{sec:ellstability-as-weak-stability}. For objects in the moduli
stack $\hat\fN^{\fl,\ell,\ell+1}$ from
Definition~\ref{def:mochizuki-master-space-ambient-stack}, define
\[ \hat\tau_{i_{\vec e}(\ell+1)}^{\fl}(\gamma,\vec{f}, f_{-1}) \coloneqq \begin{cases}
  \left((-\infty,-\infty), \, 0\right) & \gamma = 0,\, \vec{f} = \vec 0, \, f_{-1} > 0, \\
  \left(\tau_{i_{\vec e}(\ell+1)}^{\fl}(\gamma, \vec{f}), \, 1/2 - f_{-1} \right) & \gamma = (0,0, -\star), \, i_{\min}(\vec{f}) = i_{\vec e}(\ell+1), \\
  \left(\tau_{i_{\vec e}(\ell+1)}^{\fl}(\gamma,\vec{f}), \, 0 \right) & \text{otherwise}  \\
\end{cases} \] 
where $(\gamma, \vec{f})$ is a full flag and $\star$ denotes a
positive integer. One checks that $\hat\tau_{i_{\vec
    e}(\ell+1)}^{\fl}$ is a weak stability condition for any $0 \le
\ell \le M-1$, using that it refines the weak stability condition
$\tau^{\fl}_{i_{\vec e}(\ell+1)}$ and some casework.

\begin{lemma}
  The master space $\bM^{\ell,\ell+1}$ is canonically isomorphic to
  the substack of $\hat\fN^{\fl,\ell,\ell+1}_{m, \vec{e}, 1}$ of
  $\hat\tau_{i_{\vec e}(\ell+1)}^{\fl}$-stable objects.
\end{lemma}

\begin{proof}[Proof sketch.]
  Straightforward to verify from the definitions, using the
  observations that for any object in class $(\beta,\vec{e},1)$, we have:
  \begin{itemize}
  \item if $\hat\rho_{-1} = 0$ and $\hat\rho_0 = 0$, then there is a
    quotient-object of numerical class $(0, \vec 0, 1)$;
  \item if $\hat\rho_0 \neq 0$, then all sub-objects of class
    $(\gamma, \vec f, 0)$ must have $\gamma = (0, 0, -\star)$;
  \item if $\hat\rho_{-1} \neq 0$, then all sub-objects of class $(\gamma, \vec f, 0)$ must have $f_{i_{\vec e}(\ell+1)} =
    f_{i_{\vec e}(\ell)}$. \qedhere
  \end{itemize}
\end{proof}

\subsubsection{}

In addition to the $\sT$-action that it inherits from $\fN_m$, the
master space $\bM^{\ell,\ell+1}$ has a natural $\bC^\times$-action
given by scaling $\hat\rho_{-1}$ with weight $z^{-1}$.

\begin{proposition} \label{prop:mochizuki-master-space-properness}
  The master space $\bM^{\ell,\ell+1}$ is a separated algebraic space
  with an induced $\bC^\times \times \sT$ action. The $\sT$-fixed
  locus is proper.
\end{proposition}

\begin{proof}
  It follows directly from Lemma~\ref{lem:separated-algebraic-space}
  that the master space is a separated algebraic space.

  To show properness of $\sT$-fixed loci, we can argue, like we did in
  \S\ref{sec:mochizuki-properness-flag-stacks}, that they are
  components of the $\sT$-fixed loci of a proper space
  $\bar{\bM}^{\ell,\ell+1}$ obtained by working with $\fM^\circ$
  instead of $\fN$ in the constructions of
  Definition~\ref{def:mochizuki-master-space-ambient-stack} and
  \S\ref{sec:masterspace-stability}. In particular we can realize
  $\bar{\bM}^{\ell,\ell+1}$ as an open sub-stack of
  $\hat\tau_{i_{\vec{e}}(\ell+1)}^{\fl}$-stable objects in an ambient stack
  $\hat\fM^{\circ,\fl,\ell,\ell+1}_{\beta,\vec{e},1}$.

  It remains to check the valuative criterion of properness for
  $\bar{\bM}$. Note that any family of $\hat\tau_{i_{\vec{e}}(\ell+1)}^{\fl}$-stable
  objects has an underlying family of $(\ell,\ell+1)$-semistable
  objects. Thus, by
  Proposition~\ref{prop:valuative-criterion-existence}, we can
  complete a family over the generic point of a DVR to some family of
  objects in $\hat\fM^{\circ,\fl,\ell,\ell+1}_{\beta,\vec{e},1}$,
  which yields a $(\ell,\ell+1)$-semistable family under the forgetful
  map to $\fM^{\circ,\fl}_{\beta,\vec{e}}$. Applying the strategy of
  \S\ref{sec:properness-strategy} using the description of
  \S\ref{sec:masterspace-stability}, it remains to show that
  assumption (A2) holds. This follows in the same way as in
  Proposition~\ref{prop:valuative-criterion-existence}.
\end{proof}

\subsubsection{}

From \S\ref{sec:masterspace_alt}, $\bM^{\ell,\ell+1}$ is smooth over
$\fN_m$, and Proposition~\ref{prop:mochizuki-master-space-properness}
shows it is a separated algebraic space with proper $\sT$-fixed loci.
Then Theorem~\ref{thm:symmetric-pullback} endows $\bM^{\ell,\ell+1}$
with a symmetric APOT by symmetrized pullback along the forgetful map
to $\fN_m$. By Theorem~\ref{thm:symmetrized-pullback-localization} and
Lemma~\ref{lem:obstruction-theory-square-root}, we obtain a
symmetrized virtual structure sheaf $\hat\cO^\vir$ amenable to
$(\bC^\times \times \sT)$-equivariant localization.

\subsubsection{}

\begin{proposition}\label{prop:masterspace-fixedloci}
  The $\bC^\times$-fixed locus of $\bM^{\ell,\ell+1}$ is the disjoint
  union of the following pieces.
  \begin{enumerate}
  \item Let $Z_{\hat\rho_{-1}=0} \coloneqq \{\hat\rho_{-1}=0\} \subset
    \bM^{\ell,\ell+1}$. By definition, $\hat\rho_0 \neq 0$. There is a
    natural isomorphism of stacks
    \[ Z_{\hat\rho_{-1}=0} \xrightarrow{\sim} \fN_{m,\vec e}^{\fl,(\ell+1)\text{-}\st} \qquad \Big[(I, (\hat V_{-1}, \vec V), (0, \hat\rho_0, \vec\rho))\Big] \mapsto [(I, \vec V, \vec\rho)] \]
    which identifies the $\hat\cO^\vir$. The virtual normal bundle is
    $z \otimes \scL^\vee$.

  \item Let $Z_{\hat \rho_0=0} \coloneqq \{\hat \rho_0 = 0\} \subset
    \bM^{\ell,\ell+1}$. By definition, $\hat \rho_{-1} \neq 0$. There is a
    natural isomorphism of stacks
    \[ Z_{\hat\rho_0=0} \xrightarrow{\sim} \fN_{m,\vec e}^{\fl,\ell\text{-}\st} \qquad \Big[(I, (\hat V_{-1}, \vec V), (\hat\rho_{-1}, 0, \vec\rho))\Big] \mapsto [(I, \vec V, \vec\rho)] \]
    which identifies the $\hat\cO^\vir$. The virtual normal bundle is
    $z^{-1} \otimes \scL$.

  \item For each splitting $(\beta, \vec e) = (\gamma, \vec f) +
    (\delta, \vec g)$ appearing in Lemma~\ref{lem:pos-stab}, let
    \[ Z_{(\gamma,\vec f), (\delta,\vec g)} \coloneqq \left\{ \Big[(I' \oplus I'', (\hat V_{-1}, \vec V' \oplus \vec V''), (\hat\rho_{-1}', \hat\rho_0', \vec\rho') \oplus (\hat\rho_{-1}'', \hat\rho_0'', \vec\rho''))\Big] : \begin{array}{c} \hat\rho_{-1} \neq 0 \\ \hat\rho_0 \neq 0\end{array}\right\} \subset \bM^{\ell,\ell+1} \]
    where $\cl((I', \vec V', \vec \rho')) = (\gamma, \vec f)$ and
    $\cl((I'', \vec V'', \vec \rho'')) = (\delta, \vec g)$. If $\delta
    = (0, 0, -m')$ then there is a natural isomorphism of stacks
    \begin{align*}
      Z_{(\gamma,\vec f), (\delta,\vec g)} &\xrightarrow{\sim} \fN_{m-m', \vec f}^{\fl,\ell\text{-}\st} \times \fQ_{m', \vec g}^{\st} \\
      \Big[(I, (\hat V_{-1}, \vec V), (\hat\rho_{-1}, \hat\rho_0, \vec\rho))\Big] &\mapsto \left([(I', \vec V', \vec\rho')], [(I'', \vec V'', \vec\rho'')]\right)
    \end{align*}
    which identifies the $\hat\cO^\vir$. Under this isomorphism, the
    virtual normal bundle is
    \begin{align*}
      \cN^\vir_{(\gamma,\vec f),(\delta,\vec g)}
      &= \left(z^{-1} \bF(\vec f, \vec g) + z \bF(\vec g, \vec f)\right) - \kappa \cdot (\cdots)^\vee \\
      &\quad - \left(z^{-1} R\pi_*\cExt(\scI_\gamma, \scI_\delta(-D)) + z R\pi_*\cExt(\scI_\delta, \scI_\gamma(-D))\right)
    \end{align*}
    where $\bF(\vec f, \vec g)$ is the restriction of
    \eqref{eq:quiver-bilinear-obstruction} to numerical class $(\vec
    f, \vec g)$, and $\scI_\gamma$ and $\scI_\delta$ are pullbacks of
    universal families for numerical classes $\gamma$ and $\delta$
    respectively. The $(\cdots)^\vee$ indicates the dual of the
    preceding bracketed term.
  \end{enumerate}
\end{proposition}

\begin{proof}
  The proof is analogous to the one in \cite{Kuhn2021}. We will sketch
  the main parts. The descriptions of the first and second types of
  fixed loci are immediate from the construction. The identification
  of normal bundles is also straightforward: the induced APOT on these
  fixed loci agrees with the natural ones on
  $\fN_{m,\vec{e}}^{\fl,\ell\text{-}\st}$ and
  $\fN_{m,\vec{e}}^{\fl,(\ell+1)\text{-}\st}$ respectively induced by
  symmetrized pullback from $\fN_m$.

  The description of the third type of fixed locus is essentially
  Lemma~\ref{lem:pos-stab} and an argument that this gives the right
  sub-scheme structure on the fixed locus; see \cite[Prop. 4.59 and
    Lemma 4.28]{Kuhn2021}. By Proposition
  \ref{prop:master_space_vir_class_comparison}, to show the
  identification of virtual structure sheaves, it is enough to check
  that the K-theory class of the fixed obstruction theory is the one
  on the target. This, and showing the virtual normal bundle is of the
  claimed form are slightly tedious, but straightforward checks. We do
  this in \S\ref{prop:joyce-master-space-fixed-loci} in detail.
\end{proof}

\subsubsection{}

By Propositions~\ref{prop:valuative-criterion-existence} and
\ref{prop:mochizuki-master-space-properness}, the enumerative
invariants
\begin{align*}
  \tilde \sN_{m,\vec e}(\ell) &\coloneqq \chi\left(\fN_{m,\vec e}^{\fl,\ell\text{-}\st}, \hat\cO^\vir\right) \\
  \tilde \sQ_{m,\vec e} &\coloneqq \chi\left(\fQ_{m,\vec e}^{\fl,\st}, \hat\cO^\vir\right),
\end{align*}
and also $\chi(\bM^{\ell,\ell+1}, \hat\cO^\vir)$, are well-defined by
$\sT$-equivariant localization. Using
Proposition~\ref{prop:masterspace-fixedloci} and applying appropriate
residues to the $\bC^\times$-localization formula on
$\bM^{\ell,\ell+1}$, as written in
Proposition~\ref{prop:master-space-relation}, we immediately obtain
the ``single-step'' wall-crossing formula
\begin{align*}
  0 &= (-1)\left(\kappa^{1/2}-\kappa^{-1/2}\right) \tilde \sN_{m, \vec e}(\ell+1) + (-1)\left(\kappa^{-1/2}-\kappa^{1/2}\right) \tilde \sN_{m, \vec e}(\ell) \\
  &\quad +\sum_{\substack{(\beta, \vec e) = (\gamma, \vec f) + (\delta, \vec g)\\ \gamma \eqqcolon (1,-\beta_C,-m+m')\\ \delta \eqqcolon (0,0,-m')\\i_{\min}(\vec{g}) = i_{\vec e}(\ell+1)}} (-1)^{\ind} (\kappa^{\frac{\ind}{2}}-\kappa^{-\frac{\ind}{2}}) \tilde \sN_{m-m', \vec f}(\ell) \tilde \sQ_{m',\vec g}
\end{align*}
where, according to Proposition~\ref{prop:master-space-relation},
\[\ind = -\chi\left((1, -\beta_C,-m+m'), (0,0,-m')\right) + \chi_Q(\vec{f}, \vec{g}) - \chi_{Q}(\vec{g}, \vec{f}) = m' + c_{Q}(\vec{f}, \vec{g}). \]
Rearranging and rewriting in terms of the quantum integers
\eqref{eq:quantum-integer}, we get
\begin{equation} \label{eq:mochizuki-single-step-WCF}
  \tilde \sN_{m, \vec e}(\ell+1) - \tilde \sN_{m, \vec e}(\ell) = -\sum_{\substack{(\beta, \vec e) = (\gamma, \vec f) + (\delta, \vec g)\\ \gamma \eqqcolon (1,-\beta_C,-m+m')\\ \delta \eqqcolon (0,0,-m')\\i_{\min}(\vec{g}) = i_{\vec e}(\ell+1)}} [m' + c_Q(\vec f, \vec g)]_\kappa \cdot \tilde \sN_{m-m', \vec f}(\ell) \tilde \sQ_{m',\vec g}.
\end{equation}

\subsubsection{}

Recall the full flag $(\alpha, \vec d)$, where $\alpha = (1, -\beta_C,
-n)$ was the original class of interest and $\vec d = (1, 2, \ldots,
N)$. We now iteratively apply \eqref{eq:mochizuki-single-step-WCF} to
express $\tilde \sN_{n, \vec d}(N)$ in terms of invariants $\tilde
\sN_{m, \vec e}(0)$ for varying $m \le n$ and full flags $((1,
-\beta_C, -m), \vec e)$. The resulting wall-crossing formula is
\begin{equation} \label{eq:mochizuki-WCF}
  \tilde \sN_{n, \vec d}(N) = \hspace{-4em}\sum_{\substack{j>0, \, (\alpha,\vec d) = \sum_{i=1}^j (\gamma_i, \vec e_i)\\\{\gamma_i \eqqcolon (0,0,-m_i)\}_{i=1}^{j-1}, \, \gamma_j \eqqcolon (1, -\beta_C, -m_j)\\ \forall i: \; (\gamma_i, \vec e_i) \text{ full flag}\\\vec e_1 < \cdots < \vec e_j}} \hspace{-4em} (-1)^j \tilde \sN_{m_j, \vec e_j}(0) \prod_{i=1}^{j-1} \Big[ m_i - c_Q(\vec{e}_i, \sum_{\ell=i+1}^j \vec{e}_{\ell})\Big]_\kappa \tilde\sQ_{m_i, \vec e_i}.
\end{equation}

Recall that $\sN_{(\lambda,\mu,\nu),m}^{*} \coloneqq \chi(\fN_m^*,
\hat\cO^\vir)$, for $* \in \{\DT, \PT\}$, were the original DT and PT
invariants of interest. Apply \eqref{eq:full-flag-bundle-pushforward}
and Proposition~\ref{prop:on-wall-invariants-Q} to get
\begin{gather*}
  \tilde \sN_{n, \vec d}(N) = [N]_\kappa! \cdot \sN_{(\lambda,\mu,\nu),n}^\PT, \qquad \tilde \sN_{m, \vec e}(0) = [N-n+m]_\kappa! \cdot \sN_{(\lambda,\mu,\nu),m}^\DT \\
  \tilde \sQ_{m, \vec e} = [m-1]_\kappa! \cdot \sN_{(\emptyset,\emptyset,\emptyset),m}^\DT.
\end{gather*}
In the word notation of Definition~\ref{def:word-rearrangements}, the
wall-crossing formula \eqref{eq:mochizuki-WCF} then becomes
\begin{align*}
  \sN_{(\lambda,\mu,\nu),n}^\PT = \sum_{\substack{j\ge 0\\\vec m \in \bZ_{>0}^j}} (-1)^j &\frac{[N-|\vec m|]_\kappa! \, \prod_{i=1}^j [m_i-1]_\kappa! }{[N]_\kappa!} c_>'(m_1, \ldots, m_j, N-|\vec m|) \\[-1.5em]
  &\cdot \sN_{(\lambda,\mu,\nu), n-|\vec m|}^\DT \prod_{i=1}^j \sN_{(\emptyset,\emptyset,\emptyset),m_i}^\DT
\end{align*}
where
\[ c_>'(m_1,\ldots,m_{j+1}) \coloneqq \sum_{\substack{w \in R(m_1,\ldots,m_{j+1})\\o_1(w) > \cdots > o_j(w)}} \prod_{i=1}^j \Big[m_i - \sum_{\ell=i+1}^{j+1} c_{i,j}(w)\Big]_\kappa. \]
Applying Proposition~\ref{prop:mochizuki-combinatorics} below, which
we will prove below using purely combinatorial methods, this becomes
\[ \sN_{(\lambda,\mu,\nu),n}^\PT = \sum_{\substack{j\ge 0\\\vec m \in \bZ_{>0}^j}} (-1)^j \sN_{(\lambda,\mu,\nu), n-|\vec m|}^\DT \prod_{i=1}^j \sN_{(\emptyset,\emptyset,\emptyset),m_i}^\DT, \]
which translates into the identity of generating series
\[ \sum_n Q^n \sN_{(\lambda,\mu,\nu),n}^\PT = \bigg(\sum_n Q^n \sN_{(\lambda,\mu,\nu),n}^\DT\bigg) \bigg(\sum_n Q^n \sN_{(\emptyset,\emptyset,\emptyset),n}^\DT\bigg)^{-1}. \]
This concludes the Mochizuki-style wall-crossing proof of the DT/PT
vertex correspondence (Theorem~\ref{thm:dt-pt}). \qed

\subsection{Combinatorial reduction}
\label{sec:prop_comb}

\subsubsection{}

\begin{proposition} \label{prop:mochizuki-combinatorics}
  For integers $\ell \ge 0$ and $N > 0$, and $\vec m = (m_1, \ldots,
  m_\ell) \in \bZ_{>0}^\ell$,
  \begin{equation} \label{eq:wall_comb}
    \sum_{\sigma \in S_\ell} c'_>(m_{\sigma(1)}, \ldots, m_{\sigma(\ell)}, N-|\vec m|) = \frac{[N]_\kappa!}{[N - |\vec m|]_\kappa! \prod_i [m_i-1]_\kappa!} \cdot \ell!.
  \end{equation}
\end{proposition}

We will prove this combinatorial result by induction on $\ell$. Note
that the symmetrization is really necessary in order to obtain a nice
formula; the individual terms $c'_>(m_{\sigma(1)}, \ldots,
m_{\sigma(\ell)}, N-|\vec m|)$ are significantly more complicated.

\subsubsection{}
\label{sec:mochizuki-base-case}

\begin{proof}[Proof of Proposition~\ref{prop:mochizuki-combinatorics}.]
  
The base case is $\ell = 1$ with $m \coloneqq m_1$. The left hand side
of \eqref{eq:wall_comb} becomes
\[ \sum_{w \in R(m, N-m)} [m - c_{1,2}(w)]_\kappa. \]
Our quantum integers satisfy $[a - b]_\kappa = (-1)^b
\kappa^{\frac{b}{2}} [a]_\kappa - (-1)^a \kappa^{\frac{a}{2}} [b]_\kappa$,
so this sum becomes
\[ \sum_{w \in R(m, N-m)} (-1)^{c_{1,2}(w)} \kappa^{\frac{c_{1,2}(w)}{2}} [m]_\kappa + (-1)^m \kappa^{\frac{m}{2}} [c_{1,2}(w)]_\kappa. \]
The sum over the second term is zero: there is an involution on $R(m,
N-m)$ which takes a word and reverses it, and this acts as $c_{1,2}
\mapsto -c_{1,2}$. In the first term, the quantity $c_{1,2}(w)
\bmod{2}$ is constant on $R(m, N-m)$ because, by
\eqref{eq:word-symmetrized-inversion-number}, every inversion changes
$c_{1,2}(w)$ by $2$. So $(-1)^{c_{1,2}(w)}$ factors out of the sum; it
is enough to evaluate it on the word $w = 1^m 2^{N-m}$, for which
$c_{1,2}(w) = m(N-m)$. We conclude the base case by applying the
following result with $n \coloneqq N-m$.

\begin{proposition} \label{prop:q-binomial}
  For integers $m, n > 0$, 
  \[ (-1)^{mn} \sum_{w \in R(m, n)} \kappa^{\frac{c_{1,2}(w)}{2}} = \frac{[m+n]_\kappa!}{[m]_\kappa! [n]_\kappa!}. \]
\end{proposition}

\begin{proof}
  Let $f(m, n)$ denote the left hand side. By induction on $m+n$, it
  suffices to verify the recursion
  \[ f(m, n) = (-1)^n \kappa^{\frac{n}{2}} f(m-1, n) + (-1)^m \kappa^{-\frac{m}{2}} f(m, n-1). \]
  It is an easy algebraic exercise to verify that the right hand side
  satisfies the same recursion. Clearly $f(m, 0) = 1 = f(0, n)$ for
  any $m, n > 0$, so the necessary base cases hold.

  View $w \in R(m, n)$ as a string of $m$ ones and $n$ twos. Let $w'$
  be $w$ without the first letter $w_1$, so that $w' \in R(m-1, n)$
  if $w_1 = 1$ or $w' \in R(m, n-1)$ if $w_1 = 2$. Using
  \eqref{eq:word-symmetrized-inversion-number}, the contribution of
  the first letter to $c_{1,2}$ is
  \[ c_{1,2}(w) - c_{1,2}(w') = \begin{cases} n & w_1 = 1 \\ -m & w_1 = 2. \end{cases} \]
  The desired recursion follows immediately.
\end{proof}

\subsubsection{}

\begin{remark} \label{rem:q-multinomial}
  Proposition~\ref{prop:q-binomial} is a special case of a more
  general model for $\kappa$-multinomial coefficients:
  \[ (-1)^{\sum_{i=1}^k m_i \sum_{j=1}^{i-1} m_j} \sum_{w \in R(m_1, \ldots, m_k)} \kappa^{\frac{c_{i,j}(w)}{2}} = \frac{[\sum_{i=1}^k m_i]_\kappa!}{\prod_{i=1}^k [m_i]_\kappa!} \]
  for integers $k > 0$ and $m_1, \ldots, m_k > 0$. This can be proved
  using the explicit bijection
  \[ R(m_1, \ldots, m_k) \xrightarrow{\sim} R(m_1 + \cdots + m_{k-1}, m_k) \times R(m_1, \ldots, m_{k-1}), \quad w \mapsto (w', w''), \]
  where $w'$ arises from $w$ by replacing all the letters $< k$ (resp.
  $=k$) by $1$ (resp. $2$), and $w''$ is the subword of $w$ given by
  deleting all the letters $k$. We leave the details to the reader.

  Note, however, that this identity is not so useful for proving
  Proposition~\ref{prop:q-binomial} when $k > 2$, because the ordering
  condition $o_1(w) > \cdots > o_\ell(w)$ becomes a non-trivial
  constraint and the symmetrization becomes important.
\end{remark}
        
\subsubsection{}\label{sec:flag_lemma}

\begin{lemma}
  Let $\{m_i\}_{i=1}^\ell$ and $\{c_{i,j}\}_{i,j=1}^{\ell+1}$ be
  arbitrary sets of integers satisfying $c_{i,j} = -c_{j,i}$. Fix
  $\sigma \in S_\ell$, viewed as an element $\sigma \in S_{\ell+1}$
  such that $\sigma(\ell+1) = \ell+1$, and let $i_0 \coloneqq
  \sigma^{-1}1$. Then
  \begin{align*}
    &\prod_{i=1}^\ell \Big[m_i + \sum_{\substack{j \in [1,\ell+1]\\\sigma^{-1}i< \sigma^{-1}j}}c_{j,i}\Big]_\kappa \\
    =&\sum_{\Omega \subset \sigma[1,i_0-1]} \prod_{i \in [2,\ell]\setminus \Omega} \Big[m_i + \sum_{\substack{j \in [2,\ell+1]\setminus \Omega\\\sigma^{-1}i < \sigma^{-1}j}} c_{j,i}\Big]_\kappa \prod_{i \in \Omega} \Big[\sum_{\substack{j \in \Omega \cup \{1\}\\\sigma^{-1}i<\sigma^{-1}j}} c_{j,i}\Big]_\kappa \Big[\sum_{i\in \Omega \cup \{1\}} m_i +\sum_{\substack{\ell\in [2,\ell+1]\setminus \Omega}} c_{j,i}\Big]_\kappa.
  \end{align*}
\end{lemma}

\begin{proof}
  The identity is the case $i_1=1$ of the formula
  \begin{align*}
    &\prod_{i=1}^\ell \Big[m_i + \sum_{\substack{j \in [1,\ell+1]\\\sigma^{-1}i< \sigma^{-1}j}} c_{j,i}\Big]_\kappa \\
    =& \prod_{1 \leq \sigma^{-1} i < i_1}\Big[m_i +\sum_{\substack{j \in [1,\ell+1]\\\sigma^{-1}i< \sigma^{-1}j}} c_{j,i}\Big]_\kappa\\
    &\sum_{\Omega \subset \sigma[i_1,i_0-1]} \prod_{\substack{i\in [2,\ell]\setminus \Omega\\ i_1\leq \sigma^{-1} i}}\Big[m_i + \sum_{\substack{j \in [2,\ell+1]\setminus\Omega \\ \sigma^{-1}i < \sigma^{-1}j}} c_{j,i}\Big]_\kappa \prod_{i\in \Omega} \Big[\sum_{\substack{j \in \Omega \cup \{1\}\\\sigma^{-1}i<\sigma^{-1}j}} c_{j,i}\Big]_\kappa \Big[\sum_{i\in \Omega \cup \{1\}}  m_i +\sum_{\substack{j\in [2,\ell+1]\setminus \Omega\\i_1 \le \sigma^{-1}j}} c_{j,i}\Big]_\kappa
  \end{align*}
  which holds for all $1\leq i_1\leq i_0$, and which one shows by
  descending induction on $i_1$. To go from $i_1$ to $i_1-1$, one
  takes the term
  \[ \left[m_{\sigma i_1} + (c_{\sigma(i_1+1), \sigma i_1} + \cdots + c_{\sigma i_0, \sigma i_1}) + c_{\sigma(i_0+1), \sigma i_1} + \cdots + c_{\ell+1, \sigma i_1}\right]_\kappa, \]
  denoting the bracketed terms by $C$, and applies the quantum integer
  ``Jacobi identity''
  \[ [A+C]_\kappa [B]_\kappa = [A]_\kappa [B-C]_\kappa + [A+B]_\kappa [C]_\kappa \]
  which holds for arbitrary $A, B, C \in \bZ$. Here $[B]_\kappa$ is
  whichever quantum integer currently contains $m_1$.
\end{proof}

\subsubsection{}
\label{sec:sum_rw}

We begin the inductive proof of
Proposition~\ref{prop:mochizuki-combinatorics}. The case $\ell=0$ is
trivially true, and $\ell=1$ was already done in
\S\ref{sec:mochizuki-base-case}. Suppose that $\ell>1$. Applying
Lemma~\ref{sec:flag_lemma} with $w \in R(m_1, \ldots, m_\ell,
m_{\ell+1})$ and $c_{i,j} = c_{i,j}(w)$, and doing some re-ordering,
the left hand side of Proposition \ref{prop:mochizuki-combinatorics}
becomes
\[ \sum_{\Omega \subset [2,\ell]} \sum_{\substack{w \in R(m_1, \ldots, m_{\ell+1})\\\forall i \in \Omega: \; o_i(w) > o_1(w)}} \Big[\sum_{i \in \Omega \cup \{1\}} m_i + \sum_{j \in [2,\ell+1] \setminus \Omega} c_{j,i}\Big]_\kappa \prod_{i \in [2,\ell] \setminus \Omega} \Big[m_i + \sum_{\substack{j \in [2,\ell+1] \setminus \Omega\\(j=\ell+1 \text{ or}\\o_i(w) > o_j(w))}} c_{j,i}\Big]_\kappa \prod_{i \in \Omega} \Big[\sum_{\substack{j \in \Omega \cup \{1\}\\ o_i(w) > o_j(w)}} c_{j,i}\Big]_\kappa. \]
Given $\Omega \subset [2,\ell]$, let $\vec m_\Omega \coloneqq \{m_i : i \in
\Omega \cup \{1\}\}$ and $\vec m_{\bar\Omega} \coloneqq \{m_i : i \in [2,\ell]
\setminus \Omega\}$. Using the linearity of $c_{j,i} = c_Q(\vec e_j,
\vec e_i)$ in both $\vec e_j$ and $\vec e_i$, we can rewrite this sum
as
\begin{align}
  &\sum_{\Omega\subseteq [2,j]} \sum_{w' \in R(|\vec m_\Omega|, N-|\vec m_\Omega|)} \Big[\sum_{i \in \Omega \cup \{1\}} m_i + c_{2,1}(w')\Big]_\kappa \label{eq:flagsum1}\\
  &\sum_{\substack{w_\Omega \in R(\vec m_\Omega)\\\forall i \in \Omega: \; o_i(w_\Omega) > o_1(w_\Omega)}} \prod_{i \in \Omega} \Big[\sum_{\substack{j \in \Omega \cup \{1\}\\ o_i(w_\Omega) > o_j(w_\Omega)}} c_{j,i}(w_\Omega)\Big]_\kappa \label{eq:flagsum2}\\
  &\sum_{w_{\bar\Omega} \in R(\vec m_{\bar\Omega}, N-|\vec m|)} \prod_{i \in [2,\ell] \setminus \Omega} \Big[m_i + c_{\ell+1,i}(w_{\bar\Omega}) + \sum_{\substack{j \in [2,\ell] \setminus \Omega\\o_i(w_{\bar\Omega}) > o_j(w_{\bar\Omega})}} c_{j,i}(w_{\bar\Omega})\Big]_\kappa. \label{eq:flagsum3}
\end{align}
Here we are treating elements of $R(\vec m_\Omega, m_1)$ words on the
alphabet $\Omega \cup \{1\}$, i.e. rearrangements of the word with
$m_j$ instances of the letter $j$ for $j \in \Omega \cup \{1\}$, and
similarly for $R(\vec m_{\bar\Omega}, N-|\vec m|)$. Evidently $w
\mapsto (w', w_\Omega, w_{\bar\Omega})$ is a bijection.

\subsubsection{}

We analyze the three lines \eqref{eq:flagsum1}, \eqref{eq:flagsum2},
and \eqref{eq:flagsum3}. First, \eqref{eq:flagsum3} is an instance of
the left hand side of \eqref{eq:wall_comb} for a lower value of $\ell$.
By inductive assumption, it is equal to
\[(\ell-|\Omega|-1)! \cdot \frac{[N- |\vec m_\Omega|]_\kappa!}{[N-|\vec m|]_\kappa! \prod_{i\in [2,\ell]\setminus \Omega} [m_i-1]_\kappa!}.\]
Second, for the sum \eqref{eq:flagsum2}, since $o_1(w_\Omega)$ is
minimal among the $o_i(w_\Omega)$, the word $w_\Omega$ begins with the
letter $1$. Removing this $1$ changes the terms $c_{1,i}(w_\Omega)$ to
$m_i + c_{1,i}(w_\Omega')$ and also removes the ordering condition
$o_i(w_\Omega) > o_1(w_\Omega)$, for all $i \in \Omega$. This exhibits
\eqref{eq:flagsum2} as another instance of the left hand side of
\eqref{eq:wall_comb}, so by induction it equals
\[|\Omega|! \cdot \frac{[|\vec m_\Omega|-1]_\kappa!}{\prod_{i\in \Omega \cup \{1\}} [m_i-1]_\kappa!}.\]
Finally, note that from what we have shown, \eqref{eq:flagsum2} and
\eqref{eq:flagsum3} are independent of the choice of $w \in R(|\vec
m_\Omega|, N-|\vec m_\Omega|)$ in \eqref{eq:flagsum1}. Therefore, we
may evaluate the sum in \eqref{eq:flagsum1} directly, which is the
special case $\ell=1$ of Proposition \ref{sec:prop_comb}, and therefore
equals
\[\frac{[N]_\kappa!}{[N-|\vec m_\Omega|]_\kappa! \; [|\vec m_\Omega| - 1]_\kappa!}.\]
In total, and taking the sum over $\Omega$, the left hand
side of \eqref{eq:wall_comb} equals
\[ \sum_{\omega = 0}^{\ell-1} \binom{\ell-1}{\omega} \omega! (\ell-1-\omega)! \frac{[N]_\kappa!}{[N-|\vec m|]_\kappa! \, \prod_i [m_i -1]_\kappa!} = j(j-1)!\frac{[N]_\kappa!}{[N-|\vec m|]_\kappa! \prod_i [m_i -1]_\kappa !}. \]
This concludes the proof of Proposition~\ref{prop:mochizuki-combinatorics}.

\end{proof}

\section{DT/PT via Joyce-style wall-crossing}
\label{sec:joyce-WCF}

\subsection{Quiver-framed invariants}

\subsubsection{} %stability conditions

In this section, we give a proof of the K-theoretic DT/PT vertex using
Joyce's so-called {\it dominant wall-crossing} setup \cite[Theorem
  5.8]{joyce_wc_2021}, which relates invariants within a stability
chamber with invariants on a wall. We avoid most of the generalities
of Joyce's machine, \footnote{In fact, as written in
  \cite{joyce_wc_2021}, Joyce did not have our technology of
  symmetrized pullback using APOTs and therefore does not deal with
  wall-crossing in the CY3 setting.} including all appearances of
vertex and/or Lie algebras, and we hope the content of this section
serves as a concrete illustration of Joyce's machinery in its simplest
form.

The content of this section differs from the Mochizuki-style
wall-crossing of \S\ref{sec:mochizuki-WCF} in two important ways.

First, the family $(\tau_\xi)_\xi$ of weak stability conditions is
genuinely lifted to a family of weak stability conditions on auxiliary
stacks, so that it becomes possible to study enumerative invariants
{\it on} the wall at $\tau_0$. We will consider two separate but
similar setups: wall-crossing from $\tau^-$ to $\tau_0$, and
wall-crossing from $\tau_0$ to $\tau^+$. Both of these wall-crossings
are of individual interest, see \S\ref{sec:invariants-at-the-wall}.
This is in contrast to \S\ref{sec:mochizuki-WCF} where the
``intermediate'' stability conditions have no relation to $\tau_0$,
and the wall-crossing from $\tau^-$ to $\tau^+$ is done all in one go.

Second, in \S\ref{sec:joyce-combinatorics-trick}, we explain a trick
to avoid the combinatorics that appeared in \S\ref{sec:prop_comb}. The
idea is, after proving a wall-crossing formula for a fixed $\alpha =
(1, -\beta_C, -n)$, to sum over all $n$ to get a wall-crossing formula
for generating series. The unknown combinatorial quantity in this
formula for $(\lambda,\mu,\nu)$ can be made equal (up to arbitrary
order, as generating series) to a similar quantity in the same formula
for $(\emptyset,\emptyset,\emptyset)$, using the extra freedom to
choose the framing parameter $p$ in the abstract wall-crossing setup
of \S\ref{sec:quiver-framed-stacks}. But on the other hand every piece
of the formula for $(\emptyset,\emptyset,\emptyset)$ is explicit and
well-understood, so we have avoided any actual combinatorial work.

\subsubsection{}

As in \S\ref{sec:quiver-framed-stacks}, fix the integer partitions
$(\lambda, \mu, \nu)$ and therefore the class $\beta_C =
(\abs{\lambda}, \abs{\mu}, \abs{\nu})$. Furthermore, fix a class
$\alpha = (1, -\beta_C, -n)$ with $\fN_{(\lambda,\mu,\nu),n} \neq
\emptyset$. We will prove a wall-crossing formula for objects in class
$\alpha$. In particular, in what follows, several choices of
parameters depend on $\alpha$.

Recall the family of stability conditions on
$\fN_{(\lambda,\mu,\nu),n}$ from \S\ref{sec:stability-conditions}; let
$\tau^-, \tau^+$ be weak stability conditions in the DT and PT
stability chambers respectively, and $\tau_0$ be on the (only) wall
separating the two chambers.

\subsubsection{}

\begin{definition}[cf. {\cite[Definition 5.5]{joyce_wc_2021}}] \label{def:joyce-auxiliary-stability}
  Fix $k\gg 0$ large enough (depending on $\alpha = (1,-\beta_C,-n)$)
  as described in \S\ref{def:auxiliary-stacks}, and $p \geq 1$
  arbitrary. Let $N \coloneqq f_{k,p}(\alpha)$. For $m \le n$ and a
  dimension vector $\vec e = (e_i)_{i=1}^N$, recall that the auxiliary
  moduli stacks $\fN_{(\lambda,\mu,\nu),m,\vec e}^{Q(N)}$ and
  $\fQ_{m,\vec e}^{Q(N)}$ from Definition~\ref{def:auxiliary-stacks}
  are defined using a quiver of length $N$. We will put a family of
  weak stability conditions on these stacks. Let $n_0$ be the maximal
  integer such that $\fN_{(\lambda,\mu,\nu),m'} = \emptyset$ for all
  $m' \le n_0$, and assume that $m > n_0$ when considering classes
  $(1, -\beta_C, -m)$. Fix the following parameters.
  \begin{itemize}
  \item Pick a generic $\vec \mu = (\mu_i)_{i=1}^N \in \bR^N$ such
    that such that $1 > \mu_1 \gg \mu_2 \gg \cdots \gg \mu_r > 0$.
    Genericity means that the $\mu_i$ must be $\bZ$-linearly
    independent, so that if $\vec\mu \cdot \vec f = 0$ for an integer
    vector $\vec f$ then $\vec f = 0$. The symbols $\gg$ mean that the
    ratios $1/\mu_1$ and $\mu_i/\mu_{i+1}$ for $1 \le i < r$ must
    satisfy finitely many lower bounds.

  \item Pick an additive function $\lambda^\pm$ on numerical classes
    such that $\lambda^\pm(\alpha) = 0$ and $\lambda^-(0, 0, -1) \ll
    0$ (for DT wall-crossing) and $\lambda^+(0, 0, -1) \gg 0$ (for PT
    wall-crossing). \footnote{In \cite[Assumption
        5.2(d)]{joyce_wc_2021}, Joyce requires $\lambda^\pm(\beta) >
      0$ (resp. $\lambda^\pm(\beta) < 0$) iff $\tau^\pm(\beta) >
      \tau^\pm(\alpha)$ (resp. $\tau^\pm(\beta) < \tau^\pm(\alpha)$),
      but this is not satisfied by our choice of $\lambda^\pm$. Joyce
      uses this assumption to prove
      Lemma~\ref{lem:off-wall-invariants} in a more uniform way; see
      \cite[Prop. 10.5]{joyce_wc_2021} which is the only place where
      he uses this assumption. Our proof is rather hands-on but does
      not rely on this assumption.} The symbols $\ll$ and $\gg$ mean
    the finitely many upper/lower bounds such that the proofs of
    Lemma~\ref{lem:off-wall-invariants} and Lemma~\ref{lem:off-wall-flag-invariants} hold.
    
  \item Pick an additive function $r$ on numerical classes such that
    $r(1, -\beta_C, -n_0) = 0$ and $r(0, 0, -1) = 1$, so that $r(1,
    -\beta_C, -m) = m - n_0 > 0$.
  \end{itemize}

  For $1 \le a, b \le N$, let $\vec 1_{[a,b]} \coloneqq (0, \ldots, 0,
  1, \ldots, 1, 0, \ldots, 0)$ where the $1$'s appear exactly in the
  interval $[a, b]$. For brevity, write $\vec 1 \coloneqq \vec
  1_{[1,N]}$. For effective classes $(\beta, \vec e)$ with $\beta \le
  \alpha$, define the family of weak stability conditions
  \begin{equation} \label{eq:joyce-framed-stack-stability}
    \tau_0^{\pm}(s,x)\colon (\beta,\vec e) \mapsto \begin{cases}
      \left(\tau_0(\beta), \frac{s\lambda^\pm(\beta)+(\vec\mu+x\vec 1)\cdot \vec e}{r(\beta)}\right), &\beta = (1, -\beta_C, -\star) \text{ or } (0, 0, -\star), \\
      \left(\tau_0(\beta), \infty\right) & \beta = (1, -\beta_C', -\star), \; \beta_C' \neq \beta_C,\\
      \left(\tau_0(\beta), -\infty\right) & \beta = (0, -\beta_C', -\star), \; \beta_C' \neq 0, \\
      \left(\infty, \frac{(\vec\mu+x\vec 1)\cdot \vec e}{\vec 1\cdot \vec e}\right), & \beta=0,(\vec\mu+x\vec 1)\cdot \vec e>0,\\
      \left(-\infty, \frac{(\vec\mu+x\vec 1)\cdot \vec e}{\vec 1\cdot \vec e}\right), &\beta=0,(\vec\mu+x\vec 1)\cdot \vec e\leq 0
    \end{cases}
  \end{equation}
  for $s \in [0, 1]$ and $x \in [-1, 0]$. Here $\star$ stands for an
  arbitrary integer. This depends continuously on $s$ but possibly
  {\it discontinuously} on $x$ because of the transition between $\pm
  \infty$ in the last two cases. \footnote{The $x$-dependence of
    $\tau_0^\pm(s, x)$ is a remnant from Joyce's machine, particularly
    his construction of {\it semistable invariants}. We only really
    use it in the proof of Lemma~\ref{lem:off-wall-flag-invariants}.
    In principle, with a bit more work, we could have remained at
    $x=0$ throughout this paper.}

  Take the lexicographic ordering on $(\bR \cup \{\pm \infty\})^2$. We
  leave it to the reader to verify that $\tau_0^\pm(s, x)$ is a weak
  stability condition, using the additivity of $\lambda^\pm$ and $r$.
  (The second and third cases in
  \eqref{eq:joyce-framed-stack-stability} are there to ensure that
  this claim is true.) Let
  \begin{equation} \label{eq:joyce-WCF-sst-loci}
    \fN_{(\lambda,\mu,\nu),m,\vec e}^{Q(N),\sst}(\tau_0^\pm(s,x)) \subset \fN_{(\lambda,\mu,\nu),m,\vec e}^{Q(N),\pl}, \qquad \fQ_{m,\vec e}^{Q(N),\sst}(\tau_0^\pm(s,x)) \subset \fQ_{m,\vec e}^{Q(N),\pl}
  \end{equation}
  be the semistable loci, following the notation in
  \S\ref{sec:rigidification} for $\bC^\times$-rigidified stacks. One
  checks easily that the forgetful maps $\Pi_\fN$ and $\Pi_\fQ$ of
  \eqref{eq:framed-stack-forgetful-maps} take
  $\tau_0^\pm(s,x)$-semistable loci to $\tau_0$-semistable loci.
\end{definition}

\subsubsection{}

\begin{lemma} \label{lem:semistable-flag-framing-must-be-injective}
  A $\tau_0^\pm(s, x)$-semistable object of class $(\beta, \vec e)$
  with $\beta \neq 0$ must have all framing maps $\rho_i$ injective.
\end{lemma}

\begin{proof}
  Let $(I, \vec V, \vec \rho)$ be the object. If $\ker\rho_i \neq 0$
  for some $i$, then it induces a non-trivial sub-object and one has
  a splitting
  \[ (I, \vec V, \vec \rho) = (I, \vec V', \vec \rho') \oplus (0, \vec V'', \vec \rho'') \]
  where
  $V''_j \coloneqq \ker(\rho_i \circ \cdots \circ \rho_{j+1} \circ \rho_j)$
  for $j \le i$. Comparing $\tau_0^\pm(s, x)$ for the two summands, it
  is clear that one of the two summands must be destabilizing, a
  contradiction.
\end{proof}

\subsubsection{}%Properness

\begin{proposition}\label{prop:properness-quiver}
  Suppose there are no strictly semi-stable objects in $\fN^{Q(N),
    \sst}_{\beta_C, m, \vec{e}}(\tau_0^{\pm}(s,x))$. Then,
  \[ \fN^{Q(N), \sst}_{(\lambda,\mu,\nu), m, \vec{e}}(\tau_0^{\pm}(s,x)) \subset \fN^{Q(N), \sst}_{\beta_C, m, \vec{e}}(\tau_0^{\pm}(s,x)) \]
  is a separated algebraic space with proper $\sT$-fixed loci.
\end{proposition}

\begin{proof}
  By Lemma~\ref{lem:separated-algebraic-space}, we only need to
  address properness of $\sT$-fixed loci. As in
  \S\ref{sec:mochizuki-properness-flag-stacks}, it is enough to show
  that
  \[ \fM_{1, \beta_C, m, \vec{e}}^{\circ, Q(N), \sst}(\tau_0^\pm(s,x)) \supset \fN_{\beta_C,m, \vec{e}}^{Q(N), \sst}(\tau_0^\pm(s,x)) \]
  is proper; this open inclusion induces an inclusion of connected
  components of $\sT$-fixed loci.

  We follow the strategy of \S\ref{sec:properness-strategy}: let $R$
  be a DVR with fraction field $K$ and closed point $\xi$, and assume
  we are given a family $A_K \coloneqq (I_K, \vec V_K, \vec \rho_K)$
  of $\tau_0^\pm(s,x)$-semistable objects in $\fM_{1,\beta_C,m,
    \vec{e}}^{\circ, Q(N), \pl}$. Since the underlying object $I_K$ is
  $\tau_0$-semistable, it can be extended to a family of objects
  $I_R^{(0)}$ in $\fM_{1,\beta_C,m}^\circ$. Since the quiver $Q(N)$
  has no oriented cycles, the quiver data can also be extended
  compatibly. So we obtain an extension $A_R^{(0)} = (I_R^{(0)}, \vec
  V_R^{(0)}, \vec \rho_R^{(0)})$.

  This addresses (A0) in \S\ref{sec:properness-strategy}. By our
  choice of $I^{(0)}_{\xi}$, all destabilizing sub- or quotient-
  objects of $A_{\xi}^{(0)}$ have numerical class of the form
  $(\gamma, \vec{f})$ with either $\gamma = (1, -\beta_C, -\star)$ or
  $\gamma = (0,0,-\star)$. Throughout, $\star$ stands for an arbitrary
  integer.
	
  Define a sequence of families $A_R^{(n)}$ by following the following
  procedure, which is a slight variation on
  \S\ref{sec:properness-strategy}.
  \begin{enumerate}[label = \arabic*)]
  \item As long as $\tau_{\max}(A_{\xi}^{(n)}) = (\infty, \star)$,
    i.e. $A^{(n)}_{\xi}$ has a sub-object of class $(0, \vec f)$ and
    $(\vec \mu + x \vec 1) \cdot \vec f > 0$, perform an elementary
    modification along a maximally-destabilizing sub-object $B^{(n)}$
    of $A_{\xi}^{(n)}$ to obtain $A_{R}^{(n+1)}$. This means a
    sub-object whose class $(0, \vec f')$ maximizes $(\vec \mu + x
    \vec 1) \cdot \vec f' / \vec 1 \cdot \vec f'$ and, among such
    sub-objects, maximizes $\vec 1 \cdot \vec f'$.
    
  \item Once $\tau_{\max}(A^{(n)}_{\xi}) < (\infty, \star)$, if
    $\tau_{\min}(A^{(n)}_{\xi}) = (-\infty, \star)$, then perform the
    dual procedure to 1), i.e. an elementary modification along a
    maximally-destabilizing quotient object $C^{(n)}$ of $A^{(n)}_\xi$
    to obtain $A_{R}^{(n+1)}$. This means a quotient object whose
    class $(0, \vec f')$ minimizes $(\mu + x \vec 1) \cdot \vec f' /
    \vec 1 \cdot \vec f'$ and, among such quotient objects, maximizes
    $\vec 1 \cdot \vec f'$.

  \item Once $\tau_{\min}(A_{\xi}^{(n)}),
    \tau_{\max}(A_{\xi}^{(n)})\in\{ \tau_0(\beta)\}\times\mathbb{R}$,
    but $A_{\xi}^{(n)}$ is not yet stable, perform an elementary
    modification along a maximally-destabilizing sub-object $B^{(n)}$
    of $A^{(n)}_{\xi}$ to obtain $A_R^{(n+1)}$. This means a
    sub-object whose class $(\gamma, \vec f')$ maximizes the tuple
    \[ \left(\frac{\lambda(B^{(n)}) + (\vec{\mu} + x \vec 1) \cdot \vec f'}{r(B^{(n)})}, r(B^{(n)}), \vec{1} \cdot \vec f'\right) \]
    with respect to the lexicographic ordering.
  \end{enumerate}
  Then, as in Remark~\ref{rem:properness-assumptions}, assumption (A1)
  holds for this sequence. Assumption (A3) holds, since the possible
  values of the quantities $\lambda(\beta), r(\beta),
  \vec{1}\cdot\vec{e}$ and $\vec{\mu}\cdot \vec f'$ that appear for
  effective classes $(\gamma, \vec{f}') \leq (\beta, \vec{e})$ are a
  priori bounded in terms of $N$.
	
  It remains to check (A2). In case 1), one checks that if
  $\tau_0^\pm(s,x)$ is equal for both $B^{(n+1)}$ and $B^{(n)}$, of
  classes the induced map $B^{(n+1)}\to B^{(n)}$ must be injective,
  thus either an isomorphism or we have $\vec{1}\cdot \vec
  f'(B^{(n+1)}) < \vec{1}\cdot \vec f'(B^{(n)})$, which can only
  happen finitely many times. An analogous argument works in case 2)
  (using the quotients $C^{(n)}$), and in case 3), where one shows
  that the tuple $(r(B^{(n)}), \vec{1}\cdot \vec{f}'(B^{(n)}))$
  decreases.
\end{proof}

\subsubsection{}%invariants

We now define enumerative invariants for the semistable loci
\eqref{eq:joyce-WCF-sst-loci} for any $s, x$ such that there are no
strictly $\tau_0^\pm(s, x)$-semistables. Under this assumption,
Theorem~\ref{thm:symmetric-pullback} yields a symmetric APOT on the
semistable (= stable) loci by symmetrized pullback along $\Pi_\fN$ or
$\Pi_\fQ$. By Theorem~\ref{thm:symmetrized-pullback-localization} and
Lemma~\ref{lem:obstruction-theory-square-root}, we obtain a
symmetrized virtual structure sheaf $\hat\cO^\vir$ amenable to
$\sT$-equivariant localization. This allows us to define the
enumerative invariants
\[ \tilde \sZ^\pm_{\beta,\vec e}(s,x) \coloneqq \begin{cases}
    \tilde \sN^\pm_{(\lambda,\mu,\nu),m,\vec e}(s,x) \coloneqq \chi\left(\fN_{(\lambda,\mu,\nu), m,\vec e}^{Q(N),\sst}(\tau^\pm_0(s,x)), \hat\cO^\vir\right) & \beta = (1,-\beta_C,-m)\\
    \tilde\sQ^\pm_{m,\vec e}(s,x) \coloneqq \chi\left(\fQ_{m,\vec e}^{Q(N),\sst}(\tau^\pm_0(s,x)), \hat\cO^\vir\right) & \beta=(0,0,-m)
  \end{cases} \]
by $\sT$-equivariant localization, because
Proposition~\ref{prop:properness-quiver} guarantees that the
$\sT$-fixed loci are proper.

To prove the desired DT/PT vertex correspondence, we will obtain
wall-crossing formulas between $\tilde \sZ^\pm_{\beta,\vec e}(s,x)$
with varying parameters $(s,x)$. The core wall-crossing result
(\S\ref{sec:horizontal_wall_crossing}) will be for
$\tilde \sN^\pm_{(\lambda,\mu,\nu), m, \vec e}(s, 0)$ for
$s \in [0, 1]$ and appropriate $\vec e$.

\subsubsection{}%relations and master space wall-crossing following joyce

Joyce's strategy for wall-crossing between $\tau_0$ and $\tau^\pm$,
using the auxiliary invariants $\tilde \sZ^\pm_{\beta,\vec e}(s, x)$,
consists of the steps illustrated in Figure~\ref{fig:joyce-WCF}.

\begin{figure}[!ht]
  \centering
  \includegraphics{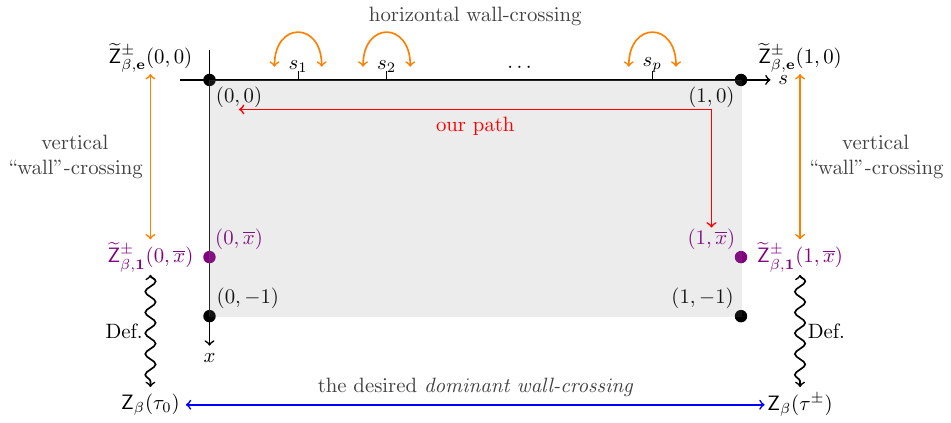}
  \caption{The steps involved in Joyce's dominant wall-crossing}
  \label{fig:joyce-WCF}
\end{figure}

In Joyce's original setup, there is a special $\bar x \in (-1, 0)$
where the auxiliary invariants
$\tilde \sZ^\pm_{\beta,\vec 1}(s, \bar x)$ for $s = 0, 1$ are used to
formally define {\it semistable invariants} $\sZ_\beta(\tau_0)$ and
$\sZ_\beta(\tau^\pm)$ for the original moduli stack. These semistable
invariants are well-defined regardless of whether $\beta$ has strictly
$\tau_0$- or $\tau^\pm$-semistable objects, and Joyce's wall-crossing
formulas are written in terms of them. For their definition in our
situation, see \cite[\S 6]{Liu23}. However, semistable invariants are
irrelevant for this paper because Lemma~\ref{lem:off-wall-invariants}
directly relates
$\tilde \sN^\pm_{(\lambda,\mu,\nu),m,\vec 1}(1, \bar x)$ to the DT/PT
invariants $\sN^\pm_{(\lambda,\mu,\nu),m}$ in a simple way.

For us, the first step (Lemma~\ref{lem:off-wall-flag-invariants}) is
to move from $(s, x) = (1, \bar x)$ to $(s, x) = (1, 0)$. This gives
formulas relating $\tilde\sZ^\pm_{\alpha,\vec d}(1,0)$ and
$\tilde\sZ^\pm_{\alpha,\vec 1}(1,\bar x)$ for a framing dimension
$\vec d$ such that $(\alpha, \vec d)$ is a full flag. The argument
here is not geometric and only involves some fiddling around with the
weak stability condition $\tau_0^\pm(s, x)$.

The second step (\S\ref{sec:horizontal_wall_crossing}) is to vary the
parameter $s$ with $x=0$. This is a truly geometric wall-crossing
argument. The advantage of wall-crossing at $x=0$, instead of $x=\bar
x$ or even $x=-1$, is that objects framed by full flags destabilize
into at most two pieces at each of the walls $s_i$
(Lemma~\ref{lem:joyce-full-flag-splitting-lemma}), which makes master
space techniques applicable.

\subsubsection{}
\label{sec:ss_inv}%semistable invariants

For integers $1 \le a \le N$, let
\[ x_0(a) \coloneqq -\frac{1}{N-a+1}\sum_{i=a}^N\mu_i \in (-1, 0). \]
This is the point where $\tau_0^\pm(s, x)$ has a discontinuity in $x$
for classes $(0,{\vec 1}_{[a,N]})$. In particular, for $x \le x_0(a)$,
$(0,{\vec 1}_{[a,N]})$ becomes a $\tau_0^\pm(s, x)$-destabilizing
quotient class for any numerical class $(\beta,\vec e)$ with $\beta
\neq 0$; cf. \cite[(10.30)]{joyce_wc_2021}.

\begin{lemma} \label{lem:off-wall-invariants}
  Let $1 \le a \le N$ and $\bar x \in (x_0(a), 1]$. Then
  \[ [(I, \vec V, \vec \rho)] \in \fN^{Q(N)}_{(\lambda,\mu,\nu),m,\vec 1_{[a,N]}} \]
  is $\tau_0^\pm(1, \bar x)$-semistable if and only if $I$ is
  $\tau^\pm$-semistable and
  \begin{equation} \label{eq:off-wall-rank-1-framing}
    V_a \xrightarrow{\sim} V_{a+1} \xrightarrow{\sim} \cdots \xrightarrow{\sim} V_N \hookrightarrow V_{N+1} = F_{k,p}(I)
  \end{equation}
  is a chain of isomorphisms followed by an inclusion.
\end{lemma}

This relates our auxiliary invariants to the DT and PT invariants
(\S\ref{sec:enumerative-invariants}) in the corresponding stability
chambers: $\Pi_\fN$ restricted to the class $(\beta, \vec 1_{[a,N]})$,
where $\beta = (1, -\beta_C, -m)$, is therefore a
$\bP^{f_{k,p}(\beta)-1}$-bundle and so
\begin{equation}\label{eq:ss-invars-off-wall}
  \tilde \sN^\pm_{(\lambda,\mu,\nu),m,\vec 1_{[a,N]}}(1, \bar x) = [f_{k,p}(\beta)]_\kappa \cdot \sN_{(\lambda,\mu,\nu),m}(\tau^\pm).
\end{equation}

\subsubsection{}
\label{sec:off-wall-invariants-proof-framing}

\begin{proof}[Proof of Lemma~\ref{lem:off-wall-invariants}.]
  The basic ideas are from \cite[Propositions 10.3(a), 10.5, 10.13(a),
    10.14(v)]{joyce_wc_2021}, which the reader should consult for more
  details. Let $(I, \vec V, \vec \rho)$ be an object of numerical
  class $(\beta, \vec 1_{[a,N]})$.

  We begin with the framing data. Given $(I, \vec V, \vec \rho)$ with
  framing of the form \eqref{eq:off-wall-rank-1-framing}, there are
  sub-objects of numerical class $(\beta, \vec 1_{[b,N]})$ for $a < b
  \le N$, given by the object $I$ along with a sub-chain of framing
  isomorphisms. None of these sub-objects are $\tau_0^\pm(1, \bar
  x)$-destabilizing by our choice of $\bar x > x_0(a)$.

  On the other hand, if the quiver framing is not of the form
  \eqref{eq:off-wall-rank-1-framing}, then $\rho_b = 0$ for some
  $a \le b \le N$ and the object $(I, \vec V, \vec \rho)$ splits as a
  direct sum of objects of class $(0, \vec 1_{[a,b]})$ and
  $(\beta, \vec 1_{[b+1,N]})$. One of these sub-objects is
  $\tau_0^\pm(1, \bar x)$-destabilizing.

  \subsubsection{}

  Now we consider the object $I$. Note first that a sub-object
  $I' \subset I$ induces a sub-object
  $(I', \vec 0, \vec 0) \subset (I, \vec V, \vec \rho)$. If $I'$ has
  class $(1, -\beta_C', -m')$ for $\beta_C' \neq \beta_C$, then $I'$
  and $(I', \vec 0, \vec 0)$ are $\tau^\pm$- and
  $\tau_0^\pm(s, x)$-destabilizing, respectively. Similarly, if $I'$
  has class $(0, -\beta_C', -m')$ for $\beta_C' \neq 0$, then neither
  $I'$ nor any sub-object $(I', \vec V', \vec \rho')$ can cause
  $\tau^\pm$- or $\tau_0^\pm(s, x)$-instability or strict
  semistability. Hence it is enough to assume that $I'$ has class
  $(1, -\beta_C, -m')$ or $(0, 0, -m')$.

  In what follows, we work in the DT chamber with stability condition
  $\tau^-$. The proof for $\tau^+$ in the PT chamber is analogous.
  Recall from Definition~\ref{def:stability-conditions} that $\tau_0$
  is the wall where $\arg z_1 = \arg z_0$. Let $\theta$ denote this
  angle, so that $\theta = \tau_0(1, -\beta_C, -m') = \tau_0(0, 0,
  -m')$. For an effective splitting $(\beta, \vec 1) =
  ((1,-\beta_C,-m'), \vec e') + ((0, 0, m'-m), \vec e'')$,
  \begin{align*}
    \tau_0^-(1,\bar x)((1,-\beta_C,-m'), \vec e') &= \bigg(\theta, \,\frac{m'-n}{m'-n_0}\lambda^-(0,0,-1) + \frac{(\vec \mu +\bar x\vec 1)\cdot \vec e'}{m'-n_0}\bigg), \\
    \tau_0^-(1,\bar x)((0,0,m'-m), \vec e'') &= \bigg(\theta, \,\lambda^-(0,0,-1) + \frac{(\vec \mu +\bar x\vec 1)\cdot \vec e''}{m-m'}\bigg).
  \end{align*}
  Compare the second entries. The absolute values of the second terms
  $((\vec \mu +\bar x \vec 1)\cdot \vec e')/(m'-n_0)$ and $((\vec \mu
  + \bar x \vec 1) \cdot \vec e'')/(m-m')$ are upper-bounded by $N$.
  By our choice of a sufficiently negative $\lambda^-(0,0,-1)$ in
  Definition~\ref{def:joyce-auxiliary-stability}, the first terms
  dominate. Since $n_0 < m' \le n$, it follows that
  \begin{equation} \label{eq:framing-stability-DT-subobject}
    \tau_0^-(1,\bar x)((1,-\beta_C,-m'), \vec e') > \tau_0^-(1,\bar x)((0,0,m'-m), \vec e'').
  \end{equation}

  Suppose $I$ is $\tau^-$-unstable. Then its destabilizing sub-object
  must have class $(1, -\beta_C, -m')$. So $(I, \vec V, \vec \rho)$
  has a sub-object of class $((1, -\beta_C, -m'), \vec 0)$, which is
  $\tau_0^-(1, \bar x)$-destabilizing by
  \eqref{eq:framing-stability-DT-subobject}.
  
  Conversely, suppose $(I, \vec V, \vec \rho)$ is
  $\tau_0^-(1, \bar x)$-unstable, with destabilizing sub-object of
  class $(\gamma, \vec d')$. If $\gamma = \beta$, then the choice of
  $\bar x$ in \S\ref{sec:off-wall-invariants-proof-framing} means the
  framing data cannot be of the form
  \eqref{eq:off-wall-rank-1-framing}. Similarly if $\gamma = 0$, the
  framing maps cannot all be injective and therefore
  \eqref{eq:off-wall-rank-1-framing} cannot hold again. Finally, if
  $\gamma \neq 0, \beta$, then
  \eqref{eq:framing-stability-DT-subobject} means
  $\gamma = (1, -\beta_C, -m')$ is a $\tau^-$-destabilizing sub-object
  of $I$.
\end{proof}

\subsubsection{}

\begin{lemma} \label{lem:off-wall-flag-invariants}
  Let $(\beta, \vec e)$ be an effective numerical class with $\beta
  \neq 0$, $0\leq e_1\leq 1$, $e_N>0$ and $e_i \le e_{i+1} \le e_i + 1$ for all $1 \le
  i < N$. If $a$ is the smallest index where $e_a = e_N$, then for $s
  \in \{0, 1\}$ and $x \in (x_0(a), 0]$:
  \begin{enumerate}[label=(\roman*)]
  \item there are no strictly $\tau_0^\pm(s, x)$-semistable objects
    and $\tilde\sZ_{\beta,\vec e}^\pm(s, x)$ is independent of $x$;
  \item if $e_N > 1$ then
    \begin{equation} \label{eq:flag-invars-off-wall}
      \tilde\sZ^\pm_{\beta,\vec e}(s, x) = [f_{k,p}(\beta) - e_a + 1]_\kappa \cdot \tilde\sZ^\pm_{\beta,\vec e - \vec 1_{[a,N]}}(s, x).
    \end{equation}
  \end{enumerate}
\end{lemma}

The inequalities $1 > \mu_1 > \mu_2 > \cdots > \mu_N > 0$ imply that
$-1 < x_0(1) < x_0(2) < \cdots < x_0(N) < 0$, so
\eqref{eq:flag-invars-off-wall} may be applied recursively.

\begin{proof}
  We sketch the main ideas and refer the reader to \cite[Proposition
  10.14]{joyce_wc_2021} for details.
	
  We first show (i). By
  Lemma~\ref{lem:semistable-flag-framing-must-be-injective}, it
  suffices to consider objects $(I, \vec V, \vec \rho)$ of class
  $(\beta,\vec e)$ with all $\rho_i$ injective. Assume that $(I, \vec
  V, \vec \rho)$ has weakly de-stabilizing sub- and quotient objects
  of classes $(\gamma,\vec f)$ and $(\delta,\vec g)$ respectively.
  Equating the values of $\tau^\pm_0(s,x)$ on these classes, we get
  \begin{equation}\label{eq:strictly_ss_eq}
    \frac{s\lambda^\pm(\gamma)}{r(\gamma)}-\frac{s\lambda^\pm(\delta)}{r(\delta)} = \left(\vec \mu + x\vec 1\right)\cdot\left(\frac{\vec g}{r(\delta)}-\frac{\vec f}{r(\gamma)}\right).
  \end{equation}
  We argue that this cannot hold for $s\in\{0,1\}$ and
  $x\in(x_0(a),0]$. For $s=1$ and $\abs x \leq 1$, this follows from
  our choice of sufficiently positive/negative $\lambda^\pm$ in
  Definition~\ref{def:joyce-auxiliary-stability} unless the terms on
  the left hand side of \eqref{eq:strictly_ss_eq} cancel. In this
  latter case, and for $s=0$, \eqref{eq:strictly_ss_eq} has at most
  one solution. Suppose $\tilde{x}\in \mathbb{R}$ is a solution, and let
  $1\leq b\leq N$ be the minimal index with $e_b>0$. If $a>b$, then by
  our choice of $\vec{\mu}$, the $\mu_b$-term in
  \eqref{eq:strictly_ss_eq} dominates, and $\tilde x$ is close to a
  nonzero integer multiple of $\mu_b$. In particular, it cannot be in
  $(x_0(a),0]$. Finally, if $a=b$, then $\vec e=\vec 1_{[b,N]}$, and
  by Lemma~\ref{lem:semistable-flag-framing-must-be-injective}, $\vec
  f=\vec 1_{[c,N]}$ for some $b\leq c\leq N+1$. Then $\tilde{x}$ can
  be checked to either be $\leq x_0(a)$ or positive. So, in particular
  \eqref{eq:strictly_ss_eq} does not hold for $x\in (x_0(a),0]$. We
    conclude that there are no strictly semi-stable objects.
  
  Next, we claim that $\tau_0^\pm(s, x)$ is continuous in $x\in
  (x_0(a), 0]$ for sub and quotient objects of $(I, \vec V, \vec
  \rho)$. Indeed, $\tau_0^\pm(s, x)$ is only discontinuous for classes
  of the form $(0, \vec g)$, for which the point of discontinuity is
  $x(\vec g) \coloneqq -\vec\mu \cdot \vec g/\vec 1 \cdot \vec g$. By
  injectivity of the framing maps, such classes can only come from
  quotient objects and moreover, if $g_a = 0$ then $g_b = 0$ for all
  $b > a$. It follows that $x(\vec g)$ is maximized, among quotient
  classes $(0, \vec g)$, by $\vec g = \vec 1_{[a, N]}$, and this
  maximum is exactly $x(\vec g) = x_0(a)$. From continuity and the
  absence of strictly semi-stables, we conclude that $\tilde \sZ_{\beta,
    \vec e}(s, x)$ is independent of $x \in (x_0(a), 0]$.

  For (ii), it is enough to prove \eqref{eq:flag-invars-off-wall} for
  $x > x_0(a)$ very close to $x_0(a)$, and then to apply (i). Let $(I,
  \vec V, \vec \rho)$ be a $\tau_0^\pm(s,x)$-semistable object of
  class $(\beta, \vec e)$. Note that our assumptions imply $a>1$. We
  canonically obtain an object sub-object $(I, \vec V', \vec \rho')$
  defined as follows: for $i<a$, take $V'_i\coloneqq V_i$ and for
  $i\geq a$, take $V'_i$ to be the image of $V_{a-1}$ in $V_i$. The
  corresponding quotient is of class $(0, \vec{1}_{[a,N]})$ and is
  given by a chain of isomorphisms of $1$-dimensional vector spaces
  $L_a\cong L_{a+1}\cong \cdots \cong L_N$. It is easy to check
  numerically that $(I, \vec V', \vec \rho')$ is still
  $\tau_0^\pm(s,x)$-semistable for $x$ very close to $x_0(a)$. Hence
  $\Pi\colon (I, \vec V, \vec \rho) \mapsto (I, \vec V', \vec \rho')$
  is a morphism of moduli stacks. The construction can be reversed
  after choosing a generator for $L_a$ among non-zero vectors in
  $F_{k,p}(I)/\im(V_{a-1})$ up to scaling. This means $\Pi$ is a
  $\bP^{f_{k,p}(\beta)-e_a+1}$-bundle. Then
  \eqref{eq:flag-invars-off-wall} follows from
  Proposition~\ref{prop:projective-bundle-pushforward} and
  Lemma~\ref{prop:symmetrized-pullback-functoriality}.
\end{proof}

\subsection{Wall-crossing formulas}
\label{sec:horizontal_wall_crossing}

\subsubsection{}\label{sec:horizontal_wall_crossing_full-flag}

The strategy to obtain the desired DT/PT wall-crossing formula for the
class $\alpha = (1, -\beta_C, -n)$ is to perform the wall-crossing in
$s$ twice: wall-cross using $\tau_0^-(s,0)$ from the DT chamber onto
the wall, by varying $s$ from $1$ to $0$, and then wall-cross using
$\tau_0^+(s,0)$ from the wall into the PT chamber, by varying $s$ from
$0$ to $1$. Schematically, we follow the path in
Figure~\ref{fig:joyce-WCF} twice:
\[ \tilde\sN^-_{(\lambda,\mu,\nu),n,\vec d}(1,0) \leadsto \tilde\sN^-_{(\lambda,\mu,\nu),n,\vec d}(0,0) = \tilde\sN^+_{(\lambda,\mu,\nu),n,\vec d}(0,0) \leadsto \tilde\sN^+_{(\lambda,\mu,\nu),n,\vec d}(1,0). \]
We begin with the class $(\alpha, \vec d)$ where $\vec d = (1, 2,
\ldots, N)$. Recall that this is a full flag, in the sense of
Definition~\ref{def:full-flag}.

\subsubsection{}

It follows immediately from Lemma~\ref{lem:off-wall-flag-invariants}
that if $(\beta, \vec e)$ is a full flag, then iterated applications
of \eqref{eq:flag-invars-off-wall} with $x=0$ gives the full vertical
``wall''-crossing
\begin{align} 
  \tilde \sN^\pm_{(\lambda,\mu,\nu),m,\vec e}(1, 0) &= [f_{k,p}(\beta)]_\kappa! \cdot \sN_{(\lambda,\mu,\nu),m}(\tau^\pm) \label{eq:joyce-x-WCF-for-N} \\
  \tilde \sN^\pm_{(\lambda,\mu,\nu),m,\vec e}(0, 0) &= [f_{k,p}(\beta)-1]_\kappa! \cdot \tilde\sN_{(\lambda,\mu,\nu),m} \label{eq:joyce-x-WCF-for-N-ss}
\end{align}
where in \eqref{eq:joyce-x-WCF-for-N} we further applied
\eqref{eq:ss-invars-off-wall} and in \eqref{eq:joyce-x-WCF-for-N-ss}
we set $\tilde\sN_{(\lambda,\mu,\nu),m} \coloneqq
\tilde\sN^\pm_{(\lambda,\mu,\nu),m,\vec 1}(0, 0)$. We can suppress
$\pm$ from the notation here since
$\tilde\sN^+_{(\lambda,\mu,\nu),m,\vec e}(0,x) =
\tilde\sN^-_{(\lambda,\mu,\nu),m,\vec e}(0,x)$ by the definition of
$\tau_0^\pm(0,x)$.

Completely analogously,
$\tilde\sQ^+_{m,\vec e}(0,0) = \tilde\sQ^-_{m,\vec e}(0,0)$ by
definition, and
Lemma~\ref{lem:joyce-full-flag-splitting-lemma}\ref{it:joyce-full-flag-splitting-dim-0}
shows that there is no wall-crossing in $s$ for the invariants
$\tilde\sQ^\pm_{m,\vec e}(s, 0)$. So we suppress both $\pm$ and $s$
from the notation. If $((0,0,-m), \vec e)$ is a full flag, iterated
applications of \eqref{eq:flag-invars-off-wall} give
\begin{equation} \label{eq:joyce-x-WCF-for-Q}
  \tilde\sQ^\pm_{m,\vec e}(s, 0) \eqqcolon \tilde\sQ_{m,\vec e} = [m-1]_\kappa! \cdot \tilde\sQ_m
\end{equation}
where $\tilde\sQ_m \coloneqq \tilde\sQ^\pm_{m, \vec 1}(0, 0)$ and we
used that $f_{k,p}((0,0,-m)) = m$. Note that there is no analogue of
\eqref{eq:ss-invars-off-wall} here; see also
\S\ref{sec:moduli-stack-of-dim-0-sheaves}.

\subsubsection{}

\begin{lemma} \label{lem:joyce-full-flag-splitting-lemma}
  Suppose $(\beta, \vec e)$ is a full flag and there is an integer $M
  \ge 1$ and a splitting $(\beta, \vec e) = \sum_{i=1}^M (\beta_i,
  \vec e_i)$ such that for some $s \in [0,1]$,
  \begin{equation} \label{eq:joyce-full-flag-splitting-condition}
    \tau_0^\pm(s, 0)(\beta_1, \vec e_1) = \cdots = \tau_0^\pm(s, 0)(\beta_M, \vec e_M)
  \end{equation}
  and there exist $\tau_0^\pm(s,0)$-semistable objects of class
  $(\beta_i, \vec e_i)$ for all $i$. Then:
  \begin{enumerate}[label=(\roman*)]
  \item \label{it:joyce-full-flag-splitting-general} $s \in (0, 1)$,
    $M \le 2$ with equality at only finitely many $s$, and each
    $(\beta_i, \vec e_i)$ is a full flag with no strictly
    $\tau_0^\pm(s,0)$-semistable objects of class $(\beta_i, \vec e_i)$;
  \item \label{it:joyce-full-flag-splitting-dim-0} if $\beta = (0, 0,
    -m)$, then $M = 1$;
  \item \label{it:joyce-full-flag-splitting-rank-1} if $\beta = (1,
    -\beta_C, -m)$ and $M = 2$, then $(\beta, \vec e) = ((1, -\beta_C,
    -m'), \vec e_1) + ((0, 0, m'-m), \vec e_2)$ where:
    \begin{enumerate}
    \item (DT) $\vec e_1 < \vec e_2$ if using $\tau_0^-(s, 0)$;
    \item (PT) $\vec e_1 > \vec e_2$ if using $\tau_0^+(s, 0)$.
    \end{enumerate}
    Here $<$ and $>$ refer to the lexicographical order.
  \end{enumerate}
\end{lemma}

Classes of the form $\beta_i = (0, -\beta_C', -m')$ for $\beta_C' \neq
0$ will never arise in \eqref{eq:joyce-full-flag-splitting-condition}
because their $\tau_0 = \pi/2$ is already smaller than the $\tau_0$ of
classes of the form $(1, -\beta_C', -m')$ or $(0, 0, -m')$.

Later, in the situation of \ref{it:joyce-full-flag-splitting-rank-1}
above, we will take $\{\vec f, \vec g\} = \{\vec e_1, \vec e_2\}$ such
that $\vec f > \vec g$. Then a DT splitting is
$((0, 0, m'-m), \vec f) + ((1, -\beta_C, -m'), \vec g)$ while a PT
splitting is $((1, -\beta_C, -m'), \vec f) + ((0,0,m'-m), \vec g)$.
\subsubsection{}

\begin{proof}[Proof of Lemma~\ref{lem:joyce-full-flag-splitting-lemma}.]
  Part (i) is \cite[Prop. 10.16]{joyce_wc_2021} which we briefly
  review. By the definition of $\tau_0^\pm(s, 0)$, evidently
  $\beta_i \neq 0$ for all $i$. If $\vec e_i = \vec 0$ for some $i$,
  then $e_{j,N} > e_{j,N+1} \coloneqq f_{k,p}(\beta_j)$ for some $j$,
  contradicting
  Lemma~\ref{lem:semistable-flag-framing-must-be-injective}.
  Consequently,
  Lemma~\ref{lem:semistable-flag-framing-must-be-injective} implies
  that each $(\beta_i, \vec e_i)$ must be a full flag. No two
  $\vec e_i$ can be proportional, otherwise their sum cannot satisfy
  the condition $e_a \le e_{a+1} \le e_a+1$ for all $1 \le a < N$.
  Hence the condition \eqref{eq:joyce-full-flag-splitting-condition}
  yields a collection of $M-1$ linear equations of the form
  \begin{equation} \label{eq:joyce-full-flag-splitting-condition-s}
    A_i s + B_i = 0, \qquad B_i \neq 0,
  \end{equation}
  for the variable $s$. Genericity of $\vec\mu$ implies that there is no
  solution if $M > 2$, and that $s \in (0,1)$. Finally, if there
  were a strictly $\tau_0^\pm(s,0)$-semistable object of class
  $(\beta_i, \vec e_i)$, then there would be a splitting
  $(\beta_i, \vec e_i) = (\beta_i', \vec e_i') + (\beta_i'', \vec e_i'')$
  such that, writing $\{i,j\} = \{1,2\}$, the splitting
  $(\beta, \vec e) = (\beta_i', \vec e_i') + (\beta_i'', \vec e_i'') + (\beta_j, \vec e_j)$
  also satisfies the conditions of the lemma, contradicting $M \le 2$.

  Parts (ii) and (iii) are clear after writing out
  \eqref{eq:joyce-full-flag-splitting-condition} explicitly. Namely,
  $((0,0,-m),\vec e)$ can only split into two classes
  $((0,0,-m'),\vec e_1)$ and $((0,0,-m+m'),\vec e_2)$, in which case
  \eqref{eq:joyce-full-flag-splitting-condition-s} reads
  \[ \frac{sm'\lambda^\pm(0,0,-1) + \vec\mu\cdot\vec e_1}{m'} = \frac{s(m-m')\lambda^\pm(0,0,-1)+\vec\mu\cdot\vec e_2}{m-m'}. \]
  This is independent of $s$, i.e. $A_i = 0$, and therefore has no
  solution. Similarly, for
  \eqref{eq:joyce-full-flag-splitting-condition} to hold,
  $((1, -\beta_C, -m), \vec e)$ can only split into two classes
  $((1,-\beta_C,-m'),\vec e_1)$ and $((0,0,-m+m'),\vec e_2)$, in which
  case \eqref{eq:joyce-full-flag-splitting-condition-s} reads
  \[ \frac{s(m'-n)\lambda^\pm(0,0,-1)+\vec\mu\cdot\vec e_1}{m'-n_0} = \frac{s(m-m')\lambda^\pm(0,0,-1)+\vec\mu\cdot\vec e_2}{m-m'}. \]
  This simplifies to the equation
  \begin{equation}\label{eq:s_walls_equation}
    -s(n-n_0)\lambda^\pm(0,0,-1) = \vec\mu\cdot\left(\frac{m'-n_0}{m-m'}\vec e_2 - \vec e_1\right),
  \end{equation}
  which defines the walls $s_i$ in the wall-crossing along $s$. The
  left hand side is positive for $\tau^-$ and negative for $\tau^+$.
  Since $\mu_1 \gg \mu_2 \gg \cdots$, and $n_0 < m' < m$, this implies
  $\vec e_1 < \vec e_2$ and $\vec e_1 > \vec e_2$ in the two cases,
  respectively.
\end{proof}

\subsubsection{}

\begin{definition}[{\cite[Definition 10.19]{joyce_wc_2021}}]
  Fix a full flag $(\beta, \vec e)$ where $\beta = (1, -\beta_C, -m)$.
  Let $s \in (0,1)$ be a wall for $\tilde \sN^\pm_{(\lambda,\mu,\nu),
    m, \vec e}(s, 0)$ from
  Lemma~\ref{lem:joyce-full-flag-splitting-lemma}\ref{it:joyce-full-flag-splitting-rank-1},
  i.e. a solution $s$ to \eqref{eq:s_walls_equation}. For this $s$,
  genericity of $\vec\mu$ implies that the integer vectors $\vec e_1,
  \vec e_2$ solving \eqref{eq:s_walls_equation} are unique.

  Let $\{\vec f, \vec g\} = \{\vec e_1, \vec e_2\}$ ordered so that
  $\vec f > \vec g$. Let $a$ (resp. $b$) be the smallest index where
  $f_a > 0$ (resp. $g_b > 0$) and consider the quiver $\hat Q(k,a,b)$
  given by
  \[ \begin{tikzcd}[column sep=1em]
      \overset{\hat V_1}{\blacksquare} \ar{r} & \blacksquare \ar{r} & \cdots \ar{r} & \blacksquare \ar{r} & \overset{\hat V_a}{\blacksquare} \ar{r} \ar{drr}[swap]{\hat\rho_{-1}} & \blacksquare \ar{r} & \cdots \ar{r} & \blacksquare \ar{r} & \blacksquare \ar{r}{\hat\rho_{b-1}} \ar[dotted]{dll}[swap]{0} & \overset{\hat V_b}{\blacksquare} \ar{r} \ar{dlll}{\hat\rho_0} & \blacksquare \ar{r} & \cdots \ar{r} & \overset{\hat V_r}{\blacksquare} \ar{r} & \overset{F_{k,p}(I)}{\bigbullet} \\
      {} & & & & & & \overset{\hat V_0}{\blacksquare}.
    \end{tikzcd} \]
  In complete analogy with Definition~\ref{def:auxiliary-stacks}, let
  \[ \fN^{\hat Q(k,a,b)}_{(\lambda,\mu,\nu),m,\hat{\vec e}} \coloneqq \left\{\left[(I, \hat{\vec V}, \hat{\vec \rho})\right] : I \in \fN_{(\lambda,\mu,\nu),m}, \; \dim \hat{\vec V} = \hat{\vec e}\right\} \]
  be the moduli stack of triples. We consider only $\hat{\vec e}
  \coloneqq (1, \vec e)$, i.e. $\dim \hat V_0 = 1$. Take the weak
  stability condition $\hat\tau_0^\pm(s,x)$ defined using the formula
  \eqref{eq:joyce-framed-stack-stability} but with the parameter
  $\hat{\vec\mu} \coloneqq (-\epsilon, \vec\mu)$ for a very small
  $\epsilon > 0$. The dotted arrow is a relation, not a quiver map; it
  means to take the closed substack
  \[ \bM \subset \fN^{\hat Q(k,a,b),\sst}_{(\lambda,\mu,\nu),m,\hat{\vec e}}(\hat\tau_0^\pm(s,0)) \]
  where $\hat\rho_0 \circ \hat\rho_{b-1} = 0$. Since $e_b = e_{b-1} +
  1$, this is the substack where $\hat\rho_0$ factors through a
  morphism $\hat V_b/\hat V_{b-1} \to \hat V_0$.
\end{definition}

\subsubsection{}

The moduli stack $\bM$ is the master space for Joyce-style
wall-crossing. Let $\bC^\times$ act by scaling the map $\hat\rho_{-1}$
with weight denoted $z$. We work
$(\bC^\times \times \sT)$-equivariantly on $\bM$.

In any splitting $(\beta, (1, \vec e)) = (\gamma, (1, \vec f)) +
(\delta, (0, \vec g))$, the $\hat\tau_0^\pm(s,0)$ of the two terms
cannot be equal because $\epsilon$ is small and contributes to only
one of them. Hence $\bM$ contains no strictly semistable objects. By
the same reasoning as in Theorem~\ref{prop:properness-quiver}, one
concludes that the master space is a separated algebraic space with
proper $\sT$-fixed loci. Then Theorem~\ref{thm:symmetric-pullback}
endows $\bM$ with a symmetric APOT by symmetrized pullback along the
forgetful map to $\fN^\pl_{(\lambda,\mu,\nu),m}$. By
Theorem~\ref{thm:symmetrized-pullback-localization} and
Lemma~\ref{lem:obstruction-theory-square-root}, we obtain a
symmetrized virtual structure sheaf $\hat\cO^\vir$ amenable to
$(\bC^\times \times \sT)$-equivariant localization.

\subsubsection{}

\begin{proposition}[{\cite[Props. 10.20, 10.21]{joyce_wc_2021}}] \label{prop:joyce-master-space-fixed-loci}
  The $\bC^\times$-fixed locus of $\bM$ is the disjoint union of the
  following pieces.
  \begin{enumerate}
  \item Let $Z_{\hat\rho_{-1}=0} \coloneqq \{\hat\rho_{-1} = 0\}
    \subset \bM$. By stability, $\hat\rho_0 \neq 0$. For $s_- < s$
    very close to $s$, there is a natural isomorphism of stacks
    \begin{align*}
      Z_{\hat\rho_{-1}=0} &\xrightarrow{\sim} \fN^{Q(N),\sst}_{(\lambda,\mu,\nu),m,\vec e}(\tau_0^\pm(s_-,0)) \\
      \left[(I, (\hat V_0, \vec V), (0, \hat \rho_0, \vec \rho))\right] &\mapsto \left[(I, \vec V, \vec \rho)\right]
    \end{align*} 
    which identifies the $\hat\cO^\vir$. The virtual normal bundle
    $\cN^\vir_{\hat\rho_{-1}=0}$ is a line bundle with
    $\bC^\times$-weight $z$.
  \item Let $Z_{\hat\rho_0=0} \coloneqq \{\hat\rho_0 = 0\} \subset
    \bM$. By stability, $\hat\rho_{-1} \neq 0$. For $s_+ > s$ very
    close to $s$, there is a natural isomorphism of stacks
    \begin{align*}
      Z_{\hat\rho_0=0} &\xrightarrow{\sim} \fN^{Q(N),\sst}_{(\lambda,\mu,\nu),m,\vec e}(\tau_0^\pm(s_+,0)) \\
      \left[(I, (\hat V_0, \vec V), (\hat \rho_{-1}, 0, \vec \rho))\right] &\mapsto \left[(I, \vec V, \vec \rho)\right]
    \end{align*}
    which identifies the $\hat\cO^\vir$. The virtual normal bundle
    $\cN^\vir_{\hat\rho_0=0}$ is a line bundle with
    $\bC^\times$-weight $z^{-1}$.
  \item For each splitting
    $(\beta, \vec e) = (\gamma, \vec f) + (\delta, \vec g)$ satisfying
    Lemma~\ref{lem:joyce-full-flag-splitting-lemma}, let
    \[ Z^\pm_{(\gamma,\vec f),(\delta, \vec g)} \coloneqq \left\{\left[(I' \oplus I'', (\hat V_0, \vec V' \oplus \vec V''), (\hat\rho_{-1}, \hat\rho_0, \vec \rho' \oplus \vec \rho''))\right] : \begin{array}{c} \hat\rho_{-1} \neq 0 \\ \hat\rho_0\big|_{V'_b} = 0, \; \hat\rho_0\big|_{V''_b} \neq 0 \end{array}\right\} \subset \bM \]
    where $\cl((I', \vec V', \vec \rho')) = (\gamma, \vec f)$ and
    $\cl((I'', \vec V'', \vec \rho'')) = (\delta, \vec g)$. There are
    natural isomorphisms of stacks
    \begin{align}
      Z^-_{(\gamma,\vec f), (\delta, \vec g)} &\xrightarrow{\sim} \fQ^{Q(N),\sst}_{m',\vec f}(\tau_0^-(s,0)) \times \fN^{Q(N),\sst}_{(\lambda,\mu,\nu),m-m',\vec g}(\tau_0^-(s,0)) \label{eq:fixed-locus-3-isomorphism} \\
      Z^+_{(\gamma,\vec f), (\delta, \vec g)} &\xrightarrow{\sim} \fN^{Q(N),\sst}_{(\lambda,\mu,\nu),m-m',\vec f}(\tau_0^+(s,0)) \times \fQ^{Q(N),\sst}_{m',\vec g}(\tau_0^+(s,0)) \nonumber \\
      \left[(I, (\hat V_0, \vec V), (\hat\rho_{-1}, \hat\rho_0, \vec\rho))\right] &\mapsto \left(\left[(I', \vec V', \vec \rho')\right], \left[(I'', \vec V'', \vec \rho'')\right]\right) \nonumber
    \end{align}
    which identifies the $\hat\cO^\vir$. The target stacks have no
    strictly $\tau_0^\pm(s,0)$-semistable objects by
    Lemma~\ref{lem:joyce-full-flag-splitting-lemma}\ref{it:joyce-full-flag-splitting-general}.
    Under this isomorphism, the virtual normal bundle is
    \begin{align*}
      \cN^\vir_{(\gamma,\vec f),(\delta,\vec g)}
      &= \left(z^{-1} \bF(\vec f, \vec g) + z \bF(\vec g, \vec f)\right) - \kappa \cdot (\cdots)^\vee \\
      &\quad - \left(z^{-1} R\pi_*\cExt(\scI_\gamma, \scI_\delta(-D)) + z R\pi_*\cExt(\scI_\delta, \scI_\gamma(-D))\right)
    \end{align*}
    where $\bF(\vec f, \vec g)$ is the restriction of
    \eqref{eq:quiver-bilinear-obstruction} to numerical class $(\vec
    f, \vec g)$, and $\scI_\gamma$ and $\scI_\delta$ are pullbacks of
    universal families for numerical classes $\gamma$ and $\delta$
    respectively. The $(\cdots)^\vee$ indicates the dual of the
    preceding bracketed term.
  \end{enumerate}
\end{proposition}

\begin{proof}
  The proof exactly follows Joyce's, except when showing that the
  various identifications of fixed loci also identify their
  $\hat\cO^\vir$. While Joyce carefully shows that the virtual cycles 
  agree, in our case it is easier to apply Proposition
  \ref{prop:master_space_vir_class_comparison} and match K-theory
  classes of APOTs. We show how this works in case 3, for $Z \coloneqq
  Z^-_{(\gamma,\vec f), (\delta, \vec g)}$; the other cases are
  similar, but simpler.
  
  On the left hand side of \eqref{eq:fixed-locus-3-isomorphism}, there
  is an APOT which is the $\bC^\times$-fixed part of the APOT on
  $\bM$, which was obtained by symmetrized pullback along the
  forgetful morphism $g\colon \bM \to \fN^\pl_{(\lambda,\mu,\nu),m}$.
  So, in K-theory,
  \begin{equation} \label{eq:master-space-split-fixed-locus-obstruction-theory}
    \begin{aligned}
      &\left(\bE_{\fN^\pl_{(\lambda,\mu,\nu),m}} + \Omega_g - \kappa \Omega_g^\vee\right)\Big|_Z \\
      &= \left(R\pi_* \cExt(\scI_\beta, \scI_\beta(-D)) + \bF(\vec e, \vec e) - \kappa \bF(\vec e, \vec e)^\vee\right)\Big|_Z \\
      &= -R\pi_* \cExt((z\scI_\gamma) \oplus \scI_\delta, (z\scI_\gamma) \oplus \scI_\delta) \\
      &\qquad + \left(\bF(\vec f, \vec f) + \bF(\vec g, \vec g) + z^{-1} \bF(\vec f, \vec g) + z \bF(\vec g, \vec f)\right) - \kappa \otimes (\cdots)^\vee
    \end{aligned}
  \end{equation}
  where the second equality passes through the isomorphism
  \eqref{eq:fixed-locus-3-isomorphism}, e.g. $\scI_\beta$ splits as
  $z\scI_\gamma \oplus \scI_\delta$. We omitted some obvious
  pullbacks.

  On the right hand side of \eqref{eq:fixed-locus-3-isomorphism},
  there is an APOT which is the $\boxplus$ of two APOTs obtained by
  symmetrized pullback along forgetful morphisms to $\fQ^\pl_{m'}$ and
  $\fN^\pl_{(\lambda,\mu,\nu),m-m'}$ respectively. This is equivalent
  to symmetrized pullback along the map $f$ which is the $\boxplus$ of
  the two forgetful morphisms. So
  \begin{align*}
    &\bE_{\fQ^\pl_{m'} \times \fN^\pl_{(\lambda,\mu,\nu),m-m'}} + \Omega_f - \kappa \Omega_f^\vee \\
    &= -R\pi_* \cExt(\scI_\gamma, \scI_\gamma) + R\pi_* \cExt(\scI_\delta, \scI_\delta) + \left(\bF(\vec f, \vec f) + \bF(\vec g, \vec g)\right) - \kappa \otimes (\cdots)^\vee.
  \end{align*}
  This is evidently the $\bC^\times$-fixed part of
  \eqref{eq:master-space-split-fixed-locus-obstruction-theory}. Hence
  the natural $\hat\cO^\vir$ on both sides of
  \eqref{eq:fixed-locus-3-isomorphism} are identified by
  Proposition~\ref{prop:master_space_vir_class_comparison}.

  It is also clear that the non-$\bC^\times$-fixed part of
  \eqref{eq:master-space-split-fixed-locus-obstruction-theory} is the
  claimed virtual normal bundle $\cN^\vir_{(\gamma,\vec f), (\delta,
    \vec g)}$.
\end{proof}

\subsubsection{}

Using Proposition~\ref{prop:joyce-master-space-fixed-loci} and
applying appropriate residues to the $\bC^\times$-localization formula
on $\bM$, as written in Proposition~\ref{prop:master-space-relation},
we immediately obtain the ``single-step'' wall-crossing formula
\begin{equation}\label{eq:single_step_s_wc}
  0 = \tilde\sN_{\beta,\vec e}^\pm(s_-,0) - \tilde\sN_{\beta,\vec e}^\pm(s_+,0) + \hspace{-1.5em}\sum_{\substack{(\beta, \vec e) = (\gamma, \vec f) + (\delta, \vec g)\\\tau^\pm_0(s,0)(\gamma, \vec f) = \tau^\pm_0(s,0)(\delta, \vec g)\\\vec f > \vec g}} \hspace{-1.5em} [\chi(\gamma,\delta) - c_Q(\vec f, \vec g)]_\kappa \cdot \tilde\sZ^\pm_{\gamma,\vec f}(s, 0) \tilde\sZ^\pm_{\delta, \vec g}(s, 0)
\end{equation}
using notation from \S\ref{sec:framing-quiver-notation}. Each term
corresponds to a $\bC^\times$-fixed locus. Note that the splittings
have $\gamma = (0, 0, -m')$ for $\tau^-$ and $\delta = (0, 0, -m')$
for $\tau^+$, so we use the more uniform notation $\tilde\sZ^\pm$ in
the sum.

Collecting the wall-crossings at each wall $s$ which occurs in
Lemma~\ref{lem:joyce-full-flag-splitting-lemma}\ref{it:joyce-full-flag-splitting-rank-1}
yields a combinatorial wall-crossing formula relating
$\tilde\sZ^{\pm}_{\beta,\vec e}(0,0)$ and
$\tilde\sZ^{\pm}_{\beta,\vec e}(1,0)$. These formulas depend on the
direction in which we are moving along the $s$ axis.

\subsubsection{}%DT to wall

For the DT wall-crossing, recursively apply the ``$-$'' case of the
single-step wall-crossing formula \eqref{eq:single_step_s_wc} to move
from $s=1$ to $s=0$, starting from the class $(\alpha, \vec d)$ where
$\vec d = (1, 2, \ldots, N)$. We claim that the resulting formula is
\begin{equation} \label{eq:joyce-s-WCF-for-DT-rough}
  \tilde\sN^-_{(\lambda,\mu,\nu), n, \vec d}(1, 0) = \hspace{-4em} \sum_{\substack{\ell>0, \; (\alpha,\vec d) = \sum_{i=1}^\ell (\gamma_i, \vec e_i)\\\{\gamma_i \eqqcolon (0, 0, -m_i)\}_{i=1}^{\ell-1}, \; \gamma_\ell \eqqcolon (1, -\beta_C, -m_\ell)\\\forall i:\, (\gamma_i,\vec e_i) \text{ full flag}\\\vec e_1 > \cdots > \vec e_\ell}}\hspace{-4em}\tilde \sN_{(\lambda,\mu,\nu), m_\ell,\vec e_\ell}(0,0) \prod_{i=1}^{\ell-1} \Big[\chi(\gamma_i,\sum_{j=i+1}^\ell \gamma_j) - c_Q(\vec e_i, \sum_{j=i+1}^\ell \vec e_j)\Big]_\kappa \tilde \sQ_{m_i,\vec e_i}
\end{equation}
We explain the conditions in the sum. By Lemma
\ref{lem:joyce-full-flag-splitting-lemma}, all full flags split into
smaller full flags at each wall. Since we started with a full flag
$(\alpha,\vec d)$, note that all sums $\sum_{i\in I} (\gamma_i,\vec
e_i)$ for any $I\subset\{1,\dots,\ell\}$ must be full flags.
Conversely, every splitting into two smaller full flags occurs at some
$s \in (0, 1)$ because $|\lambda^\pm(0,0,-1)|$ in
\eqref{eq:s_walls_equation} is very big. Finally, the ordering
condition $\vec e_1 > \cdots > \vec e_\ell$ arises by considering two
consecutive splittings. Namely, take walls $s' < s$ where numerical
classes split as
\begin{align*}
  ((1, -\beta_C, -m), \vec f_{i-1}) &= ((0, 0, -m+m'), \vec e_i) + ((1, -\beta_C, -m'), \vec f_i) \quad \text{at } s \\
  ((1, -\beta_C, -m'), \vec f_i) &= ((0, 0, -m'+m''), \vec e_{i+1}) + ((1, -\beta_C, -m''), \vec f_{i+1}) \quad \text{at } s'.
\end{align*}
Since $\vec e_i > \vec f_i = \vec e_{i+1} + \vec f_{i+1}$, it is clear
that $\vec e_i > \vec e_{i+1}$. Iterating this over all walls yields
the ordering condition.
%% $((1,-\beta,-m),\vec d)$ splits into two classes $((0,0,-m+m'),\vec
%% e)$ and $((1,-\beta,-m'),\vec f)$. Take another wall $s'<s$ where
%% $((1,-\beta,-m'),\vec f)$ splits into two classes
%% $((0,0,-m'+m''),\mathbf{e}')$ and $((1,-\beta,-m''),\mathbf{f}')$. As
%% the left-hand side in \eqref{eq:s_walls_equation} is positive for both
%% of these walls, we obtain $\vec e > \vec f$ and $\vec e' > \vec f'$.
%% We now want to compare $\vec e$ with $\vec e'$. Using
%% \eqref{eq:s_walls_equation} and $\vec f= \vec e' +\vec f'$, we get
%% \begin{equation*}
%%   \frac{-1}{(n-n_0)\lambda^-}\vec \mu\cdot \left(\frac{m''-n_0}{m'-m''}\vec e' - \vec f'\right) = s'<s = \frac{-1}{(n-n_0)\lambda^-}\vec \mu\cdot \left(\frac{m'-n_0}{m-m'}\vec e - \vec e' - \vec f'\right).
%% \end{equation*}
%% This is equivalent to 
%% \begin{equation*}
%%   \frac{m'-n_0}{m'-m''}\vec \mu\cdot \vec e' < \frac{m'-n_0}{m-m'}\vec \mu\cdot \vec e,
%% \end{equation*}
%% which as above yields $\vec e > \vec e'$ for generic choice of $\vec\mu$ with $\mu_1\gg\mu_2\gg\cdots\gg\mu_r$, using that $((0,0,-m+m'),\mathbf{e})$, $((0,0,-m'+m''),\mathbf{e}')$ and their sum are correctly stepped. In this setting, $\vec e = \vec e_1$, $\vec e'=\vec e_2$, and $\vec f'$ splits into the remaining $\vec e_i$.
In the word notation of Definition~\ref{def:word-rearrangements},
\begin{equation} \label{eq:joyce-s-WCF-for-DT}
  \tilde\sN^-_{(\lambda,\mu,\nu), n, \vec d}(1, 0) = \hspace{-1em}\sum_{\substack{k>0, \, \vec m \in \bZ_{>0}^k\\w \in R(\vec m, N-|\vec m|)\\o_1(w) < \cdots < o_{k+1}(w)}} \hspace{-1em} \tilde \sN_{(\lambda,\mu,\nu), n-|\vec m|, \vec e_{k+1}}(0, 0) \prod_{i=1}^k \Big[m_i - \sum_{j=i+1}^{k+1} c_{i,j}(w)\Big]_\kappa \tilde \sQ_{m_i, \vec e_i}
\end{equation}
where $(\vec e_i)_{i=1}^{k+1}$ is the decomposition of $\vec d$
associated to the word $w$, and we used that
\[ \chi((0, 0, -m), (1, -\beta_C, -m')) = m. \]
The condition that $\vec e_1 > \cdots > \vec e_\ell$ becomes the
condition that $o_1(w) < \cdots < o_{k+1}(w)$.

\subsubsection{}%Wall to PT

For the PT wall-crossing, recursively apply the ``$+$'' case of the
single-step wall-crossing formula \eqref{eq:single_step_s_wc} to move
along $s$ in the other direction, from $s=0$ to $s=1$. We start from a
full flag $(\beta, \vec e)$ with $\beta = (1, -\beta_C, -m)$; later
this will be $(\gamma_\ell, \vec e_\ell)$ from
\eqref{eq:joyce-s-WCF-for-DT-rough}. We claim that the resulting
formula is
\begin{equation} \label{eq:joyce-s-WCF-for-PT-rough}
  \tilde \sN_{(\lambda,\mu,\nu), m, \vec e}(0, 0) = \hspace{-4em} \sum_{\substack{\ell>0, \; (\beta,\vec e) = \sum_{i=1}^\ell (\gamma_i, \vec f_i)\\\{\gamma_i \eqqcolon (0, 0, -m_i)\}_{i=1}^{\ell-1}, \; \gamma_\ell \eqqcolon (1, -\beta_C, -m_\ell)\\\forall i:\, (\gamma_i,\vec f_i) \text{ full flag}\\\vec f_\ell > \vec f_1 > \cdots > \vec f_{\ell-1}}} \hspace{-4em}\tilde \sN^+_{(\lambda,\mu,\nu), m_\ell,\vec f_\ell}(1,0) \prod_{i=1}^{\ell-1} \Big[\chi(\gamma_i, \sum_{j=i+1}^\ell \gamma_j)-c_Q(\vec f_i, \sum_{j=i+1}^\ell \vec f_j)\Big]_\kappa \tilde \sQ_{m_i,\vec f_i}
\end{equation}
where we absorbed a sign $(-1)^{\ell-1}$ into the quantum integers
using the skew-symmetry of $\chi$ and $c_Q$. To explain the ordering
condition in the above sum, we again consider two consecutive walls
and splittings
\begin{align*}
  ((1, -\beta_C, -m), \vec e_{i-1}) &= ((1, -\beta_C, -m'), \vec e_i) + ((0, 0, -m+m'), \vec f_i) \quad \text{at } s \\
  ((1, -\beta_C, -m'), \vec e_i) &= ((1, -\beta_C, -m''), \vec e_{i+1}) + ((0, 0, -m'+m''), \vec f_{i+1}) \quad \text{at } s',
\end{align*}
but now $s' > s$ since we are moving in the other direction along $s$,
and $\vec e_{i+1} + \vec f_{i+1} = \vec e_i > \vec f_i$ and we must
compare $\vec f_i$ with $\vec f_{i+1}$. Using
\eqref{eq:s_walls_equation}, we get
\[ \frac{-1}{(n-n_0)\lambda^+}\vec \mu\cdot \left(\frac{m''-n_0}{m'-m''}\vec f_{i+1} - \vec e_{i+1}\right) = s'>s = \frac{-1}{(n-n_0)\lambda^+}\vec \mu\cdot \left(\frac{m'-n_0}{m-m'}\vec f_i - \vec e_{i+1} - \vec f_{i+1}\right) \]
where $\lambda^+$ is short for $\lambda^+(0,0,-1)$. As the prefactor
$-1/(n-n_0)\lambda^+$ is negative, this is equivalent to
\[ \frac{m'-n_0}{m'-m''}\vec \mu\cdot \vec f_{i+1} < \frac{m'-n_0}{m-m'}\vec \mu\cdot \vec f_i, \]
which yields $\vec f_i > \vec f_{i+1}$ since $\vec\mu$ is generic.
Iterating this over all walls yields the ordering condition $\vec f_1
> \cdots > \vec f_{\ell-1}$. Finally, in a sequence of splittings
\[ \vec e = \vec e_1 + \vec f_1 = (\vec e_2 + \vec f_2) + \vec f_1 = \cdots = ((\vec e_{\ell-1} + \vec f_{\ell-1}) + \cdots) + \vec f_1, \]
where $\vec e_i > \vec f_i$ by
Lemma~\ref{lem:joyce-full-flag-splitting-lemma}\ref{it:joyce-full-flag-splitting-rank-1},
every $\vec e_i$ has the same first non-zero entry as $\vec e$ since
$(\beta, \vec e)$ is a full flag. Hence $\vec f_\ell \coloneqq \vec
e_{\ell-1} > \vec f_1$. In the word notation of
Definition~\ref{def:word-rearrangements},
\begin{equation} \label{eq:joyce-s-WCF-for-PT}
  \tilde\sN_{(\lambda,\mu,\nu), m, \vec e}(0, 0) = \hspace{-2em}\sum_{\substack{k>0, \, \vec m \in \bZ_{>0}^k\\w \in R(\vec m, M-|\vec m|)\\ o_{k+1}(w) < o_1(w) < \cdots < o_k(w)}} \hspace{-2em}\tilde \sN_{(\lambda,\mu,\nu), m-|\vec m|, \vec f_{k+1}}^+(1, 0) \prod_{i=1}^k \Big[m_i - \sum_{j=i+1}^{k+1} c_{i,j}(w)\Big]_\kappa \tilde \sQ_{m_i, \vec f_i}
\end{equation}
where $M \coloneqq f_{k,p}((1, -\beta_C, -m))$ and $(\vec
f_i)_{i=1}^{k+1}$ is the decomposition of $\vec e$ associated to the
word $w$.

\subsubsection{}

\begin{proposition}
  In the word notation of Definition~\ref{def:word-rearrangements},
  \begin{equation} \label{eq:joyce-s-WCF}
    \tilde \sN_{(\lambda,\mu,\nu),n,\vec d}^-(1,0) = \hspace{-1em}\sum_{\substack{k>0, \, \vec m \in \bZ_{>0}^k\\w \in R(\vec m, N-|\vec m|)\\o_1(w) < \cdots < o_k(w)}}\hspace{-1em} \tilde \sN_{(\lambda,\mu,\nu),n-|\vec m|,\vec e_{k+1}}^+(1,0) \prod_{i=1}^k \Big[m_i - \sum_{j=i+1}^{k+1} c_{i,j}(w)\Big]_\kappa \tilde\sQ_{m_i,\vec e_i}.
  \end{equation}
\end{proposition}

\begin{proof}
  Plug \eqref{eq:joyce-s-WCF-for-PT-rough} into
  \eqref{eq:joyce-s-WCF-for-DT-rough}. The composite of a DT splitting
  in \eqref{eq:joyce-s-WCF-for-DT-rough} with a PT splitting in
  \eqref{eq:joyce-s-WCF-for-PT-rough} has the form
  \begin{align*}
    \vec d
    &= \vec e_1 + \cdots + \vec e_{\ell_--1} + \vec e_{\ell_-} \\
    &= \vec e_1 + \cdots + \vec e_{\ell_--1} + (\vec f_{\ell_+} + \vec f_{\ell_+-1} + \cdots + \vec f_1)
  \end{align*}
  on framing dimensions, where $\vec e_1 > \cdots > \vec e_{\ell_--1} >
  \vec e_{\ell_-}$ (DT ordering) and $\vec f_{\ell_+} > \vec f_1 >
  \cdots > \vec f_{\ell_+-1}$ (PT ordering). Consequently
  \[ \vec e_{\ell_--1} > \vec e_{\ell_-} > \vec f_{\ell_+} \]
  because $\vec f_{\ell_+}$ is a summand of $\vec e_{\ell_-}$. Hence
  the composite ordering condition is
  \[ \vec e_1 > \cdots > \vec e_{\ell_--1} > \vec f_{\ell_+} > \vec f_1 > \cdots > \vec f_{\ell_+-1}. \]
  Since we sum over all $\ell_->0$ and $\ell_+>0$, the result is that
  there is no constraint on $\vec f_{\ell_+}$. Setting $\vec
  e_{\ell_--1+i} \coloneqq \vec f_i$ and $\ell \coloneqq \ell_- +
  \ell_+ - 1$, the composite splitting therefore has the ordering
  $\vec e_1 > \cdots > \vec e_{\ell-1}$. In word notation this becomes
  $o_1(w) < \cdots < o_{\ell-1}(w)$.

  Conversely, the position of $o_\ell(w)$ specifies the splitting
  $\ell = (\ell_- -1) + \ell_+$ into the DT and PT parts of the
  splitting.
\end{proof}

\subsection{``Combinatorial reduction''}
\label{sec:joyce-combinatorics-trick}

\subsubsection{}

We can use the formulas \eqref{eq:joyce-x-WCF-for-N} and
\eqref{eq:joyce-x-WCF-for-Q} to remove the dependence on framing
dimensions $\vec e$ from the auxiliary invariants $\tilde \sN$ and
$\tilde \sQ$ in \eqref{eq:joyce-s-WCF}. The result is
\begin{equation} \label{eq:joyce-DT-PT-WCF}
  \sN_{(\lambda,\mu,\nu),n}(\tau^-) = \sum_{m \in \bZ} \sN_{(\lambda,\mu,\nu),n-m}(\tau^+) \cdot \sW_{m, N}
\end{equation}
where
\begin{align*}
  \sW_{m, N} &\coloneqq \sum_{\substack{k>0, \; \vec m \in \bZ_{>0}^k\\|\vec m| = m}} c_<'(\vec m, N-m) \frac{[N-m]_\kappa!}{[N]_\kappa!} \prod_{i=1}^k [m_i-1]_\kappa!\cdot \tilde\sQ_{m_i} \\
  c_<'(m_1,\ldots,m_\ell) &\coloneqq \sum_{\substack{w \in R(m_1,\ldots,m_\ell)\\o_1(w) < \cdots < o_{\ell-1}(w)}} \prod_{i=1}^{\ell-1} \Big[m_i - \sum_{j=i+1}^\ell c_{i,j}(w)\Big]_\kappa.
\end{align*}
The combinatorial quantity $c_<'$ should be compared with the similar
quantity $c_>'$ from \S\ref{sec:prop_comb}.

\subsubsection{}

In principle, one can now evaluate $\sW_{m,N}$ explicitly in order to
obtain the desired DT/PT vertex correspondence; see
Proposition~\ref{prop:joyce-combinatorics-explicit}. The combinatorics
involved is similar to that of \S\ref{sec:prop_comb}, so, instead, we
provide a different method of obtaining the DT/PT vertex
correspondence from \eqref{eq:joyce-DT-PT-WCF}.

The idea is as follows. All the wall-crossing so far was done in a
setup for fixed $\alpha = (1,-\beta_C,-n)$, especially the single-step
wall-crossing \eqref{eq:single_step_s_wc} for full flags $(\beta, \vec
e)$. We now wish to allow $n$ to vary because the DT/PT vertex
correspondence is an equality of generating series. However, the $N$
in \eqref{eq:joyce-DT-PT-WCF} depends non-trivially on $n$ (and on the curve class $\beta_C$). We will
remove this dependence as follows.

\subsubsection{}
\label{sec:joyce-WCF-generating-functions} %DT/PT correspondence

Fix an integer $M$. The DT/PT vertex correspondence is an equality of
Laurent series in $Q$, and we will prove it modulo $Q^M$. Since $\fN_{(\lambda,\mu,\nu),m} = \emptyset$ for $m \le n_0$ and
$\fQ_{m} = \emptyset$ for $m<0$, this means that only
the enumerative invariants associated to the finite collection of moduli stacks
\[ \{\fN_{(\lambda,\mu,\nu),m} : n_0 < m < M\}  \mbox{ and }\{\fQ_m : 0\leq m < M - n_0\}\] are involved. Hence:
\begin{itemize}
\item we may choose a single $k \in \bZ$ which satisfies the regularity
  condition of Definition~\ref{def:framing-functor} for each $\fN_{(\lambda,\mu,\nu), m}$ for $n_0<m<M$ ;
\item choosing $\tilde{N}\geq f_{k,1}(1,-\beta_C,-M)$, we set $p_{m}:= \tilde{N} -(M-m)$. This is bounded below by $f_{k,1}(1,\beta_C,-m)$, hence positive for $n_0<m<M$. This choice guarantees that $\tilde N = f_{k,p_{m}}(\alpha_\fX)$ is independent of $n_0<m<M$.
\end{itemize}
For each $n_0<m<M$, using the framing functor $F_{k,p_{m}}$, we
therefore obtain the wall-crossing formula \eqref{eq:joyce-DT-PT-WCF}.
Collecting them into a generating series, we obtain the factorization
\begin{equation}\label{eq:joyce-DT-PT-WCF-series}
  \sum_{n \geq n_0} Q^n \sN_{(\lambda,\mu,\nu),n}(\tau^-) \equiv \bigg(\sum_{n \geq n_0} Q^n \sN_{(\lambda,\mu,\nu),n}(\tau^+)\bigg) \bigg(\sum_{n \geq 0} Q^n \sW_{n,\tilde N}\bigg) \bmod{Q^M}.
\end{equation}

This is almost Theorem~\ref{thm:dt-pt}, since the first two bracketed
terms are $\sV^{\DT,K}_{\lambda,\mu,\nu}$ and
$\sV^{\PT,K}_{\lambda,\mu,\nu}$ respectively, and it remains to
evaluate the third bracketed term. The formula \eqref{eq:joyce-DT-PT-WCF-series} is stated for a
given $(\lambda,\mu,\nu)$, and the constant $\tilde N$ depends a
priori on the curve class $\beta_C = (|\lambda|, |\mu|, |\nu|)$.  We apply the same sort of trick
again to eliminate this dependence. 
For this, we may choose $k$, and $\tilde N$ large enough so that they work for both the original $\beta_C$ and chosen $M$, as well as for
$\beta_C=0$ and $M$ replaced by $M+n_0$. This yields the same formula
\eqref{eq:joyce-DT-PT-WCF-series} when $(\lambda, \mu, \nu) =
(\emptyset, \emptyset, \emptyset)$, i.e.
\begin{equation}\label{eq:joyce-DT0-PT0-WCF-series}
  \sum_{n \in \bZ} Q^n \sN_{(\emptyset,\emptyset,\emptyset),n}(\tau^-) \equiv 1 \cdot \sum_{n \in \bZ} Q^n \sW_{n,\tilde N} \bmod{Q^{M+n_0}}.
\end{equation}
by the triviality of the PT vertex for
$(\emptyset,\emptyset,\emptyset)$. To emphasize, it is the same
$\tilde N$ in \eqref{eq:joyce-DT-PT-WCF-series} which appears in
\eqref{eq:joyce-DT0-PT0-WCF-series}. So, inserting
\eqref{eq:joyce-DT0-PT0-WCF-series} into
\eqref{eq:joyce-DT-PT-WCF-series} gives
\[ \bigg(\sum_{n \in \bZ} Q^n \sN_{(\lambda,\mu,\nu),n}(\tau^-)\bigg) \equiv \bigg(\sum_{n \in \bZ} Q^n \sN_{(\lambda,\mu,\nu),n}(\tau^+)\bigg) \bigg(\sum_{n \in \bZ} Q^n \sN_{(\emptyset,\emptyset,\emptyset),n}(\tau^-)\bigg) \bmod{Q^M}. \]
The dependence on $\tilde N$ has been completely eliminated. Since $M$
was arbitrary, this equality holds up to any finite order. This
concludes the Joyce-style wall-crossing proof of the DT/PT vertex
correspondence (Theorem~\ref{thm:dt-pt}). \qed

\subsubsection{}

For completeness, we record the explicit formula for $\sW_{m,N}$, even
though we did not require it for the DT/PT proof. It should be
compared with \eqref{eq:wall_comb}.

\begin{proposition} \label{prop:joyce-combinatorics-explicit}
  For integers $k \ge 0$ and $N > 0$, and $\vec m = (m_1, \ldots,
  m_k) \in \bZ_{>0}^k$,
  \[ \sum_{\sigma \in S_k} c'_<(m_{\sigma(1)}, \ldots, m_{\sigma(k)}, N-|\vec m|) = \frac{[N]_\kappa!}{[N - |\vec m|]_\kappa! \prod_i [m_i-1]_\kappa!} \cdot \begin{cases} 1 & k \le 1 \\ 0 & \text{otherwise}. \end{cases} \]
\end{proposition}

\begin{proof}
  Exercise for the reader, using the same ideas as in the proof of
  \eqref{eq:wall_comb}.
\end{proof}

Plugging this back into $\sW_{m,N}$ shows that $\sW_{m,N} = \tilde
\sQ_m$, which, in particular, is independent of $N$. This independence
is an alternate way to obtain the factorization
\eqref{eq:joyce-DT-PT-WCF-series}, as an exact equality of series
instead of merely up to some finite order.

\subsubsection{}

For use in \S\ref{sec:invariants-at-the-wall}, we can also apply the
formulas \eqref{eq:joyce-x-WCF-for-N}, \eqref{eq:joyce-x-WCF-for-N-ss}
and \eqref{eq:joyce-x-WCF-for-Q} to the individual DT and PT
wall-crossing formulas \eqref{eq:joyce-s-WCF-for-DT} and
\eqref{eq:joyce-s-WCF-for-PT}. For DT, the result is
\begin{equation} \label{eq:joyce-s-WCF-for-DT-short}
  \sN_{(\lambda,\mu,\nu),n}(\tau^-) = \sum_{m \in \bZ} \tilde \sN_{(\lambda,\mu,\nu),n-m} \cdot \sW^-_{m, N}
\end{equation}
where
\begin{align*}
  \sW^-_{m,N} &\coloneqq \sum_{\substack{k>0, \; \vec m \in \bZ_{>0}^k\\|\vec m| = m}} c_<(\vec m, N-m) \frac{[N-m-1]_\kappa!}{[N]_\kappa!} \prod_{i=1}^k [m_i-1]_\kappa!\cdot \tilde\sQ_{m_i} \\
  c_<(m_1,\ldots,m_\ell) &\coloneqq \sum_{\substack{w \in R(m_1,\ldots,m_\ell)\\o_1(w) < \cdots < o_\ell(w)}} \prod_{i=1}^{\ell-1} \Big[m_i - \sum_{j=i+1}^\ell c_{i,j}(w)\Big]_\kappa.
\end{align*}
For PT, the result is
\begin{equation} \label{eq:joyce-s-WCF-for-PT-short}
  \tilde\sN_{(\lambda,\mu,\nu),n} = \sum_{m \in \bZ} \sN_{(\lambda,\mu,\nu),n-m}(\tau^+) \cdot \sW^+_{m, N}
\end{equation}
where
\begin{align*}
  \sW^+_{m,N} &\coloneqq \sum_{\substack{k>0, \; \vec m \in \bZ_{>0}^k\\|\vec m| = m}} b_<(\vec m, N-m) \frac{[N-m]_\kappa!}{[N-1]_\kappa!} \prod_{i=1}^k [m_i-1]_\kappa!\cdot \tilde\sQ_{m_i} \\
  b_<(m_1,\ldots,m_\ell) &\coloneqq \sum_{\substack{w \in R(m_1,\ldots,m_\ell)\\o_\ell(w) < o_1(w) < \cdots < o_{\ell-1}(w)}} \prod_{i=1}^{\ell-1} \Big[m_i - \sum_{j=i+1}^\ell c_{i,j}(w)\Big]_\kappa.
\end{align*}

\subsubsection{}

\begin{proposition} \label{prop:joyce-PT-combinatorics-explicit}
  For integers $k > 0$ and $N > 0$, and $\vec m = (m_1, \ldots, m_k)
  \in \bZ_{>0}^k$,
  \begin{align*}
    &\sum_{\sigma \in S_k} b_<(m_{\sigma(1)}, \ldots, m_{\sigma(k)}, N-|\vec m|) \\
    &= \frac{[N-1]_\kappa!}{[N - |\vec m| - 1]_\kappa! \prod_i [m_i-1]_\kappa!} \cdot \begin{cases} \left(-\kappa^{\frac{1}{2}}\right)^{m_1} + \left(-\kappa^{\frac{1}{2}}\right)^{-m_1} & k = 1 \\ \left(-\kappa^{\frac{1}{2}}\right)^{m_1-m_2} + \left(-\kappa^{\frac{1}{2}}\right)^{m_2-m_1} & k = 2 \\ 0 & \text{otherwise}. \end{cases}
  \end{align*}
\end{proposition}

\begin{proof}
  Exercise for the reader, using
  Proposition~\ref{prop:joyce-combinatorics-explicit}.
\end{proof}

\subsection{Invariants at the wall}%wall invariants and quot
\label{sec:invariants-at-the-wall}

\subsubsection{}

One advantage of the Joyce-style wall-crossing setup, compared to the
Mochizuki-style setup of \S\ref{sec:mochizuki-WCF}, is that we get for
free some wall-crossing formulas \eqref{eq:joyce-s-WCF-for-DT-short}
and \eqref{eq:joyce-s-WCF-for-PT-short} for the invariants
$\tilde\sQ_m$ and $\tilde\sN_{(\lambda,\mu,\nu),m}$ {\it on} the wall
$\tau_0$. These wall-crossing formulas are interesting in their own
right and have not appeared in the literature, to the best of our
knowledge. We can explicitly identify $\tilde\sQ_m$ and
$\tilde\sN_{(\lambda,\mu,\nu),m}$ as enumerative invariants of
$\Hilb(\bC^3)$ (Proposition~\ref{prop:on-wall-invariants-Q-joyce}) and
of $\Quot(\cO_{\bC^3}^{\oplus 2})$
(Proposition~\ref{prop:on-wall-invariants-N}) respectively.

\subsubsection{}

\begin{proposition} \label{prop:on-wall-invariants-Q-joyce}
  Let $1 \le a \le N$. In the notation of
  Lemma~\ref{lem:off-wall-flag-invariants}, for $\bar x(a) \ge x_0(a)$
  there is an isomorphism
  \[ \fQ^{Q(N),\sst}_{m,\vec 1_{[a,N]}}(\tau_0^\pm(0, \bar x(a))) \cong \fQ^{\st}_{m, \vec 1_{[a,N]}}. \]
\end{proposition}

\begin{proof}
  This is basically an analogue of Lemma~\ref{lem:off-wall-invariants}
  for $\fQ^{Q(N)}$ instead of $\fN^{Q(N)}$. Let $(\cE[-1], \vec V,
  \vec \rho)$ be an object on the left hand side. By
  Lemma~\ref{lem:semistable-flag-framing-must-be-injective}, the
  framing is a chain of isomorphisms followed by an inclusion, like
  \eqref{eq:off-wall-rank-1-framing}. Since $\cE$ is a
  zero-dimensional sheaf, $\cE(k) \cong \cE$ and $V_{N+1} = H^0(\cE)$.
  Hence the data $(\cE[-1], \vec V, \vec \rho)$ is equivalent to the
  data $(\cE, s)$ where $s\colon \cO_{\bC^3} \to \cE$ is a section. If
  $s$ were not surjective, it factors through a non-trivial sub-sheaf
  $\cE' \subset \cE$ of length $m' < m$. Then
  \[ \tau_0^\pm(0, \bar x(a))((\cE'[-1], \vec V, \vec \rho)) = \frac{(\vec \mu + \bar x(a)) \cdot \vec 1_{[a,N]}}{m'} > 0 = \tau_0^\pm(0, \bar x(a))((\cE/\cE'[-1], \vec 0, \vec 0)) \]
  and therefore the sub-object $(\cE'[-1], \vec V, \vec \rho)$ is
  destabilizing. Furthermore, this is the only way in which $(\cE, s)$
  can have a destabilizing sub-object. This gives the desired
  isomorphism.
\end{proof}

\subsubsection{}

Combining the isomorphisms of
Proposition~\ref{prop:on-wall-invariants-Q-joyce} and
Proposition~\ref{prop:on-wall-invariants-Q} and taking enumerative
invariants, we get the equality $\tilde \sQ_m =
\sN_{(\emptyset,\emptyset,\emptyset),m}(\tau^-)$. Applying this to the
left hand side of the $(\emptyset,\emptyset,\emptyset)$ case of DT/PT
wall-crossing formula \eqref{eq:joyce-DT-PT-WCF} yields
\[ \tilde \sQ_n = \sW_{n, N}. \]
Matching coefficients of $\tilde \sQ_m$ on both sides would reproduce
the combinatorial statement of
Proposition~\ref{prop:joyce-combinatorics-explicit}, but note that
this argument alone is {\it not sufficient} to provide an alternate
proof of Proposition~\ref{prop:joyce-combinatorics-explicit} since the
$\tilde \sQ_m$ are not algebraically independent.

\subsubsection{}

\begin{proposition} \label{prop:on-wall-invariants-N}
  Let $1 \le a \le N$. In the notation of
  Lemma~\ref{lem:off-wall-flag-invariants}, for $\bar x(a) \ge x_0(a)$
  there is an isomorphism
  \[ \fN^{Q(N),\sst}_{(\emptyset,\emptyset,\emptyset),m,\vec 1_{[a,N]}}(\tau_0^\pm(0, \bar x(a))) \cong \Quot(\cO_{\bC^3}^{\oplus 2}, m) \]
  which identifies the (symmetrized) virtual structure sheaves.
\end{proposition}

We view $\Quot(\cO_{\bC^3}, m)$ as the moduli scheme of pairs
$[\cO_{\bar X}^{\oplus 2} \twoheadrightarrow \cE]$ where $\cE$ is a
zero-dimensional of length $m$ and support on $\bC^3 = \bar X
\setminus D$, and the map is a surjection. It carries a natural
symmetric perfect obstruction theory coming from its presentation as a
critical locus in an associated non-commutative Quot scheme. The
generating series
\begin{equation*}
  \sum_{m\in\bZ} Q^m \tilde\sN_{(\emptyset,\emptyset,\emptyset),m, \vec 1_{[a,N]}}(0,\bar x(a)) = \sum_{m \in \bZ} Q^m \chi(\Quot(\cO_{\bC^3}^{\oplus 2}, m), \hat\cO^\vir) \eqqcolon \sV^{\DT(2),K}_{\emptyset,\emptyset,\emptyset}
\end{equation*}
is therefore the K-theoretic rank-$2$ degree-$0$ DT vertex computed in
\cite[Theorem A]{FaMoRi21}.

\subsubsection{}%moduli identification

\begin{proof}
  This will be a longer version of the proof of
  Lemma~\ref{prop:on-wall-invariants-Q-joyce}. By the same argument as
  in that proof, the relevant objects on the left hand side are pairs
  \begin{equation} \label{eq:on-wall-quot-objects}
    ([L \otimes \cO_{\bar X} \xrightarrow{s} \cE], V \xrightarrow{\rho} L \oplus H^0(\cE))
  \end{equation}
  where $L$ and $V$ are $1$-dimensional vector spaces and $\cE$ is a
  zero-dimensional sheaf on $\bC^3$. Let $\rho_1$ and $\rho_2$ be the
  $L$ and $H^0(\cE)$ components of $s$ respectively. The component
  $\rho_1$ cannot be zero, because otherwise there would be the
  destabilizing sub-object $([0 \to \cE], \, V \xrightarrow{\rho_1}
  H^0(E))$. Hence $\rho_1$ is an isomorphism and we only need to study
  $s$ and $\rho_2$.

  Note that $s\colon L \otimes \cO \to \cE$ is equivalent to a
  morphism $s\colon L\to H^0(\cE)$. Abusing notation, we write $s$ for
  both. Similarly, write $\rho_2\colon V \otimes \cO \to \cE$. Now
  assume for the sake of contradiction that
  \[ s \times \rho_2\colon (L\oplus V) \otimes \cO \to E \]
  is not surjective. Take $\cE'$ to be its image. Then the pair
  $([L\otimes\cO \to \cE'], \, V \to L \oplus H^0(\cE'))$ is a
  sub-object because both $s$ and $\rho_2$ factor through $\cE'$. By
  the same calculation as in the proof of
  Lemma~\ref{prop:on-wall-invariants-Q-joyce}, this sub-object is
  destabilizing. Furthermore, this is the only way in which
  \eqref{eq:on-wall-quot-objects} can have a destabilizing sub-object,
  given that $\rho_1$ is an isomorphism. This gives the desired
  isomorphism.

  \subsubsection{}%obstruction theory

  To match virtual cycles, recall that the virtual tangent space of
  $\Quot(\cO_{\bC^3}^{\oplus 2})$ used in \cite{FaMoRi21} is
  \begin{align*}
    &\Ext_{\bar X}(W \otimes \cO, W \otimes \cO) - \Ext_{\bar X}([W \otimes \cO \to \cE], [W \otimes \cO \to \cE]) \\
    &= \Ext_{\bar X}(W \otimes \cO, \cE) + \Ext_{\bar X}(\cE, W \otimes \cO) - \Ext_{\bar X}(\cE, \cE)
  \end{align*}
  where $W \coloneqq V \oplus L \cong \bC^2$ and $\Ext_{\bar X}(-, -)
  \coloneqq \sum_i (-1)^i \Ext^i_{\bar X}(-, -)$. On the other hand,
  abbreviating $I \coloneqq [L \otimes \cO \to \cE]$, the APOT on
  stable loci of $\fN^{Q(N),\pl}_{m,\vec 1}$ comes from symmetrized
  pullback, with virtual tangent space
  \begin{equation} \label{eq:quot-obstruction-theory-symmetrized-pullback}
    -\Ext_{\bar X}(I, I(-D)) + \left(\Hom(V \otimes \cO, L \otimes \cO \oplus \cE) - \bC\right) - \left(\Hom(V \otimes \cO, L \otimes \cO \oplus \cE) - \bC \right)^\vee \kappa^{-1}
  \end{equation}
  coming from \eqref{eq:quiver-bilinear-obstruction} (see also
  Remark~\ref{rem:quiver-framing-isomorphisms}). From
  \eqref{eq:obstruction-theory-rank-1-LES},
  \begin{align*}
    \Ext_{\bar X}(I, I)
    &= \Ext_{\bar X}(I, I(-D)) + \Ext_{\bar X}(I\big|_D, I\big|_D) \\
    &= \Ext_{\bar X}(I, I(-D)) + \Ext_{\bar X}(L \otimes \cO_D, L \otimes \cO_D).
  \end{align*}
  Evidently this last term is equal to $\Ext_{\bar X}(L \otimes \cO, L
  \otimes \cO)$. The $\bC = \Hom(V \otimes \cO, V \otimes \cO)$ may be
  identified with $\Hom(V \otimes \cO, L \otimes \cO)$ using the
  isomorphism $\rho_1$. Putting it all together,
  \eqref{eq:quot-obstruction-theory-symmetrized-pullback} becomes
  \[ \left(\Ext_{\bar X}(L \otimes \cO, \cE) + \Ext_{\bar X}(\cE, L \otimes \cO) - \Ext_{\bar X}(\cE, \cE)\right) + \Ext_{\bar X}(V \otimes \cO, \cE) + \Ext_{\bar X}(\cE, V \otimes \cO). \]
  So the K-theory classes of the two virtual tangent spaces match,
  proving \eqref{eq:APOT-k-class-matching}, so that we can use
  Proposition \ref{prop:master_space_vir_class_comparison} to identify
  $\hat\cO^\vir$.
\end{proof}

\subsubsection{}
\label{sec:rank-2-DT0-vertex}

Applying the identifications of
Proposition~\ref{prop:on-wall-invariants-N} and
Proposition~\ref{prop:on-wall-invariants-Q} to the wall-crossing
formula \eqref{eq:joyce-s-WCF-for-PT-short} when $(\lambda, \mu, \nu)
= (\emptyset, \emptyset, \emptyset)$, the combinatorial result of
Proposition~\ref{prop:joyce-PT-combinatorics-explicit} yields the
identity
\begin{equation} \label{eq:rank-2-DT0-vertex}
  \begin{aligned}
    \sV^{\DT(2),K}_{\emptyset,\emptyset,\emptyset}
    &= \bigg(\sum_{n \in \bZ} Q^n (-\kappa^{\frac{1}{2}})^n \tilde \sQ_n\bigg)\bigg(\sum_{n \in \bZ} Q^n (-\kappa^{\frac{1}{2}})^{-n} \tilde \sQ_n\bigg) \\
    &= \sV^{\DT,K}_{\emptyset,\emptyset,\emptyset}\Big|_{Q \mapsto -Q\kappa^{\frac{1}{2}}} \cdot \sV^{\DT,K}_{\emptyset,\emptyset,\emptyset}\Big|_{Q \mapsto -Q\kappa^{-\frac{1}{2}}}.
  \end{aligned}
\end{equation}
This is exactly the rank-$2$ case of \cite[Theorem A]{FaMoRi21},
though of course their proof via a rigidity argument is very
different, much simpler, and more conceptual than the path we took.
Conversely, matching coefficients of $\tilde \sQ_{m_1} \tilde
\sQ_{m_2}$ on both sides would reproduce the combinatorial statement
of Proposition~\ref{prop:joyce-PT-combinatorics-explicit}, but again
note that this argument alone is {\it not sufficient} to provide an
alternate proof of
Proposition~\ref{prop:joyce-PT-combinatorics-explicit} since the
$\tilde \sQ_m$ are not algebraically independent.

\phantomsection
\addcontentsline{toc}{section}{References}

\begin{small}
\bibliographystyle{alpha}
\bibliography{paper}
\end{small}

\end{document}